\documentclass[11pt]{amsart}
\usepackage{amsmath,amssymb,amsthm}
\usepackage[a4paper,margin=25mm,includeheadfoot]{geometry}
\newtheorem{theorem}{Theorem}[section]
\newtheorem{proposition}[theorem]{Proposition}
\newtheorem{lemma}[theorem]{Lemma}
\newtheorem{corollary}[theorem]{Corollary}
\theoremstyle{remark}
\newtheorem{remark}[theorem]{Remark}
\newcommand{\C}{\mathbb C}

\newcommand{\ord}{\operatorname{ord}}
\newcommand{\A}{\mathcal A}
\newcommand{\D}{\mathcal D}
\newcommand{\Needspace}[1]{\par\ifdim\dimexpr\pagegoal-\pagetotal\relax<#1\newpage\fi}
\title[Six-body central configurations, Part I]{Generic finiteness of planar six-body central configurations: Part I, exceptional-mass first limits}

\author{Zhe Chen}
\address{School of Wind Energy, Technical University of Denmark, 2800 Kongens Lyngby, Lyngby-Taarbæk Municipality, Copenhagen, Denmark}
\email{hitchenzhe@gmail.com}

\author{Zhicheng Tong}
\address{School of Mathematics, Jilin University, Changchun 130012, P. R. China}
\email{tongzc25@jlu.edu.cn}
\begin{document}
\begin{abstract}
We separate the generic positive-mass finiteness problem for planar six-body central configurations from finiteness on exceptional mass loci. The Chang--Chen diagram enumeration has 117 graphs. The arguments below confine 40 first-limit diagram obligations to proper necessary mass loci. In particular, a direct function-field argument establishes the No.~66 mass condition for both zero and nonzero drift. We establish the algebraic reduction from an infinite positive family to a germ with an unbounded seven-product, and use the stated diagram rules and computational certificates to confine its mass vector to the union of these loci. Finiteness on the exceptional positive mass loci, the proposed subject of Part II, will be discussed in another paper.
\end{abstract}
\maketitle
\tableofcontents

\section{Introduction}\label{sec:intro}
Let $m_1,\ldots,m_6>0$ be masses at distinct positions $q_1,\ldots,q_6\in\mathbb R^2$, and let $c=(\sum_i m_iq_i)/(\sum_i m_i)$ be their centre of mass. The configuration is \emph{central} if some $\lambda\in\mathbb R$ satisfies
\begin{equation}\label{eq:physical}
\lambda(q_i-c)=\sum_{j\ne i}m_j\frac{q_j-q_i}{|q_j-q_i|^3},\qquad i=1,\ldots,6.
\end{equation}
Translation, rotation and positive scaling preserve this condition. Smale's sixth problem asks whether, for each fixed positive mass vector, there are only finitely many such configurations modulo these symmetries.

For five positive masses, finiteness is known~\cite{CTFive,FiveT}; the corresponding question for six positive masses remains open. Following the normalized algebraic formulation, set $z_i=x_i+iy_i$, $w_i=x_i-iy_i$ on real configurations and regard $z_i,w_i$ as independent complex coordinates. Put $z_{ij}=z_i-z_j$, $w_{ij}=w_i-w_j$, and use signed inverse distances $\delta_{ij}=\delta_{ji}$. With the rotational gauge $z_{12}=w_{12}$, the normalized system is
\begin{equation}\label{eq:AK}
\begin{aligned}
z_i&=\sum_{j\ne i}m_j\delta_{ij}^3(z_i-z_j),&
 w_i&=\sum_{j\ne i}m_j\delta_{ij}^3(w_i-w_j),\\
\delta_{ij}^2z_{ij}w_{ij}&=1 &&(1\le i<j\le6),\qquad z_{12}=w_{12}.
\end{aligned}
\end{equation}
For fixed $m=(m_1,\ldots,m_6)\in\mathbb R_{>0}^6$, let $\A_m$ denote its complex zero set. A point of $\A_m$ is \emph{positive} when $w_i=\overline{z_i}$ and every $\delta_{ij}>0$. The finiteness of these normalized positive points is equivalent to finiteness modulo the symmetries of~\eqref{eq:physical}.

Put $r_{ij}^2=z_{ij}w_{ij}$ and, for each pair $ij$, define
\begin{equation}\label{eq:sevenproducts}
P_{ij}=r_{ij}^2\prod_{\substack{k<\ell\\\{k,\ell\}\cap\{i,j\}=\varnothing}}r_{k\ell}^2\qquad(1\le i<j\le6).
\end{equation}
These are the 15 \emph{seven-products} of squared distances. A \emph{positive-point first limit} is a Puiseux branch at infinity of an irreducible curve in $\A_m$ containing a nonconstant arc of positive points, along which some $P_{ij}$ has a pole. This is a condition on a curve that passes through physical configurations, not an assertion that the singular limit itself is positive real.

The normalized-curve approach for five bodies~\cite{CTFive} motivates our reduction to such first limits. Chang and Chen's Part~I~\cite{CCI} enumerates 117 six-body diagrams and eliminates 31 by Algorithm~II. Their Part~II~\cite[Theorem~6.1, Sections~6.2--6.12 and Table~6.1]{CCII} excludes No.~116 for positive masses and confines 61 further diagrams to proper algebraic mass loci, leaving 24 diagrams historically open. This use of cited results does not independently reproduce their symbolic computations. By the local arguments below, 28 of the 61 reported-relation diagrams and 17 historically open diagrams are excluded at the positive-point first-limit tier, leaving 33 published and 7 derived relation obligations. Let $\D_{\mathrm{CC}}$ denote those 33 surviving reported labels and put $\D_{\mathrm{new}}=\{8,63,66,88,93,94,98\}$ and $\D_{40}=\D_{\mathrm{CC}}\cup\D_{\mathrm{new}}$.

On curves through positive configurations, all 28 additional reported-relation diagrams are excluded as first limits by the local arguments below. No.~4 follows from Lemma~\ref{lem:reported-no4}, and No.~20 from Section~\ref{sec:no20-energy}. No.~2 and 22 follow from Lemma~\ref{lem:reported-force-transfer}; No.~3 from Lemma~\ref{lem:reported-no3}, and No.~78 from Lemma~\ref{lem:reported-no78}. No.~40, 44 and 45 follow from Sections~\ref{sec:no40-pole}, \ref{sec:no44-force} and~\ref{sec:no45-pole}; No.~86 and 95 from Sections~\ref{sec:no86-wing} and~\ref{sec:no95-wing}; No.~99 and 100 from Sections~\ref{sec:no99-cycle} and~\ref{sec:no100-cycle}; No.~109/112 from Sections~\ref{sec:no109-poles} and~\ref{sec:no112-poles}, and No.~111 from Section~\ref{sec:no111-wing}; No.~51/55 from Sections~\ref{sec:no51-pendant} and~\ref{sec:no55-pendant}; No.~105/77 from Sections~\ref{sec:no105-sign} and~\ref{sec:no77-matching}; No.~31/32/36 from Sections~\ref{sec:no31-bridge}, \ref{sec:no32-tie} and~\ref{sec:no36-matches}; No.~49/50 from Sections~\ref{sec:no49-path} and~\ref{sec:no50-resonance}; No.~57/61/84 from Sections~\ref{sec:no57-triangle}, \ref{sec:no61-sign} and~\ref{sec:no84-cycle}.

The five historical connectivity cases follow from Lemma~\ref{lem:excluded-connectivity}; No.~6 follows from Section~\ref{sec:no6-inertia} and No.~71 from Section~\ref{sec:no71-satellite}. No.~74 and 75 follow from Lemma~\ref{lem:excluded-satellite}; No.~82, 89 and 90 from Sections~\ref{sec:no82-torque}, \ref{sec:no89-torque} and~\ref{sec:no90-torque}; No.~81 from Section~\ref{sec:no81-energy}; No.~64 and 65 from Sections~\ref{sec:no64-energy} and~\ref{sec:no65-pair}. No.~79 follows from Section~\ref{sec:no79-sector}. No.~52 is excluded by Proposition~\ref{prop:no52-excluded}.

Using the published relation loci for $r\in\D_{\mathrm{CC}}$ and Propositions~\ref{prop:no8-locus}, \ref{prop:no63-locus}, \ref{prop:no66-locus}, \ref{prop:no88-locus}, \ref{prop:no93-locus}, \ref{prop:no94-locus} and~\ref{prop:no98-locus} for the seven derived classes, choose an orbit-completed nonzero necessary polynomial $P_r(m)$ for each $r\in\D_{40}$. Every first limit of type $r$ forces $P_r(m)=0$; the converse is not claimed. Define
\begin{equation}\label{eq:E40}
E_{40}=\left\{m\in\mathbb R_{>0}^6:\prod_{r\in\D_{40}}P_r(m)=0\right\}.
\end{equation}
The product is nonzero, so $E_{40}$ is a proper algebraic subset of the positive mass cone. The number 40 counts diagram obligations, not distinct irreducible varieties.

\begin{theorem}[First-limit confinement]\label{thm:obstruction}
Using Chang--Chen's cited classification and the specified computational certificates, if the planar six-body problem has infinitely many positive normalized central configurations at $m\in\mathbb R_{>0}^6$, then $m\in E_{40}$.
\end{theorem}
\begin{corollary}[Generic finiteness off $E_{40}$]\label{cor:generic}
Positive normalized central configurations are finite for every $m\in\mathbb R_{>0}^6\setminus E_{40}$.
\end{corollary}
This deduction uses the published Chang--Chen theorem; its computations and the local certificates still require independent publication-level audit. All-positive-mass finiteness, including $E_{40}$, is a distinct and unfinished Part II problem.

\section{The normalized six-body system}\label{sec:system}
In the normalized system~\eqref{eq:AK}, no $\delta_{ij}$ or $z_{ij}w_{ij}$ vanishes on $\A_m$. Summing mass-weighted coordinate rows gives $\sum_i m_iz_i=\sum_i m_iw_i=0$. The finiteness question concerns the positive points of $\A_m$, not the entire complex zero set.

Relabelling the bodies or exchanging $z$ and $w$ preserves the ungauged equations and all squared-distance and seven-product orders. Since $z_{ij}w_{ij}\ne0$, after a relabelling $\pi$ choose $\lambda^2=w_{\pi(1)\pi(2)}/z_{\pi(1)\pi(2)}$; the complex rotation $(z,w)\mapsto(\lambda z,\lambda^{-1}w)$ restores the gauge $z_{12}=w_{12}$. Diagram conclusions below therefore apply to the full labelled and colour-swapped orbit, with the masses permuted along with the body labels.

When a necessary mass factor has been obtained for only a canonical labelling, we replace it by the finite product of its pullbacks under all body permutations. We likewise include the factor obtained by the $z\leftrightarrow w$ involution if its canonical proof is colour-specific. Each factor is nonzero as a polynomial, so the orbit product remains nonzero and vanishes on every labelled occurrence of the diagram. In~\eqref{eq:E40}, $P_r$ denotes this orbit-completed necessary polynomial, not just the expression in one choice of body labels.

Set $U=\sum_{i<j}m_im_j\delta_{ij}$ and $I=\sum_i m_iz_iw_i$.
\begin{lemma}[Constant positive energy on a positive-point curve]\label{lem:energy}
If an irreducible curve $\Gamma\subset\A_m$ contains a positive real-analytic arc, then $U=I=c>0$ on $\Gamma$ for a constant $c$.
\end{lemma}
\begin{proof}
On the smooth distance-tie variety, differentiate $2U+I$ along coordinate directions. The derivatives are the mass-weighted coordinate residuals in~\eqref{eq:AK}, so they vanish along an algebraic curve of solutions. The function is constant on that irreducible curve. Simultaneously scaling $(z,w,\delta)$ by $(s,s,s^{-1})$ and differentiating at $s=1$ gives $U=I$ at each solution. On the positive arc, $U>0$, fixing the constant sign. The gauge is not used in this differential argument.
\end{proof}

\begin{lemma}[Positive arc and bounded distance fibres]\label{lem:arc}
Infinitely many positive points of $\A_m$ yield an irreducible algebraic curve containing a nonconstant positive real-analytic arc. Fixing all 15 nonzero squared distances $z_{ij}w_{ij}$ on $\A_m$ leaves only finitely many normalized configurations.
\end{lemma}
\begin{proof}
The positive locus is semialgebraic. If it is infinite, a positive-dimensional semialgebraic component admits a Nash arc; its complex Zariski closure contains an irreducible curve and retains that arc. For fixed nonzero squared distances the $\delta_{ij}$ take finitely many signs. Fix one choice. The pairwise Gram data of the six vectors $(z_i-z_1,w_i-w_1)$ are fixed. Because $z_{12}w_{12}\ne0$, their Gram form is nondegenerate for the bilinear form $(a,b)\cdot(c,d)=(ad+bc)/2$. Two realizations with the same Gram data differ by an isometry of this complex plane; the mass centre fixes translation and $z_{12}=w_{12}$ leaves only finitely many scaling/reflection choices. This is the six-body version of the fixed-distance argument in the five-body proof.
\end{proof}

The 15 seven-products $P_{ij}$ were defined in~\eqref{eq:sevenproducts}.
\begin{lemma}[The six-body complement-product matrix]\label{lem:products}
On every irreducible positive-dimensional curve in $\A_m$ at least one $P_{ij}$ is nonconstant; on some branch at infinity it has a pole.
\end{lemma}
\begin{proof}
The exponent matrix of these monomials in the 15 squared distances is $E=I_{15}+A(K(6,2))$, where two pair-indices in the Kneser graph are adjacent exactly when disjoint. Its eigenvalues are $7$ once, $-2$ five times and $2$ nine times; thus $\det E=-114688\ne0$. If all $P_{ij}$ were constant, powers of all squared distances would be constant on the curve. Lemma~\ref{lem:arc} would make that curve finite, a contradiction. A nonconstant regular function on an affine irreducible algebraic curve attains arbitrarily large complex values. At a point at infinity of its normalization it has a pole, giving the required Puiseux branch. This argument does not assume that the branch itself is positive real.
\end{proof}

\section{Diagram inputs and exceptional-mass accounting}\label{sec:input}
The cited Chang--Chen classification and the local proofs separate excluded positive-point first limits from necessary mass relations. The disjoint diagram accounting is as follows.
\begin{table}[ht]
\centering\small
\begin{tabular}{p{5.3cm}r p{7.2cm}}
\hline
Chang--Chen category & Number & Positive-tier disposition \\
\hline
Algorithm II eliminated & 31 & no target diagram \\
Published positive-mass exclusion & 1 & No.~116 excluded in Chang--Chen Part~II \\
Published relation diagrams & 61 & 28 additionally excluded; 33 relation obligations \\
Historically open & 24 & 17 excluded first limits; 7 derived relations \\
\hline
Total & 117 & 40 relation obligations; no unrestricted positive-tier diagram \\
\hline
\end{tabular}
\caption{Bookkeeping of diagrams, not of distinct mass varieties.}\label{tab:status}
\end{table}
The reported-relation and historically open cases are treated separately below.
\section{Previously classified relation diagrams}\label{sec:excluded}
\begin{remark}[Disjoint membership of the excluded classes]
The 31 Algorithm II labels are

\[
\begin{gathered}
1,10,21,23,24,25,26,33,34,35,37,\\
38,39,41,42,53,54,56,59,60,68,\\
70,72,83,85,87,91,92,96,97,113.
\end{gathered}
\]

The 28 reported-relation diagrams whose positive

first limits are additionally excluded are

\[
\begin{gathered}
2,3,4,20,22,31,32,36,40,44,\\
45,49,50,51,55,57,61,77,78,84,\\
86,95,99,100,105,109,111,112.
\end{gathered}
\]

The 17 historically open diagrams whose positive

first limits are excluded are

\[
\begin{gathered}
5,6,52,64,65,67,69,71,73,\\
74,75,79,80,81,82,89,90.
\end{gathered}
\]

Together with positive-empty No.~116, the 33 published
relation labels below and the 7 derived labels, these are
disjoint and cover exactly $1,\ldots,117$.

The 31 Algorithm~II exclusions come from the cited
Chang--Chen theorem. Among the additional 28
reported-relation positive first-limit exclusions,
No.~2 and 22 follow from Lemma~\ref{lem:reported-force-transfer};
No.~3 and 78 follow from Lemmas~\ref{lem:reported-no3}
and~\ref{lem:reported-no78}; No.~40, 44 and 45 follow
from Sections~\ref{sec:no40-pole}, \ref{sec:no44-force}
and~\ref{sec:no45-pole}; No.~86 and 95 follow from
Sections~\ref{sec:no86-wing} and~\ref{sec:no95-wing};
No.~99 and 100 follow from Sections~\ref{sec:no99-cycle}
and~\ref{sec:no100-cycle}; No.~109 and 112 follow from
Sections~\ref{sec:no109-poles} and~\ref{sec:no112-poles},
and No.~111 from Section~\ref{sec:no111-wing}.

\Needspace{10\baselineskip}
No.~51 and 55 follow from Sections~\ref{sec:no51-pendant}
and~\ref{sec:no55-pendant}; No.~105 and 77 follow from
Sections~\ref{sec:no105-sign} and~\ref{sec:no77-matching};
No.~31, 32 and 36 follow from Sections~\ref{sec:no31-bridge},
\ref{sec:no32-tie} and~\ref{sec:no36-matches}; No.~49 and 50
follow from Sections~\ref{sec:no49-path} and
\ref{sec:no50-resonance}; No.~57, 61 and 84 follow from
Sections~\ref{sec:no57-triangle}, \ref{sec:no61-sign} and
\ref{sec:no84-cycle}. No.~4 and 20 follow from
Lemma~\ref{lem:reported-no4} and Section~\ref{sec:no20-energy}.

\Needspace{7\baselineskip}
Among the historical diagrams, No.~5, 67, 69, 73 and 80
follow from Lemma~\ref{lem:excluded-connectivity}; No.~6
from Section~\ref{sec:no6-inertia}; No.~71 from
Section~\ref{sec:no71-satellite}; No.~74 and 75 from
Lemma~\ref{lem:excluded-satellite}; No.~82, 89 and 90
from their torque subsections; No.~81 from
Section~\ref{sec:no81-energy}; No.~64 and 65 from
Sections~\ref{sec:no64-energy} and~\ref{sec:no65-pair};
No.~79 from Section~\ref{sec:no79-sector}; No.~52 is
treated separately below. Chang--Chen~\cite[Table~6.1]{CCII}
already excludes No.~116 for positive masses. This
membership list records coverage; it is not a proof of any
individual exclusion.

\end{remark}

\subsection*{shared first-limit tests}

The reduction from the 61 published relation diagrams to
33 still-relevant ones uses 28 additional positive-point
first-limit exclusions. For each of these 28 graphs,
the separate local argument in the subsections below
prevents it from being a \emph{first} limit on a

curve through a positive configuration. The detailed
original-row and branch-coverage arguments are scoped
to such first limits. In particular No.~4 follows from
Lemma~\ref{lem:reported-no4} and No.~20 from
section~\ref{sec:no20-energy}; neither is assumed.
\begin{lemma}[Force transfer at an uncircled $zw$ triangle]\label{lem:reported-force-transfer}
At positive masses, the signed diagrams No.~2 and No.~22 cannot occur as singular germs.
\end{lemma}
\begin{proof}
Both contain an uncircled $zw$ triangle $T=\{a,b,c\}$. The vertex $a$ has no external stroke, while $b$ has exactly one external stroke $bd$ in the $z$ colour and none in the $w$ colour (interchange the colours for No.~22). Since $a$ is uncircled, its two leading coordinate rows require $F^z_{T,a}=F^w_{T,a}=0$. Both separations $ab,ac$ have the same signed distance order. Solving these two nontrivial force equations gives $z_{ab}/w_{ab}=z_{ac}/w_{ac}=\rho$ at leading order. The triangle difference identity transfers the same ratio to $bc$, hence $F^z_{T,b}=\rho F^w_{T,b}$. With no external $w$ stroke at $b$, its original $w$ row requires $F^w_{T,b}=0$. Its original $z$ row then requires $F^z_{T,b}+m_d\delta_{bd}^3(z_d-z_b)^*=0$. The first term is zero and the second is nonzero by the external stroke and $m_d>0$, a contradiction.

For No.~2 the labelled triple and external edge are $(a,b,c;bd)=(1,3,4;35)$; for No.~22, after the colour exchange, they are $(4,6,5;16)$. This uses their circle patterns; the same stroke skeleton with additional circles need not obey the lemma.
\end{proof}
\subsection{No.~2: force transfer across triangle 134}
The vertices $(a,b,c;bd)=(1,3,4;35)$ satisfy the hypotheses of Lemma~\ref{lem:reported-force-transfer}: the triangle is uncircled, body 1 has no external stroke, and body 3 has exactly one external $z$ stroke. Its $w$ row makes the internal $w$-force zero; the common shape ratio then makes the internal $z$-force zero as well, contrary to the nonzero leading 35 force.
\subsection{No.~22: the colour-dual force transfer}
Interchange the two colours in Lemma~\ref{lem:reported-force-transfer} and take $(a,b,c;bd)=(4,6,5;16)$. Bodies 4, 5 and 6 are uncircled, and body 6 has only the 16 exterior $w$ stroke. Its internal $z$-force vanishes, so the transferred internal $w$-force vanishes, contradicting the nonzero signed 16 source in the original $w$ row.

\subsection{No.~3: bounded seven-products}
\begin{lemma}[A bounded reported-relation diagram]
\label{lem:reported-no3}

No.~3 cannot be an unbounded-seven-product first limit.
\end{lemma}

\begin{proof}For a pair $ij$ put $L_{ij}=2$ on a $zw$ edge, $0$ on a single-colour stroke, $-1$ on an unmarked pair joined by a stroke path in at least one colour, and $-4$ on an unmarked pair joined by no such path in either colour. Stroke-path separation bounds give
\[
\ord(z_{ij}w_{ij})\ge L_{ij}.
\]
The seven-product $P_{ij}$ uses $ij$ and its six disjoint pairs, so
\begin{equation}\label{eq:no3-product-bound}
\ord P_{ij}\ge L_{ij}+\sum_{\{k,\ell\}\cap\{i,j\}=\varnothing}L_{k\ell}.
\end{equation}
For the pinned No.~3 graph these fifteen lower bounds are $12$ for $12$, $3$ for the eight pairs $1j,2j$ with $j\in\{3,4,5,6\}$, $2$ for $34,35,45$, and $0$ for $36,46,56$. The minimum is $0$, so no seven-product can have a pole. This does not assert that No.~3 cannot occur elsewhere on a complex curve.
\end{proof}
\subsection{No.~78: a single-bridge energy contradiction}
\begin{lemma}[The single bridge of No.~78]\label{lem:reported-no78}
No.~78 is not a first limit on a curve through a positive configuration, at any strictly positive mass.
\end{lemma}
\begin{proof}
Partition the bodies into its two circled $zw$ triangles $A=123$ and $B=456$, with cluster masses $M_A,M_B$, $M=M_A+M_B$, and put $\mu=m_3m_6$. The original mass-weighted $w$ rows give distinct leading cluster centres and leave $36$ as the sole finite cross inverse distance $D_{36}\to d\ne0$. Their relative equation implies $\mu d^3=M_AM_B/M$. Contract the original $z$ rows against the relative $w$ coordinates of each triangle. Internal inertia tends to zero; for every crossed tie $\delta_{ij}^3z_{ij}=\delta_{ij}/w_{ij}$, so the remaining terms fade. Both internal signed potentials tend to zero, and constancy of $U=I=c>0$ therefore gives $c=\mu d$ and
\begin{equation}\label{eq:no78-bridge-energy}
Mc^3=M_AM_B\mu^2.
\end{equation}
At the positive point on that same curve, H\"older requires $\sigma^3\le Mc^3$, where $\sigma=\sum_{i<j}m_im_j$. But $\sigma>M_AM_B>\mu$ for strictly positive masses; hence $\sigma^3>M_AM_B\mu^2=Mc^3$, a contradiction. This excludes the physical first limit, not an arbitrary complex-only No.~78 curve.
\end{proof}

\subsection{No.~4: vanishing cluster potentials}
\begin{lemma}[The two vanishing clusters of No.~4]\label{lem:reported-no4}
No.~4 is not a positive-point first limit, including on its reported mass-relation zero sets.
\end{lemma}
\begin{proof}
Its negative-seven-product window has $-2<\alpha<-13/8$, where $\alpha$ is the common $z$-separation order on the $z$-only $K_4$ with vertex set $Q=\{3,4,5,6\}$. Put $T=\{1,2,6\}$ and $M=\sum_{i=1}^6m_i$. The $zw$ edges of $T$ have $z$- and $w$-orders both equal to $1$. On each edge of $Q$, the signed tie and the marked $z$-force order $-2$ give
\[
 \ord z_{ij}=\alpha,\qquad
 \ord w_{ij}=\beta:=\frac{4-\alpha}{3}>1.
\]
For $i\in\{1,2\}$ and $j\in\{3,4,5\}$, the separation identities through body $6$ therefore give
\[
 \ord z_{ij}=\alpha,\qquad \ord w_{ij}=1,\qquad
 \ord\delta_{ij}=-\frac{\alpha+1}{2}>0.
\]
The original mass centres yield
\[
 Mz_6=-\sum_{i\ne6}m_i z_{i6},\qquad
 Mw_6=-\sum_{i\ne6}m_i w_{i6},
\]
so $z_6=O(\varepsilon^\alpha)$ and $w_6=O(\varepsilon)$.

We contract the original $w$ rows at bodies $1,2$ against $m_i z_{i6}$, retaining the exterior forces. The internal pairs combine by their exact signed ties, giving
\[
\begin{aligned}
 U_T={}&w_6\sum_{i=1,2}m_i z_{i6}
       +\sum_{i=1,2}m_i z_{i6}w_{i6}\\
 &-\sum_{i=1,2}\sum_{j=3,4,5}
       m_im_j\delta_{ij}^3w_{ij}z_{i6}.
\end{aligned}
\]
Consequently,
\[
 U_T=O(\varepsilon^2)
       +O\!\left(\varepsilon^{(1-3\alpha)/2}\right).
\]
Similarly, contracting the original $z$ rows at bodies $3,4,5$ against $m_iw_{i6}$ gives
\[
\begin{aligned}
 U_Q={}&z_6\sum_{i=3,4,5}m_iw_{i6}
       +\sum_{i=3,4,5}m_i z_{i6}w_{i6}\\
 &-\sum_{i=3,4,5}\sum_{j=1,2}
       m_im_j\delta_{ij}^3z_{ij}w_{i6},
\end{aligned}
\]
and hence
\begin{equation}\label{eq:no4-relative-orders}
 U_Q=O\!\left(\varepsilon^{2(\alpha+2)/3}\right)
       +O\!\left(\varepsilon^{-(5\alpha+1)/6}\right).
\end{equation}
All four displayed exponents are strictly positive throughout the stated window. The remaining six pairs have total signed potential
\[
 \sum_{i=1,2}\sum_{j=3,4,5}m_im_j\delta_{ij}
 =O\!\left(\varepsilon^{-(\alpha+1)/2}\right)\longrightarrow0.
\]
These pairs, the pairs in $T$, and the pairs in $Q$ partition all fifteen pairs. Thus $U\to0$, contrary to the constant $U=c>0$ on a curve through a positive configuration.
\end{proof}

\begin{lemma}[Reference-body potential identity]\label{lem:reported-reference-triangle}
Let $T\subset\{1,\ldots,6\}$ have at least two elements, let $r\in T$, and put $U_T=\sum_{i<j,\ i,j\in T}m_im_j\delta_{ij}$. On the original normalized six-body equations,
\begin{equation}\label{eq:reference-triangle-potential}
\begin{aligned}
 U_T={}&z_r\sum_{i\in T\setminus\{r\}}m_iw_{ir}
       +\sum_{i\in T\setminus\{r\}}m_i z_{ir}w_{ir}\\
 &-\sum_{i\in T\setminus\{r\}}\sum_{j\notin T}
       m_im_j\delta_{ij}^3z_{ij}w_{ir}.
\end{aligned}
\end{equation}
The colour-dual identity is obtained by exchanging $z$ and $w$. No row at the reference body is treated as an independent subsystem equation.
\end{lemma}
\begin{proof}
Multiply the original $z_i$ row by $m_iw_{ir}$ and sum over $i\in T\setminus\{r\}$. For an internal pair $i,j$ with neither endpoint equal to $r$, its two contributions combine as
\[
 m_im_j\delta_{ij}^3z_{ij}(w_{ir}-w_{jr})
 =m_im_j\delta_{ij}^3z_{ij}w_{ij}
 =m_im_j\delta_{ij}.
\]
For the internal pair $ir$, its single contribution has the same value by the signed tie. Thus the internal terms sum to $U_T$, while the exterior terms are exactly the double sum in \eqref{eq:reference-triangle-potential}. Expanding $z_i=z_r+z_{ir}$ on the left and moving the exterior terms to the other side proves the identity. The original $w_i$ rows give its colour-dual form.
\end{proof}

\begin{lemma}[Internal potential with complete external forcing]
\label{lem:complete-cluster-potential}
For a set $T$ of bodies, put
\[
 M_T=\sum_{i\in T}m_i,\qquad
 \bar z_T=M_T^{-1}\sum_{i\in T}m_i z_i,\qquad
 \bar w_T=M_T^{-1}\sum_{i\in T}m_i w_i.
\]
On the original normalized equations,
\begin{equation}\label{eq:complete-centred-potential}
\begin{aligned}
 U_T={}&\sum_{i\in T}m_i(z_i-\bar z_T)(w_i-\bar w_T)\\
 &-\sum_{i\in T}\sum_{j\notin T}
 m_im_j\delta_{ij}^3z_{ij}(w_i-\bar w_T).
\end{aligned}
\end{equation}
In particular, suppose $T$ is a $zw$ triangle, so its two internal separation orders are $1$. If an exterior body $j$ has a $w$-only edge to $p\in T$ with $z$-order $s\in(1,2]$, then
\[
\begin{array}{c|c|c|c}
\text{pair}&\ord z&\ord w&\ord(\delta^3z)\\ \hline
 jp&s&4-3s&4s-6\\
 ji,\ i\in T\setminus\{p\}&1&4-3s&(9s-13)/2
\end{array}
\]
and the contribution of this body to the exterior sum in
\eqref{eq:complete-centred-potential} is
\begin{equation}\label{eq:matched-exterior-potential}
 O(\varepsilon^{4s-5})+
 O\!\left(\varepsilon^{(9s-11)/2}\right).
\end{equation}
The internal inertia in \eqref{eq:complete-centred-potential} is $O(\varepsilon^2)$.
\end{lemma}
\begin{proof}
Multiply every original $z_i$ row with $i\in T$ by
$m_i(w_i-\bar w_T)$ and sum. The internal pair $ij$ contributes
$m_im_j\delta_{ij}^3z_{ij}w_{ij}=m_im_j\delta_{ij}$.
The remaining terms are exactly the exterior sum displayed in
\eqref{eq:complete-centred-potential}; its left side is the internal inertia because $\sum_{i\in T}m_i(w_i-\bar w_T)=0$.

For any pair whose separation orders are $x,y$, the signed tie gives
\begin{equation}\label{eq:force-order-formulas}
 \ord\delta=-\frac{x+y}{2},\qquad
 \ord(\delta^3z)=-\frac{x+3y}{2},\qquad
 \ord(\delta^3w)=-\frac{3x+y}{2}.
\end{equation}
A $w$ stroke has $\ord(\delta^3w)=-2$, giving $y=4-3s$ on $jp$. Since $s>1$ and $4-3s<1$, the identities through $p$ give exact orders $1,4-3s$ on every other pair $ji$. Finally $w_i-\bar w_T=O(\varepsilon)$, which proves \eqref{eq:matched-exterior-potential}. No assumption about the absolute cluster centre, or about the signs of the complex inverse distances, enters this argument.
\end{proof}

\begin{lemma}[An unmarked finite pair]\label{lem:reported-finite-pair}
Let $ij$ be a pair with $m_i+m_j>0$, and write $R_i,R_j$ for the two original $z$ forces after removing their mutual interaction. On a germ, if $(w_i-w_j)(R_i-R_j)=o(1)$, then $\delta_{ij}$ tends to a finite nonzero complex value $d$ and $d^3=(m_i+m_j)^{-1}$.
\end{lemma}
\begin{proof}
Subtract the original $z_i,z_j$ rows and use the original signed tie $\delta_{ij}^2z_{ij}w_{ij}=1$. These give the exact on-solution identity
\begin{equation}\label{eq:finite-pair-identity}
1=(m_i+m_j)\delta_{ij}^3+\delta_{ij}^2(w_i-w_j)(R_i-R_j).
\end{equation}
If $\delta_{ij}\to0$, the right side tends to zero. If its modulus diverges, divide by $\delta_{ij}^3$: the last term becomes $o(1)/\delta_{ij}$, giving $m_i+m_j=0$, a contradiction. The remaining nonzero limit has the stated cube on every phase.
\end{proof}

\subsection{No.~105: incompatible centre and pole signs}\label{sec:no105-sign}
For illustration, No.~105 has just the signed inverse distance
poles $25,36$ of order $-1$. Its original centres give
$YP/(XQ)=AD/(BC)>0$ for four positive group masses
$A,B,C,D$. Its source rows and signed ties give
$YP=-m_2m_5D_{25}^4$ and $XQ=m_6^2D_{36}^4$, while
constancy of $U$ forces
$m_2m_5D_{25}+m_3m_6D_{36}=0$. Eliminating the pole
ratio makes $YP/(XQ)<0$, a contradiction. That local
argument does not substitute for the remaining 27 diagram
proofs and their valuation-completeness checks.

\subsection{No.~20: second-order matching and energy}\label{sec:no20-energy}
In No.~20, only the outer vertices $1,2,3$ have $w$-circles; the inner $zw$ triangle $T=\{4,5,6\}$ is uncircled. Write $p(1)=6$, $p(2)=5$, $p(3)=4$ for the three matches. Every outer--inner $w$ separation has order $-2$. On a marked match, the original $w$-force order and signed tie therefore give
\[
 \ord\delta_{i p(i)}=0,\qquad \ord z_{i p(i)}=2.
\]
The inner $z$- and $w$-differences have exact order $1$. Hence all six $z$ coordinates differ by $O(\varepsilon)$, and the original mass centre gives $z_a=O(\varepsilon)$ for every $a$. On a nonmatched outer--inner edge,
\[
 z_{ia}=z_{p(i)a}+O(\varepsilon^2),\qquad
 \ord\delta_{ia}=\tfrac12,\qquad
 \ord(\delta_{ia}^3z_{ia})=\tfrac52.
\]
Apply Lemma~\ref{lem:reported-reference-triangle} to $T$, with any inner reference vertex. Its reference and relative-inertia terms are $O(\varepsilon^2)$. The matched exterior $z$ forces are $O(\varepsilon^2)$, and the remaining exterior $z$ forces are $O(\varepsilon^{5/2})$; multiplication by the inner relative $w$ coordinates gives respectively $O(\varepsilon^3)$ and $O(\varepsilon^{7/2})$. Thus
\[
 U_{456}=O(\varepsilon^2).
\]
This estimate uses the original exterior forces and precedes the second-order matching calculation.

Let $W_1,W_2,W_3\ne0$ be the outer $w$-leads. Since the inner vertices are uncircled, their original centre gives $\sum_{i=1}^3m_iW_i=0$, so the three leads cannot all coincide. If they were pairwise distinct, every outer edge would have $z$-order $1$ and $w$-order $-2$. Its inverse distance and $z$ force would consequently have orders $1/2$ and $5/2$. The original outer $z$ rows would then give $z_i=O(\varepsilon^2)$, and the matched ties would put all three inner $z$ positions at $O(\varepsilon^2)$, contrary to their exact order-one differences. Hence exactly one outer pair has equal leads, say $W_i=W_j\ne W_k$.

The inverse distances on $ik,jk$ tend to zero with order $1/2$. Moreover, $z_{ij}=z_{p(i)p(j)}+O(\varepsilon^2)$ has exact order $1$. Subtracting the original $z_i,z_j$ rows, whose forces other than their mutual interaction are $O(\varepsilon^2)$, gives
\[
 \bigl(1-(m_i+m_j)\delta_{ij}^3\bigr)z_{ij}
 =O(\varepsilon^2).
\]
Thus $\delta_{ij}\to e_{ij}\ne0$ with $e_{ij}^3=(m_i+m_j)^{-1}$. If $d_i=\lim\delta_{i p(i)}$, all fifteen pairs have now been accounted for: the inner total potential tends to zero, the six nonmatched outer--inner inverse distances and the two unequal-lead outer inverse distances tend to zero, and only the three matches and $ij$ remain. Consequently
\[
 c_0=\lim U=\sum_{i=1}^3m_i m_{p(i)}d_i+m_im_je_{ij}.
\]

For the equal pair $(i,j)$ and unmatched outer body $k$, all three choices are related by relabelling. If $p,q,r$ are their matched inner vertices, the outer $z$ rows and the inner centre give $(m_i,m_j)=\lambda(m_p,m_q)$, $\lambda>0$. The internal triangle rows and ties require, for $(a,b,c)=(m_p,m_q,m_r)$,
\begin{equation}\label{eq:no20-mass-condition}
(ab)^3=c^2(a+b)^4,\qquad a\ne b.
\end{equation}
This is not an exclusion: the zero set is nonempty. The inner signed shape root is simple because its derivative is $3c^2(a+b)^3(b-a)\ne0$. Its first correction and the original mass-weighted outer $z$ rows force the matched offsets to satisfy
\begin{equation}\label{eq:no20-offsets}
m_iA_i+m_jA_j+m_kA_k=0.
\end{equation}
The preceding finite-edge reduction gives the complete constant $U=I=c_0>0$. Keeping all signed cube phases, the real positivity of $c_0$ and the offset identity force their phases to be $1$. Put $A=a+b$, $B=c^{2/3}$, $H=a^{5/3}+b^{5/3}$, $d=\sqrt{AH/B}$ and $x=d/A\ge19/12$. With $m_k=\lambda d$, the physical H\"older inequality $\sigma^3\le Mc_0^3$ would imply
\begin{equation}\label{eq:no20-holder-gap}
(\lambda x+1+x)^3\le\bigl(\lambda(1+x)+9/8\bigr)(27/2+\lambda^2/16).
\end{equation}
The difference of left and right increases with $x\ge19/12$; at $x=19/12$ its ascending $\lambda$ coefficients are $(3547/1728,-1829/576,22301/1152,1645/432)$. The middle quadratic has discriminant test $4ac-b^2=148167571/995328>0$, making that difference strictly positive. Thus no positive-mass No.~20 first limit survives any of the three equal-pair branches. The simple-root and signed-phase checks are required certificate inputs.

\subsection{No.~57: two outer triangle regimes}\label{sec:no57-triangle}
The outer $w$-triangle $123$ has a common $z$-order $\alpha\in(1,2]$, and its inner $zw$ triangle is $456$.

\paragraph{No.~57: the range $1<\alpha<2$.}
Put $A=\{1,2,3\}$, $B=\{4,5,6\}$ and $b=4-3\alpha$.
The internal orders of $A$ are $(\alpha,b)$, and those of $B$ are $(1,1)$.
The original $w$ centre forces the two cluster leading values to be distinct: if they were equal, all six nonzero circled leading values would agree. Hence the bridge $36$ is maximal, with orders $(2,-2)$ and $\delta_{36}\to d\ne0$.
All nine cross pairs have the following exact orders:
\[
\begin{array}{c|c|c|c|c}
\text{pair}&\ord z&\ord w&\ord\delta&\ord(\delta^3z)\\ \hline
36&2&-2&0&2\\
16,26&\alpha&-2&(2-\alpha)/2&3-\alpha/2\\
 i j,\ i\in A,\ j=4,5&1&-2&1/2&5/2
\end{array}
\]
These entries follow by adding the internal differences to $z_{36},w_{36}$; in each addition the two competing orders are distinct.
Apply \eqref{eq:complete-centred-potential} first to $B$.
Its relative $w$ coordinates are $O(\varepsilon)$, so the table gives
\[
 U_B=O(\varepsilon^2)+O(\varepsilon^3)
      +O(\varepsilon^{4-\alpha/2})+O(\varepsilon^{7/2})
      =O(\varepsilon^2).
\]
For $A$, the relative $w$ coordinates are $O(\varepsilon^b)$ and its internal inertia is $O(\varepsilon^{\alpha+b})$. Thus
\begin{equation}\label{eq:no57-complete-potential}
\begin{aligned}
 U_A={}&O(\varepsilon^{4-2\alpha})
 +O(\varepsilon^{6-3\alpha})
 +O(\varepsilon^{7-7\alpha/2})
 +O(\varepsilon^{13/2-3\alpha})\longrightarrow0.
\end{aligned}
\end{equation}
Every exponent is positive for $1<\alpha<2$. The table also proves that every cross inverse distance except $36$ tends to zero, so
$c=m_3m_6d$.
In the original mass-weighted $w$ rows of $A$, only $36$ contributes at order $-2$: the other cross $w$-forces have orders $1-3\alpha/2>-2$ and $-1/2$. If $W_A,W_B$ are the cluster leading values, then
\[
 M_AW_A=m_3m_6d^3(W_A-W_B),\qquad
 M_AW_A+M_BW_B=0.
\]
Since $W_A\ne W_B$, this gives $m_3m_6d^3=M_AM_B/M$ and
$Mc^3=M_AM_B(m_3m_6)^2$. The strict inequality
$\sigma>M_AM_B>m_3m_6$ contradicts physical H\"older.

If $\alpha=2$, the outer triangle is maximal. When $W_3\ne W_{456}$, maximal $13$ and $36$ force the forbidden closing stroke $16$. When $W_3=W_{456}$, the original body-3 $z$-row has a unique nonzero $36$ term at order $4\gamma-6<2$, preceding its other terms by strict positive order gaps for $1<\gamma<2$. Thus both maximal branches are impossible even without a physical point.

\subsection{No.~51: the triangle and pendant at 16}\label{sec:no51-pendant}
Let $\alpha\in(1,2]$ be the common $z$-order on its $w$-only triangle $236$ and $\gamma\in(1,2]$ the $z$-order on pendant $16$. The four endpoint choices are exhaustive. If both equal $2$, maximal $16$ and $26$ require a missing $12$ stroke. If both are less than $2$, the six circled $w$-leads coincide, contrary to their positive-mass centre. For $\alpha=2$, $\gamma<2$, the original body-1 $z$-row has its nonzero $16$ force uniquely earliest at $4\gamma-6$. Finally, for $\alpha<2$, $\gamma=2$, the least seven-product order is $12-7\alpha$; a first limit needs $12/7<\alpha<2$.

\paragraph{Joint energy argument for No.~51 and No.~55.}
Put $T=\{2,3,6\}$, $Q=\{4,5,6\}$ and $b=4-3\alpha$, where
$12/7<\alpha<2$. The internal orders of $T$ are $(\alpha,b)$, and those of $Q$ are $(1,1)$. The pendant is maximal: it is $16$ in No.~51 and $13$ in No.~55.
For $i=2,3$ and $j=4,5$, adding differences through body $6$ gives
\begin{equation}\label{eq:overlap-cross-orders}
 \ord z_{ij}=1,\qquad \ord w_{ij}=b,\qquad
 \ord\delta_{ij}=\frac{3\alpha-5}{2},\qquad
 \ord(\delta_{ij}^3z_{ij})=\frac{9\alpha-13}{2}.
\end{equation}
In No.~51, the pairs $12,13$ have orders $(\alpha,-2)$ and $14,15$ have orders $(1,-2)$. The original body-$1$ $z$ row therefore has its pendant force at order $2$, the $12/13$ forces at order $3-\alpha/2>2$, and the $14/15$ forces at order $5/2$. Consequently
\begin{equation}\label{eq:no51-reference-centre}
 z_1=O(\varepsilon^2),\qquad z_6=O(\varepsilon^2).
\end{equation}
In No.~55, the pairs $12,16$ have orders $(\alpha,-2)$ and $14,15$ again have orders $(1,-2)$. The same original row gives
\begin{equation}\label{eq:no55-reference-centre}
 z_1,z_3=O(\varepsilon^2),\qquad z_6=O(\varepsilon^\alpha).
\end{equation}

For $U_Q$, the exterior bodies $2,3$ each attach to $6$ with order $\alpha$.
Lemma~\ref{lem:complete-cluster-potential} bounds their complete contributions by
$O(\varepsilon^{4\alpha-5})+O(\varepsilon^{(9\alpha-11)/2})$.
In No.~51 the remaining exterior body $1$ attaches maximally to $6$ and contributes $O(\varepsilon^3)+O(\varepsilon^{7/2})$.
In No.~55 its pair with $6$ has orders $(\alpha,-2)$ and its pairs with $4,5$ have orders $(1,-2)$; its contribution is therefore
$O(\varepsilon^{4-\alpha/2})+O(\varepsilon^{7/2})$.
In both diagrams,
\begin{equation}\label{eq:overlap-inner-potential}
 U_Q=O(\varepsilon^2)+O(\varepsilon^{4\alpha-5})
     +O(\varepsilon^{(9\alpha-11)/2})\longrightarrow0.
\end{equation}

For $U_T$, use the reference-body identity with reference $6$, rather than assigning a separate subsystem equation to the shared body $6$. Its relative inertia is $O(\varepsilon^{4-2\alpha})$.
Equations~\eqref{eq:no51-reference-centre} and
\eqref{eq:no55-reference-centre} bound the reference term by
$O(\varepsilon^{6-3\alpha})$ and
$O(\varepsilon^{4-2\alpha})$, respectively.
The four exterior interactions from $i=2,3$ to $j=4,5$ contribute
\[
 O\!\left(\varepsilon^{(9\alpha-13)/2+b}\right)
 =O\!\left(\varepsilon^{(3\alpha-5)/2}\right).
\]
The interactions with body $1$ contribute
$O(\varepsilon^{7-7\alpha/2})$ in No.~51, and
$O(\varepsilon^{6-3\alpha})+O(\varepsilon^{7-7\alpha/2})$ in No.~55.
Thus in both cases
\begin{equation}\label{eq:overlap-outer-potential}
\begin{aligned}
 U_T={}&O(\varepsilon^{4-2\alpha})
 +O(\varepsilon^{6-3\alpha})
 +O(\varepsilon^{(3\alpha-5)/2})
 +O(\varepsilon^{7-7\alpha/2})\longrightarrow0.
\end{aligned}
\end{equation}
All exponents are positive on $12/7<\alpha<2$.
The pair sets inside $T$ and $Q$ are disjoint. Together with the four pairs in
\eqref{eq:overlap-cross-orders} and the five pairs incident to body $1$, they exhaust the fifteen pairs. The former four inverse distances tend to zero; among the latter five, only the pendant has a nonzero finite limit $d$.
Writing $r=6$ for No.~51 and $r=3$ for No.~55, we obtain
\[
 c=m_1m_r d,\qquad m_rd^3=\frac{M-m_1}{M},\qquad
 Mc^3=m_1^3m_r^2(M-m_1).
\]
Here the cube relation follows from the original body-$1$ $w$ row and the full $w$ centre: its only order-$-2$ force is the pendant force, while bodies $2,\ldots,6$ share one leading $w$ value. Finally
$\sigma>m_1(M-m_1)$ and $M-m_1>m_r$ give
$\sigma^3>Mc^3$, contrary to physical H\"older.

\subsection{No.~40: two maximal strokes force the missing edge}
\label{sec:no40-pole}
Body $1$ has no $w$ circle, whereas bodies $3$ and $6$ have $w$ circles. Therefore the marked edges $13$ and $16$ both have exact $w$-order $-2$.
Their marked $w$-force order and the signed ties give
$\ord z_{13}=\ord z_{16}=2$.
Consequently $z_{36}=z_{31}+z_{16}=O(\varepsilon^2)$.
Put $x=\ord z_{36}\ge2$ and $y=\ord w_{36}\ge-2$, the latter inequality following from normalization. The exact signed tie implies
\[
 \ord(\delta_{36}^3w_{36})=-\frac{3x+y}{2}\le-2.
\]
Strict inequality violates normalization; equality requires a $36$ $w$ stroke. The diagram has no such stroke. This excludes No.~40 without any potential cancellation or any assumption on the order of $25$.

\subsection{No.~55: the triangle and pendant at 13}\label{sec:no55-pendant}
Here the $w$-only triangle is $236$ but the pendant is $13$, not $16$ as in No.~51. Write $\alpha,\gamma\in(1,2]$ for their $z$-orders. With $\alpha=\gamma=2$, consecutive maximal $13$ and $23$ require a missing $12$ stroke; with $\alpha,\gamma<2$, the original $w$-centre forbids a common nonzero lead at all six bodies. When $\alpha=2$, $\gamma<2$, the original body-1 $z$-row has the $13$ force uniquely first at $4\gamma-6$, earlier than its coordinate and all exterior forces. In the remaining branch $\gamma=2$, $\alpha<2$, a first-limit pole requires $12/7<\alpha<2$ by the product order $12-7\alpha$. The joint argument in Section~\ref{sec:no51-pendant}, in particular \eqref{eq:no55-reference-centre} and \eqref{eq:overlap-inner-potential}--\eqref{eq:overlap-outer-potential}, proves that both complete internal potentials tend to zero. Only the finite pendant $13$ remains, and
\begin{equation}\label{eq:no55-pole}
m_3d^3=(M-m_1)/M,\qquad c=m_1m_3d,\qquad Mc^3=m_1^3m_3^2(M-m_1).
\end{equation}
The strict inequalities $\sigma>m_1(M-m_1)$ and $M-m_1>m_3$ again contradict physical H\"older. All four branches are excluded with the correct pendant-specific reference estimate.

\subsection{No.~44: an unpaired star force}\label{sec:no44-force}
The three $w$-only spokes $16,26,36$ meet the inner $zw$ triangle $456$. Positive mass centre excludes no maximal spoke; two maximal spokes would require a missing outer closing $w$-stroke. Thus exactly one is maximal, which we label $16$. Put $a=\ord z_{26}$, $b=\ord z_{36}$, both in $(1,2)$. If $a=b$, the missing $23$ signed tie has $w$-force order at most $-2$, either drawing an absent maximal stroke or violating normalization. Otherwise swap $2$ and $3$ to take $a<b$. The original body-1 $z$-row and maximal $16$ tie give $z_1,z_6=O(\varepsilon^2)$. In the original body-2 $z$-row the nonzero $26$ force has order $4a-6$. Its coordinate, $21$, $23$ and $24/25$ forces are later by
\begin{equation}\label{eq:no44-order-gaps}
6-3a,\qquad \frac92(2-a),\qquad \frac92(b-a),\qquad \frac{a-1}{2},
\end{equation}
all strictly positive. The unmarked $23$ inverse distance was retained in this comparison; its possible pole cannot cancel the unique $26$ term. The $b<a$ case is the body-3 image. Thus no positive-mass complex No.~44 germ exists.

\Needspace{12\baselineskip}
\subsection{No.~49: all extremal sets of the $w$-path}\label{sec:no49-path}
The $w$-only path $1-6-3-2$ has maximal edge set one of
$\{16\}$, $\{23\}$, $\{36\}$ or $\{16,23\}$. The empty set
violates the original $w$-centre; consecutive maximal 16/36
or 23/36 would draw a missing closing edge. For only 16 or
only 23, equality of the two remaining nonmaximal $z$-orders
likewise forces the absent closing $w$-stroke. For unequal
orders, whichever nonmaximal edge has smaller order
contributes a unique nonzero $4s-6$ term to its original
outer $z$-row. For only 36, subtracting the original body-2/3
$z$-rows leaves the nonzero 23 term of order $4s-6$ earlier
than the position difference of order $s$ and all exterior
forces of order $2$.
When both $16$ and $23$ are maximal, retain the unmarked pair $12$. Write $g=\ord z_{36}\in(1,2)$, $\lambda=\ord w_{12}$ and $k=4g-6<2$. The maximal ties give $z_{16},z_{23}=O(\varepsilon^2)$, so $z_{12}=-z_{36}+O(\varepsilon^2)$ has exact order $g$. The missing $12$ stroke forces $\lambda<4-3g$. Consequently its original $z$-force has order
\[
 \ord(\delta_{12}^3z_{12})=-\frac{g+3\lambda}{2}
 =k+\frac32(4-3g-\lambda)>k.
\]
In the original body-$2$ $z$-row, the forces from $23$, $24/25$ and $26$ have orders $2$, $5/2$ and $3-g/2$, respectively, all greater than $k$. Together with the preceding estimate for $21$, this gives
\[
 z_2=o(\varepsilon^k),\qquad
 z_3=z_2-z_{23}=o(\varepsilon^k).
\]
In the original body-$3$ $z$-row, the $36$ force is nonzero at exact order $k$. The $32$ and $31$ forces have orders $2$ and $3-g/2$, while the $34/35$ forces have order
\[
 \frac{9g-13}{2}=k+\frac{g-1}{2}>k.
\]
Thus every other term in that row, including its left side, is $o(\varepsilon^k)$, a contradiction. All four maximality sets fail, including the potentially unbounded $12$-distance branch.
\Needspace{18\baselineskip}
\subsection{No.~45: the two possible potential-pole regimes}
\label{sec:no45-pole}
The $zw$ triangle is $Q=\{4,5,6\}$, the exterior $w$-only edges are $16,26,35$, and all six bodies have $w$ circles.
At least one exterior edge is maximal, since otherwise all six circled leading $w$ values coincide, contrary to their mass centre.
The edges $16$ and $26$ cannot both be maximal: their $z$ differences would force $z_{12}=O(\varepsilon^2)$, and normalization would then require the missing $12$ $w$ stroke.
Thus, up to interchanging bodies $1$ and $2$, the possible maximal sets are $\{16\}$, $\{35\}$ and $\{16,35\}$.

Suppose first that only $16$ is maximal. Put $a=\ord z_{26}$ and $b=\ord z_{35}$, both in $(1,2)$. The original body-$1$ $z$ row gives
$z_1,z_6=O(\varepsilon^2)$: its $16$ force has order $2$, its $12$ force has order $3-a/2>2$, and its $13/14/15$ forces have order $5/2$.
Hence $z_2$ has exact order $a$, while $z_{23}$ has exact order $1$.
If $a\ne b$, then
\[
 \ord w_{23}=4-3\max\{a,b\},\qquad
 \ord(\delta_{23}^3z_{23})=\frac{9\max\{a,b\}-13}{2}>4a-6.
\]
The $24/25$ forces have order $(9a-13)/2>4a-6$, and the $21$ force has order $3-a/2>4a-6$.
Since $a>4a-6$, the $26$ force is the unique earliest term in the original body-$2$ $z$ row, a contradiction.

It remains to consider $a=b=s\in(1,2)$. Write $\lambda=\ord w_{23}$.
The separation identities give $\lambda\ge4-3s$, while $\ord z_{23}=1$.
All terms in the original body-$2$ $z$ row except its $26$ and $23$ forces are later than $4s-6$. Their cancellation therefore requires
\[
 -\frac{1+3\lambda}{2}=4s-6,\qquad
 \lambda=\frac{11-8s}{3}>4-3s,\qquad
 \ord\delta_{23}=\frac{4s-7}{3}>s-2.
\]
The strict inequality for $\lambda$ shows that $w_{26}$ and $w_{35}$ have the same nonzero leading coefficient $u$. Write
\[
 \delta_{26}=A\varepsilon^{s-2}(1+o(1)),\qquad
 \delta_{35}=D\varepsilon^{s-2}(1+o(1)),\qquad ADu\ne0.
\]
Bodies $2,3$ and $Q$ have the same nonzero circled leading $w$ value $W$.
Their original $w_2,w_3$ rows give
\[
 W=m_6A^3u=m_5D^3u,\qquad (D/A)^3=m_6/m_5>0.
\]
Indeed the $23$ $w$-force has order $(4s-10)/3>-2$, each nonmatched force into $Q$ has order $(3s-7)/2>-2$, and the remaining exterior forces are also later than $-2$.

We now retain the complete internal potential. Lemma~\ref{lem:complete-cluster-potential} gives
\begin{equation}\label{eq:no45-relative-inner-potential}
 U_Q=O(\varepsilon^2)+O(\varepsilon^{4s-5})
      +O(\varepsilon^{(9s-11)/2})
      =o(\varepsilon^{s-2}).
\end{equation}
The final equality uses $s>1$; it does not require $U_Q\to0$.
The nonmatched outer-to-inner inverse distances have order $(3s-5)/2>s-2$, and the $23$ inverse distance has order $(4s-7)/3>s-2$.
The inverse distances on $12,13$ tend to zero, while $16$ has a finite limit.
Consequently the complete potential satisfies
\[
 U=(m_2m_6A+m_3m_5D)\varepsilon^{s-2}
      +o(\varepsilon^{s-2}).
\]
Its constancy forces $D/A=-m_2m_6/(m_3m_5)<0$, contradicting the positive cube ratio above.

Suppose next that only $35$ is maximal. Put $a=\ord z_{16}$ and $b=\ord z_{26}$.
If $a=b$, then $\ord z_{12}\ge a$ and $\ord w_{12}\ge4-3a$; the signed tie forces the missing $12$ $w$ stroke or violates normalization. Thus, after interchanging $1,2$, take $a<b$.
The original body-$3$ $z$ row gives $z_3,z_5=O(\varepsilon^2)$.
The $zw$ edge $56$ then implies that $z_6,z_1,z_2$ have exact order $1$.
The $12$ $z$-force has order $(9b-a-12)/2$, later than both the $16$ force of order $4a-6$ and the $26$ force of order $4b-6$.
All remaining forces in the respective original body-$1$ and body-$2$ $z$ rows are later than those marked forces: the nonmatched inner forces exceed them by $(a-1)/2$ and $(b-1)/2$, and the body-$3$ forces have order $5/2$.
The two original rows therefore require $4a-6=4b-6=1$, contrary to $a<b$.

Finally suppose that $16$ and $35$ are maximal, while $s=\ord z_{26}\in(1,2)$.
Multiply the original body-$1$ and body-$3$ $z$ rows by $m_1,m_3$ and add.
Their mutual $13$ interaction cancels exactly; every remaining force is $O(\varepsilon^2)$.
The maximal ties yield
\[
 m_1z_6+m_3z_5=O(\varepsilon^2),\qquad
 z_{56}\text{ has exact order }1.
\]
Since $m_1,m_3>0$, both $z_5$ and $z_6$ have exact order $1$, as does $z_2=z_6+O(\varepsilon^s)$.
In the original body-$2$ $z$ row, the $26$ force has order $4s-6$; its $24/25$ forces have order $(9s-13)/2$, its $21$ force has order $3-s/2$, and its $23$ force has order $5/2$.
Thus $4s-6=1$, giving $s=7/4$.
The unmarked pair $13$ must still be retained. Its separation $z_{13}$ has exact order $1$, and subtraction of its two original rows gives
\[
 \bigl(1-(m_1+m_3)\delta_{13}^3\bigr)z_{13}=O(\varepsilon^2).
\]
Hence $\delta_{13}^3=(m_1+m_3)^{-1}+O(\varepsilon)$, so this inverse distance has a finite nonzero limit.
For the internal triangle, \eqref{eq:matched-exterior-potential} now gives
\[
 U_Q=O(\varepsilon^2)+O(\varepsilon^{4s-5})
       +O(\varepsilon^{(9s-11)/2})=O(\varepsilon^2).
\]
All exterior inverse distances other than $26$ are finite or tend to zero: $16,35,13$ are finite, and every remaining pair follows from the distinct orders just listed.
The nonzero pole $m_2m_6\delta_{26}$ of order $s-2=-1/4$ is therefore the unique pole of the complete potential, contradicting its constancy. All three maximality sets are excluded.

\subsection{No.~31: a finite unmarked pair and the bridge gap}\label{sec:no31-bridge}
The vertices $3,4,5,6$ have $w$-circles, whereas $1,2$ are uncircled and have no strokes. For $i\in\{1,2\}$ and $a\in\{3,4,5,6\}$, the $w$ separation has order $-2$ and its unmarked $w$ force is $o(\varepsilon^{-2})$. Thus $\delta_{ia}=o(1)$ and $\delta_{ia}^3z_{ia}=\delta_{ia}/w_{ia}=o(\varepsilon^2)$. In particular, the remaining forces in the original $z_1,z_2$ rows satisfy $R_1-R_2=o(\varepsilon^2)$; since $w_{12}=o(\varepsilon^{-2})$, Lemma~\ref{lem:reported-finite-pair} gives $\delta_{12}\to e\ne0$ with $e^3=(m_1+m_2)^{-1}$.

A nonmaximal bridge $36$ would make all four circled $w$-leads equal, contradicting their original mass centre. Hence $36$ is maximal, $z_{36}=O(\varepsilon^2)$, and $\delta_{36}\to d\ne0$. The nonmatched edges $34,35$ have $z$-order $1$, $w$-order $-2$, and inverse-distance order $1/2$. The original body-3 $z$ row therefore gives $z_3=O(\varepsilon^2)$, and then $z_6=O(\varepsilon^2)$. Apply Lemma~\ref{lem:reported-reference-triangle} to $456$ with reference vertex 6. Its reference and relative terms have orders at least $3$ and $2$; its exterior terms from body 3 have order $7/2$, and those from bodies 1 and 2 are $o(\varepsilon^3)$. Thus $U_{456}=O(\varepsilon^2)$. All remaining inverse distances other than $12,36$ tend to zero. The original body-3 $w$ row and mass centre give $d^3=(m_4+m_5+m_6)/[m_6(m_3+m_4+m_5+m_6)]$, so only $12$ and $36$ remain in $c=\lim U$.

Let $A=m_1+m_2$, $C=m_3+m_4+m_5+m_6$, $X=m_1m_2$, and let $Y$ be the inner pair-mass sum. Then
\begin{equation}\label{eq:no31-holder-gap}
M^{1/3}c<X+Y+\frac{XC}{3A}+\frac{YA}{3C}\le\sigma-\frac{19}{24}AC<\sigma.
\end{equation}
This contradicts H\"older at the physical point for every signed phase of the two finite edges; it does not exclude an unrelated complex-only curve.

\subsection{No.~50: five path branches and the resonance}\label{sec:no50-resonance}
The possible maximal sets on $w$-only edges $16,23,35$ are $\{16\}$, $\{23\}$, $\{35\}$, $\{16,23\}$ and $\{16,35\}$; no maximum violates the $w$-centre, and maximal $23$ and $35$ together require a missing $25$ edge. In the only-$16$ and only-$35$ cases, a missing outer tie or a unique earlier $z$-force excludes unequal-order branches; the only-$35$ case also follows by subtracting the original body-2/3 $z$-rows. For both $16$ and $23$, the original weighted outer rows cancel the unmarked $12$ force. Their remaining $35$ edge would be the unique signed potential pole. For both $16$ and $35$, the $12$ and $13$ $w$-differences cannot simultaneously reach the order needed to cancel $23$, since their difference is $w_{23}$; the endpoint $\ord z_{23}=7/4$ again leaves a unique $23$ potential pole.

It remains to treat the only-$23$ branch in which
$\ord z_{16}=\ord z_{35}=s$, with $1<s<2$.
The $w$-stroke equations and the signed ties give
\[
 a:=\ord\delta_{16}=\ord\delta_{35}=s-2,\qquad
 b:=\ord w_{16}=\ord w_{35}=4-3s.
\]
The maximal edge $23$ has separation orders $(2,-2)$ and
$\ord\delta_{23}=0$. Every internal edge of the $zw$ triangle
$456$ has separation orders $(1,1)$. Consequently
\[
 \begin{array}{c|ccc}
 ij&\ord z_{ij}&\ord w_{ij}&\ord\delta_{ij}\\ \hline
 12,24,26&1&-2&1/2\\
 25&s&-2&(2-s)/2\\
 14,15,34,36&1&b&k:=(3s-5)/2.
 \end{array}
\]
For example, $z_{12}=z_{65}+z_{16}-z_{23}-z_{35}$ has
the nonzero order-one coefficient of $z_{65}$, while
$w_{15}=w_{16}+w_{65}$ has the nonzero order-$b$ coefficient
of $w_{16}$. The other entries follow by the same difference
identities. The five force orders in the original $z_2$ row
are $5/2,2,5/2,3-s/2,5/2$, in body order $1,3,4,5,6$.
Thus
\begin{equation}\label{eq:no50-position-orders}
 z_2,z_3=O(\varepsilon^2),\qquad z_5=O(\varepsilon^s),\qquad
 z_1=\zeta\varepsilon+o(\varepsilon),\qquad \zeta\ne0.
\end{equation}
In particular $z_{13}=\zeta\varepsilon+o(\varepsilon)$, and
all six $z$ positions are $O(\varepsilon)$.

Suppose first that $1<s<7/4$, and put
$h=4s-6<1$ and $t=(s-1)/3>0$.
The $16$ force in the original $z_1$ row has order $h$.
The coordinate term has order $1$, the $12$ force has order
$5/2$, and the $14,15$ forces have order
$(9s-13)/2=h+(s-1)/2>h$. Hence the $13$ force must have
order $h$ as well. Its signed tie determines
\begin{equation}\label{eq:no50-resonance-orders}
 \ord\delta_{13}=q:=\frac{4s-7}{3}=a+t<0,\qquad
 \ord w_{13}=\frac{11-8s}{3}=b+t.
\end{equation}
The exact identity $w_{13}=w_{16}-w_{35}+w_{65}$, with
$\ord w_{65}=1>b+t$, shows that $w_{16}$ and $w_{35}$
have the same nonzero leading coefficient $u$. Write
\[
 \begin{aligned}
 \delta_{16}&=\varepsilon^a A(\varepsilon),&A(0)&=A_0\ne0,\\
 \delta_{35}&=\varepsilon^a D(\varepsilon),&D(0)&=D_0\ne0,\\
 w_{16}&=u\varepsilon^b(1+o(1)),&
 w_{13}&=v\varepsilon^{b+t}(1+o(1)),\\
 \delta_{13}&=E_0\varepsilon^q(1+o(1)).&&
 \end{aligned}
 \qquad uvE_0\ne0.
\]
It follows for the full Puiseux functions that
\begin{equation}\label{eq:no50-w-ratio}
 \frac{w_{35}}{w_{16}}
 =1-\frac vu\varepsilon^t+o(\varepsilon^t).
\end{equation}
In the sum of the original $z_1,z_3$ rows weighted by
$m_1,m_3$, the $13$ force cancels exactly. At order $h$,
the remaining two terms and their signed ties give
\begin{equation}\label{eq:no50-leading-cancellation}
 L_0:=m_1m_6A_0+m_3m_5D_0=0.
\end{equation}
The original $z_1$ row itself, using
$\delta_{ij}^3z_{ij}=\delta_{ij}/w_{ij}$, gives
\begin{equation}\label{eq:no50-E-coefficient}
 \frac{m_6A_0}{u}+\frac{m_3E_0}{v}=0,
 \qquad m_1m_3E_0=m_3m_5D_0\frac vu.
\end{equation}

The $z$ equations above are used at order $h$. We determine
the next potential coefficient from the original $w$ rows
through order $-2+t$. Put
$T=m_1+m_4+m_5+m_6$. The exterior $w$ forces other than
$16,23,35$ have the following orders:
\[
 \begin{array}{c|c}
 ij&\ord(\delta_{ij}^3w_{ij})\\ \hline
 12,24,26&-1/2\\
 13&-2+4t\\
 14,15,34,36&-2+9t/2\\
 25&1-3s/2.
 \end{array}
\]
All these orders exceed $-2+t$; for the last entry the
difference is $(20-11s)/6>0$. The original $w_1$ row and
the sum of the original $w_2,w_3$ rows weighted by $m_2,m_3$
therefore yield
\[
 \begin{aligned}
 w_1&=m_6\delta_{16}^3w_{16}+o(\varepsilon^{-2+t}),\\
 m_2w_2+m_3w_3
 &=m_3m_5\delta_{35}^3w_{35}+o(\varepsilon^{-2+t}).
 \end{aligned}
\]
The internal $23$ force cancels in the second equation.
The full mass centre also gives
\[
 m_2w_2+m_3w_3=-Tw_1+O(\varepsilon^b),
 \qquad b-(-2+t)=\frac{19-10s}{3}>0.
\]
Consequently
\begin{equation}\label{eq:no50-full-w-relation}
 Tm_6\delta_{16}^3w_{16}
 +m_3m_5\delta_{35}^3w_{35}
 =o(\varepsilon^{-2+t}).
\end{equation}
Since $3a+b=-2$, division by $\delta_{16}^3w_{16}$ gives
\[
 \left(\frac DA\right)^3\frac{w_{35}}{w_{16}}
 =-\frac{Tm_6}{m_3m_5}+o(\varepsilon^t),\qquad
 \left(\frac{D_0}{A_0}\right)^3=-\frac{Tm_6}{m_3m_5}.
\]
Taking the cube root selected by the actual Puiseux branch
and using \eqref{eq:no50-w-ratio}, we obtain
\begin{equation}\label{eq:no50-phase-ratio}
 \frac{D/A}{D_0/A_0}
 =1+\frac{v}{3u}\varepsilon^t+o(\varepsilon^t).
\end{equation}
Thus \eqref{eq:no50-leading-cancellation} and
\eqref{eq:no50-E-coefficient} imply
\begin{align}
 m_1m_6\delta_{16}+m_3m_5\delta_{35}
 &=\varepsilon^a A\left(m_1m_6+m_3m_5\frac DA\right)\nonumber\\
 &=\frac13m_1m_3E_0\varepsilon^q+o(\varepsilon^q).
 \label{eq:no50-pair-potential}
\end{align}
This calculation retains every earlier correction to the full
functions $A$ and $D$.

We next control the complete internal triangle potential.
Multiply the original $z_5,z_6$ rows by $m_5w_{54},m_6w_{64}$
and sum. The signed ties give the exact identity
\begin{equation}\label{eq:no50-internal-potential}
 U_{456}
 =\sum_{i\in\{5,6\}}m_iz_iw_{i4}
 -\sum_{i\in\{5,6\}}\sum_{j\in\{1,2,3\}}
 m_im_j\delta_{ij}^3z_{ij}w_{i4}.
\end{equation}
The first sum is $O(\varepsilon^2)$. In the second sum, the
$16,35$ $z$ forces have order $h$, the $15,36$ forces have
order $h+(s-1)/2$, the $25$ force has order $3-s/2$, and
the $26$ force has order $5/2$. Since $w_{54},w_{64}$ have
order $1$,
\[
 U_{456}=O(\varepsilon^2)+O(\varepsilon^{4s-5})
 =o(\varepsilon^q),\qquad (4s-5)-q=\frac{8(s-1)}3>0.
\]
All remaining exterior potentials are later as well: their
orders are $k,1/2,(2-s)/2,0$, and
$k-q=(s-1)/6>0$. The unmarked $13$ potential must, however,
be retained at order $q$. Combining it with
\eqref{eq:no50-pair-potential} yields the complete expansion
\begin{equation}\label{eq:no50-complete-potential}
 U=\frac43m_1m_3E_0\varepsilon^q+o(\varepsilon^q).
\end{equation}
Its coefficient is nonzero and $q<0$, contrary to constancy
of $U$. This excludes the whole range $1<s<7/4$, including
the first-limit window $13/8<s<7/4$.

At $s=7/4$, \eqref{eq:no50-position-orders} still holds,
while $a=-1/4$, $b=-5/4$ and $h=1$. The original $z_3$ row
has coordinate term $O(\varepsilon^2)$ and a nonzero
order-one $35$ force. Its only possible order-one counterterm
is $13$. Thus $\ord\delta_{13}=0$ and $\ord w_{13}=-1>b$,
so $w_{16},w_{35}$ again share their leading coefficient $u$.
Equation~\eqref{eq:no50-internal-potential} gives
$U_{456}=O(\varepsilon^2)$, and all other exterior potentials
are bounded. Constancy of $U$ forces its order-$-1/4$
coefficient $L_0=m_1m_6A_0+m_3m_5D_0$ to vanish. But the
original weighted $z_1+z_3$ rows at order $1$ give
$m_1\zeta=L_0/u$, contradicting $\zeta\ne0$.

Finally, for $7/4<s<2$, the same weighted $z_1+z_3$ rows
cancel the $13$ force exactly. Their left side has the
nonzero order-one coefficient $m_1\zeta$; their $16,35$
forces have order $h>1$, their $23$ force has order $2$,
and every other force is strictly later than order $1$.
This is again impossible.

\subsection{No.~32: the missing 12 tie contradicts body 2}\label{sec:no32-tie}
The original $w$-centre of its $zw$ triangle $456$ and $w$-only triangle $236$ makes the three $236$ $w$-strokes maximal, with nonzero signed internal inverse-distance limits. The original body-3 $z$-row and its three maximal ties then give $z_2,z_3,z_6=O(\varepsilon^2)$, without using the unmarked 12 edge. With $R_i$ the original body-$i$ $z$-force excluding pair 12, the graph gives $R_1=O(\varepsilon^2)$, $R_2=O(\varepsilon^2)$ and $w_{12}=o(\varepsilon^{-2})$. The exact equation in Lemma~\ref{lem:reported-finite-pair} forces a finite nonzero $e=\lim\delta_{12}$ with $e^3=1/(m_1+m_2)$. Only now apply the original body-2 $z$-row: its 12 force and $z_2,R_2=O(\varepsilon^2)$ force $z_{12}=O(\varepsilon^2)$. But the original signed 12 tie has
\[
\delta_{12}^2z_{12}w_{12}=o(1)
\]
instead of $1$. This excludes even fixed-positive-mass complex germs of No.~32.

\subsection{No.~84: optional outer edges and a four-cycle}\label{sec:no84-cycle}
Its marked pairs are $14,23,56$, each maximal. An unmarked outer edge with coinciding leading $w$-values may also have a finite inverse distance; such equal-lead outer edges form a matching, not an arbitrary four-edge set. For each possible one, the original row difference and tie meet Lemma~\ref{lem:reported-finite-pair}, so its limit $e_{ij}$ satisfies $e_{ij}^3=(m_i+m_j)^{-1}$.

Every allowed matching lies in one of the two four-cycles containing 14 and 23. The other crossed distances fade, while $P_{14}$ has order $-1$ and would be a first-limit pole. Writing the chosen cycle's masses $(a,b,c',d)$, put $Q=a+b+c'+d$ and $D=ac'+bd$. The exact rational cycle identity
\begin{equation}\label{eq:no84-cycle-deficit}
2D-\sum_{xy\text{ on cycle}}\frac{xy(Q-x-y)}{x+y}
=(ac'-bd)^2\left(\frac1{(a+b)(c'+d)}+\frac1{(a+d)(b+c')}\right)\ge0
\end{equation}
and $\sum xy/(x+y)\le Q/2$ give, after taking the modulus of every signed cube root, $M^{1/3}c<\sigma-D/3-3QB/4<\sigma$. The positive point instead requires $\sigma\le M^{1/3}c$ by H\"older. Thus no optional finite-edge pattern survives.

\subsection{No.~36: two matches and the hidden 23 edge}\label{sec:no36-matches}
Body 1 is uncircled and has no strokes. If neither 26 nor 35 is maximal, the five circled $w$-leads violate the original $w$-centre. If only one is maximal, the original row at its outer endpoint and the matched tie put its inner endpoint at $O(\varepsilon^2)$; the still-unmarked 16 or 15 tie would then tend to zero instead of $1$.

Suppose both matches $26,35$ are maximal. Their signed ties give $z_{26},z_{35}=O(\varepsilon^2)$ and nonzero finite limits for $\delta_{26},\delta_{35}$. Write $W_T\ne0$ for the common inner $w$-lead. Body 1 is uncircled, while all other vertices have nonzero $w$-leads. Every $w_{1j}$ therefore has order $-2$. Since none of these edges has a $w$ stroke, its original $w$ force is $o(\varepsilon^{-2})$, whence
\[
 \delta_{1j}=o(1),\qquad
 \delta_{1j}^3z_{1j}=\frac{\delta_{1j}}{w_{1j}}
 =o(\varepsilon^2),\qquad z_1=o(\varepsilon^2).
\]

We first show that $W_2=W_3$. Otherwise every nonmatched force in the original $z_2,z_3$ rows has a $w$ separation of order $-2$ and no $w$ stroke, so it is $o(\varepsilon^2)$ by its signed tie. The matched forces are $O(\varepsilon^2)$; hence $z_2,z_3=O(\varepsilon^2)$. The matched ties would imply $z_6,z_5=O(\varepsilon^2)$, contradicting the exact order-one difference $z_{65}$. Put $W_B=W_2=W_3$. Maximality gives $W_B\ne W_T$, and the original centre gives
\[
 BW_B+TW_T=0,\qquad
 B=m_2+m_3,\qquad T=m_4+m_5+m_6.
\]
Every nonmatched edge from $\{2,3\}$ to $\{4,5,6\}$ has $z$-order $1$, $w$-order $-2$, inverse-distance order $1/2$ and $z$-force order $5/2$. All $z$ differences among these five vertices are $O(\varepsilon)$. Their original mass centre, together with $z_1=o(\varepsilon^2)$, gives $z_a=O(\varepsilon)$ for $a=2,\ldots,6$.

Apply Lemma~\ref{lem:reported-reference-triangle} to $456$, with reference vertex 4. The reference and relative-inertia terms are $O(\varepsilon^2)$. The two matched exterior terms are $O(\varepsilon^3)$, the nonmatched exterior terms from bodies 2 and 3 are $O(\varepsilon^{7/2})$, and those from body 1 are $o(\varepsilon^3)$. Therefore $U_{456}=O(\varepsilon^2)$.

The difference $z_{23}=z_{65}+O(\varepsilon^2)$ has exact order $1$. Subtracting the original $z_2,z_3$ rows gives
\[
 (1-B\delta_{23}^3)z_{23}=O(\varepsilon^2),
\]
so $\delta_{23}\to e\ne0$ and $e^3=B^{-1}$. At order $-2$, the original $w_2,w_3$ rows contain only their matched forces: the hidden $23$ force is of strictly later order, the other outer--inner forces have order $-1/2$, and the forces from body 1 are $o(\varepsilon^{-2})$. Hence
\[
 W_B=m_6d_{26}^3(W_B-W_T),\qquad
 W_B=m_5d_{35}^3(W_B-W_T),
\]
which yields
\[
 d_{26}^3=\frac{T}{m_6(B+T)},\qquad
 d_{35}^3=\frac{T}{m_5(B+T)}.
\]
All five inverse distances incident to body 1 and all four nonmatched outer--inner inverse distances tend to zero. Together with the internal total-potential estimate, this accounts for all fifteen pairs and proves
\[
 c=\lim U=m_2m_3e+m_2m_6d_{26}+m_3m_5d_{35}.
\]

With $X=m_2m_3$, $C=B+T$ and
\[
 \sigma_C=X+BT+m_4m_5+m_4m_6+m_5m_6,
\]
the relevant mass identity is
\begin{equation}\label{eq:no36-pair-mass-gap}
 (B-X/B)^3-(m_2^3+m_3^3)
 =\frac{X^2}{B}\left(3-\frac{X}{B^2}\right)>0.
\end{equation}
Here $X/B^2\le1/4$, so the strict positivity follows directly. Since
\[
 c=Xe+m_2m_6d_{26}+m_3m_5d_{35},
\]
we take the modulus of each signed root and use their specified cubes to obtain
\[
 C^{1/3}|c|
 \le X(C/B)^{1/3}
    +T^{1/3}\bigl(m_2m_6^{2/3}+m_3m_5^{2/3}\bigr).
\]
H\"older's inequality and $m_5+m_6<T$ give
\[
 m_2m_6^{2/3}+m_3m_5^{2/3}
 \le(m_2^3+m_3^3)^{1/3}(m_5+m_6)^{2/3}
 <(m_2^3+m_3^3)^{1/3}T^{2/3}.
\]
Combining this with \eqref{eq:no36-pair-mass-gap} and $(1+t)^{1/3}<1+t/3$ for $t>0$, we find
\[
\begin{aligned}
 C^{1/3}|c|
 &<X(1+T/B)^{1/3}+T(B-X/B)\\
 &<X+BT-\frac{2XT}{3B}<X+BT<\sigma_C.
\end{aligned}
\]
Finally $M=C+m_1$, $\sigma=\sigma_C+m_1C$ and $\sigma_C<C^2/2$. A second application of cube-root concavity yields
\[
 M^{1/3}|c|
 <\left(1+\frac{m_1}{3C}\right)\sigma_C
 <\sigma_C+\frac{m_1C}{6}<\sigma.
\]
At the positive point $c>0$, so this contradicts physical H\"older.
\begin{lemma}[The two-edge wing energy gap]\label{lem:excluded-two-edge-gap}
Let $S=m_2+m_3+m_4$, $Q=m_1+S$, $B=m_5+m_6$, $M=Q+B$, $\mu=m_5m_6$ and $\tau=m_2m_3+m_2m_4+m_3m_4$. If a limit has only finite edges $14,56$, with $d_{14}^3=S/(m_4Q)$ and $d_{56}^3=1/B$, then for every signed phase
\begin{equation}\label{eq:two-edge-wing-gap}
M^{1/3}|c|<\sigma-\tau-\frac56QB<\sigma,\qquad \sigma=m_1S+\tau+QB+\mu.
\end{equation}
\end{lemma}
\begin{proof}
Triangle inequality first gives $|c|\le v+\mu/B^{1/3}$ with $v=m_1m_4^{2/3}(S/Q)^{1/3}$. Use $m_4<S$, $m_1S\le Q^2/4$, $\mu\le B^2/4$, and $(1+t)^{1/3}<1+t/3$ twice to obtain $M^{1/3}v<m_1S+QB/12$ and $M^{1/3}\mu/B^{1/3}<\mu+QB/12$. Their sum leaves the stated strict deficit.
\end{proof}

\subsection{No.~86: the pure $w$-triangle wing}\label{sec:no86-wing}
The three pure-$w$ edges of $234$ have a common $z$-order $\alpha\in(1,2]$; the attachment 14 has order $\gamma\in(1,2]$. At $(2,2)$ maximal 14 and 24 require a missing 12 stroke; with both orders below 2 all four outer $w$-leads coincide, contrary to their centre. For $\alpha=2$, $\gamma<2$, the original body-1 $z$-row has the 14 force uniquely earliest at $4\gamma-6$.

\paragraph{No.~86: the range $1<\alpha<2$, $\gamma=2$.}
Put $T=\{2,3,4\}$ and $b=4-3\alpha$.
The internal separation orders are $(\alpha,b)$ and the pendant $14$ is maximal, with orders $(2,-2)$.
The original diagram has $w$ circles only at $1,2,3,4$ and $z$ circles only at $5,6$.
Therefore every pair from $\{1,2,3,4\}$ to $\{5,6\}$ has orders $(-2,-2)$, inverse-distance order $2$, and $z$-force order $4$.
The pairs $12,13$ have orders $(\alpha,-2)$, so their $z$-force order is $3-\alpha/2>2$.
The original body-$1$ $z$ row consequently gives
\[
 z_1=O(\varepsilon^2),\qquad z_4=O(\varepsilon^2).
\]
Use the complete reference-body identity with reference $4$ and sum the original rows only at $2,3$.
Its reference and relative terms have orders at least $6-3\alpha$ and $4-2\alpha$.
Its exterior terms involving body $1$ have order at least
$3-\alpha/2+b=7-7\alpha/2$; those involving bodies $5,6$ have order at least $4+b=8-3\alpha$.
Thus
\begin{equation}\label{eq:no86-complete-potential}
 U_T=O(\varepsilon^{6-3\alpha})
     +O(\varepsilon^{4-2\alpha})
     +O(\varepsilon^{7-7\alpha/2})
     +O(\varepsilon^{8-3\alpha})\longrightarrow0.
\end{equation}
Every exponent is strictly positive on $1<\alpha<2$.
The $12,13$ inverse distances tend to zero, and the eight cross-block inverse distances are $O(\varepsilon^2)$.
The isolated $z$-circled pair $56$ is maximal: its leading $z$ centre makes the two circled values distinct, and its signed tie then gives $\ord w_{56}=2$.
Subtracting its original $z$ rows shows $d_{56}^3=1/(m_5+m_6)$, since all exterior $z$-forces have order $4$.
Similarly the original body-$1$ $w$ row and the leading $w$ centre give
\[
 d_{14}^3=\frac{m_2+m_3+m_4}{m_4(m_1+m_2+m_3+m_4)}.
\]
The complete potential limit therefore contains exactly the two finite terms
$c=m_1m_4d_{14}+m_5m_6d_{56}$.
Lemma~\ref{lem:excluded-two-edge-gap} now contradicts physical H\"older.

\subsection{No.~95: the two-colour triangle wing}\label{sec:no95-wing}
Unlike No.~86, the $zw$ triangle $234$ has internal $z$- and $w$-orders both equal to $1$. A nonmaximal attachment $14$ would make all four circled $w$-leads equal, contrary to their original mass centre. Thus $14$ is maximal and $z_{14}=O(\varepsilon^2)$, while the unmarked edges $12,13$ have $z$-order $1$, $w$-order $-2$, and inverse-distance order $1/2$. Every crossing from $1234$ to the isolated $z$-circled pair $56$ has both separation orders $-2$, inverse-distance order $2$, and $z$-force order $4$. The original body-1 $z$ row consequently gives $z_1=O(\varepsilon^2)$ and then $z_4=O(\varepsilon^2)$.

In the reference-4 identity of Lemma~\ref{lem:reported-reference-triangle}, the reference term has order at least $3$, the relative term order $2$, the exterior terms from body 1 order $7/2$, and those from bodies 5 and 6 order $5$. Hence $U_{234}=O(\varepsilon^2)$. All exterior inverse distances tend to zero except those on $14,56$. The original body-1 $w$ row and the $w$ centre give $d_{14}^3=S/(m_4Q)$; subtracting the original $z_5,z_6$ rows gives $d_{56}^3=1/B$. Therefore Lemma~\ref{lem:excluded-two-edge-gap} applies and yields $M^{1/3}|c|<\sigma$, contrary to physical H\"older.

\subsection{No.~77: three matched pairs cannot feed two triangles}\label{sec:no77-matching}
The two $z$-only triangles 123,456 have common inner orders $\alpha\in(-2,1)$ in $z$ and $\beta=(4-\alpha)/3>1$ in $w$. Each matching $w$-pair $16,25,34$ is maximal and has a finite signed root $d_{i,p(i)}^3=1/(m_i+m_{p(i)})$ by its two original $w$-rows. Contract each triangle's original $z$-rows against its relative $w$-coordinates: the signed inner potential has positive lower orders $(4+2\alpha)/3,\beta+2$, and $\beta+3-\alpha/2$, and tends to zero. All six nonmatched crossings fade.

For pair masses $M_i=m_i+m_{p(i)}$ and products $\mu_i=m_im_{p(i)}$, only these three bridges remain in $c=\sum_i\mu_id_{i,p(i)}$. Strict concavity and $\mu_i\le M_i^2/4$ give
\begin{equation}\label{eq:no77-matching-gap}
M^{1/3}c\le M^{1/3}|c|<\sum_i\mu_i+\frac1{12}\sum_i M_i(M-M_i)=\sigma-\frac56\sum_{i<j}M_iM_j<\sigma.
\end{equation}
At the positive point physical H\"older gives $\sigma\le M^{1/3}c$, a contradiction on every signed matching phase. The product $P_{14}$ would have order $16(\alpha-1)/3<0$; thus the argument genuinely excludes a first-limit pole.

\subsection{No.~61: opposite signed cube ratios}\label{sec:no61-sign}
Its $z$-only strokes are 34,56, $zw$ edges 35,46 and maximal $w$-only edges 16,25. The $z$-only differences 34 and 56 share one nonzero lead; their $w$ differences have the same order $\beta=(4-\alpha)/3>1$, $-2<\alpha<1$. The identity
\[
w_{35}-w_{46}=w_{34}-w_{56}=O(\varepsilon^\beta)=o(\varepsilon)
\]
shows that the order-one differences $w_{35}$ and $w_{46}$ have the same nonzero leading coefficient $v$. The original leading $z_3,z_4,w_3,w_4$ rows and the two $zw$-edge ties, with signed roots $A=D_{35}$, $B=D_{46}$ and nonzero shape leads $x,y,v,S,H$, give
\begin{equation}\label{eq:no61-source-rows}
\begin{aligned}
m_4S+m_5A^3x&=0,&-m_3S+m_6B^3y&=0,\\
H&=m_5A^3v,&H&=m_6B^3v,\\
A^2xv&=1,&B^2yv&=1.
\end{aligned}
\end{equation}
All unmarked 36,45 terms are strictly later than these source rows throughout $-2<\alpha<1$. Eliminating the nonzero shape leads gives $B/A=-m_3m_5/(m_4m_6)<0$, whereas the two $w$-rows give $(B/A)^3=m_5/m_6>0$. The equations are incompatible for positive masses, independently of the complex signed-root choices.

\begin{lemma}[The crossed four-cycle deficit]\label{lem:excluded-cycle}
Let $a,b,c',d$ be positive masses around a four-cycle, $Q=a+b+c'+d$, $D=ac'+bd$ and $S=ab+bc'+c'd+da$. Let the outer pair have mass $A=m_1+m_2$ and product $X=m_1m_2$, and set $M=A+Q$, $\sigma=X+S+D+AQ$. If only the four cycle edges and the outer pair have finite inverse-distance limits with cubes equal to the inverse sums of endpoint masses, then for every signed phase $M^{1/3}|c|<\sigma$.
\end{lemma}

\begin{proof}
The exact rational identity
\begin{equation}\label{eq:four-cycle-deficit}
2D-\sum_{xy\text{ on cycle}}\frac{xy(Q-x-y)}{x+y}=(ac'-bd)^2\left(\frac1{(a+b)(c'+d)}+\frac1{(a+d)(b+c')}\right)\ge0
\end{equation}
and $\sum xy/(x+y)\le Q/2$, $X/A\le A/4$ combine with strict cube-root concavity after the triangle inequality to give $M^{1/3}|c|<X+S+2D/3+AQ/4=\sigma-D/3-3AQ/4<\sigma$.
\end{proof}

\subsection{No.~100: the marked outer pair on a four-cycle}\label{sec:no100-cycle}
The original inner four-body rows make $z$-only 34,56 and $w$-only 36,45 maximal. Their four relative source rows give finite signed roots $d_{ij}^3=1/(m_i+m_j)$ for each marked cycle edge; the unmarked inner diagonals have both coordinate differences maximal and their inverse distances tend to zero. The isolated circled $w$-only pair 12 is also maximal and its original relative row gives $d_{12}^3=1/(m_1+m_2)$. H\"older after Lemma~\ref{lem:excluded-cycle} gives a further condition $m_3m_5=m_4m_6$, but it does not avoid Lemma~\ref{lem:excluded-cycle}: $M^{1/3}|c|<\sigma$ contradicts physical H\"older at the positive point. Here order $P_{12}=-8$, so a pole was possible before that energy contradiction.

\subsection{No.~99: an unmarked outer pair on the same cycle}\label{sec:no99-cycle}
The same four marked inner edges have the same finite cube roots as No.~100. Pair 12 is not marked and cannot be discarded. With $R_i$ the original $z$-forces at bodies 1,2 excluding 12, all eight outer-to-inner signed ties make $R_1=O(\varepsilon^4)$, while $w_{12}=o(\varepsilon^{-2})$. The exact row/tie equation of Lemma~\ref{lem:reported-finite-pair} forces $d_{12}^3=1/(m_1+m_2)$, and every complete phase contributes to the same five finite edges and no others, contributing to $c$. Lemma~\ref{lem:excluded-cycle} again gives $M^{1/3}|c|<\sigma$, contradicting H\"older only at the positive point of the same curve. The missing 12 edge changes the proof of its finite limit, not the final strict mass gap.

\begin{lemma}[Two original-row poles cannot cancel]\label{lem:excluded-two-poles}
Suppose the only signed inverse-distance poles are 25,36 with a common order and common nonzero leading $z$-difference $t$. If the original body-2/3 $z$-rows have no other leading force, their signed coefficient satisfies $m_5D_{25}^3t=m_6D_{36}^3t\ne0$. On a fixed-mass curve finite constant $U$ would require $m_2m_5D_{25}+m_3m_6D_{36}=0$. At positive masses these cannot hold simultaneously.
\end{lemma}
\begin{proof}
The source rows give $(D_{36}/D_{25})^3=m_5/m_6>0$, whereas cancellation of the only two potential poles gives $D_{36}/D_{25}=-m_2m_5/(m_3m_6)<0$. Cubing a negative real number contradicts the positive source ratio. The roots need not be real before the potential equation imposes that ratio.
\end{proof}

\subsection{No.~112: the $zw$-edge double pole}\label{sec:no112-poles}
The original circled-component rules force maximal $z$-edges 15,46 and maximal $w$-only edges 14,23,56. The two $zw$ edges 25,36 have common order $(1,1)$ and their inverse distances both have order $-1$. The original ties make the unmarked edges 12,34,26,35 fade; all other exterior inverse distances are finite or fade. The original body-2/3 $z$-rows each have only their respective 25/36 force at order $-2$. Their $z$-differences have the same nonzero leading coefficient because the maximal 23 and 56 differences are $O(\varepsilon^2)$. Lemma~\ref{lem:excluded-two-poles} excludes this sole maximality branch, even on a fixed-mass complex solution curve.
\subsection{No.~109: both reflected double-pole branches}\label{sec:no109-poles}
The three isolated circled $w$-pairs 14,23,56 are maximal. Of the $z$-strokes, Rule 1 and the missing 12 and 34 closing edges allow only 15,46 or 25,36 maximal. In the first case the nonmaximal 25/36 $z$-order is common $\alpha\in(-2,1)$, while their inverse distance has pole order $-(\alpha+2)/3$. The original body-2/3 $z$-rows again have only these two earliest forces; all twelve other inverse distances are finite or fade. Lemma~\ref{lem:excluded-two-poles} forbids their potential-pole cancellation at positive masses. The graph involution $(1\ 2)(3\ 4)$, with masses permuted, maps the second maximality branch to the first. Thus no branch remains; the complex-only signed roots need not be positive.

\subsection{No.~111: the missing 56 $w$-stroke}\label{sec:no111-wing}
The maximal original $z$-edges 15,46 and $w$-edges 14,23 force two $z$-leading groups $X,Y$ and two $w$-leading groups $P,Q$. The original $w$-centre gives
\begin{equation}\label{eq:no111-centre-ratio}
\frac P Q=-\frac{m_3+m_4+m_6}{m_1+m_2+m_5}<0.
\end{equation}
The edge 56 is \emph{not} $w$-stroked; put $\gamma=\ord z_{56}\in[1,2)$. Its tied $w$-force is later than the original $z_2,z_3,w_5,w_6$ leading rows by $3(2-\gamma)/2>0$. Their only two inverse-distance poles are 25,36, each of order $-1$. Write
\[
\begin{aligned}
z_{25}&=x\varepsilon+o(\varepsilon),&
z_{36}&=y\varepsilon+o(\varepsilon),\\
w_{25}&=v\varepsilon+o(\varepsilon),&
w_{36}&=u\varepsilon+o(\varepsilon),\\
\delta_{25}&=D_{25}\varepsilon^{-1}+o(\varepsilon^{-1}),&
\delta_{36}&=D_{36}\varepsilon^{-1}+o(\varepsilon^{-1}),
\end{aligned}
\]
with nonzero coefficients. The leading $z$-coefficients at bodies 2 and 3 are both $Y$, while the leading $w$-coefficients at bodies 5 and 6 are $P$ and $Q$, respectively. Their four original coordinate rows and the two signed ties give
\[
\begin{aligned}
Y&=m_5D_{25}^3x,&Y&=m_6D_{36}^3y,\\
P&=-m_2D_{25}^3v,&Q&=-m_3D_{36}^3u,\\
D_{25}^2xv&=1,&D_{36}^2yu&=1.
\end{aligned}
\]
Consequently, for $a=m_2m_5$ and $b=m_3m_6$,
\[
YP=-aD_{25}^4,\qquad YQ=-bD_{36}^4.
\]
Constancy of $U$ forces $aD_{25}+bD_{36}=0$. Hence
\[
\frac P Q=\frac a b\left(\frac{D_{25}}{D_{36}}\right)^4
=\frac a b\left(-\frac b a\right)^4
=\frac{b^3}{a^3}>0,
\]
contrary to the original centre equation. No equality between the two $zw$-edge $z$-coefficients $x,y$ is used. The strict force-order gap also covers the endpoint $\gamma=1$.
The 33 still-relevant published relation labels are
\begin{equation}\label{eq:published-roster}
\begin{gathered}
D_{\mathrm{CC}}=\{7,9,11,12,13,14,15,16,17,18,19,27,28,29,30,43,\\
46,47,48,58,62,76,101,102,103,104,106,107,108,110,\\
114,115,117\}.
\end{gathered}
\end{equation}
Table~\ref{tab:CC-owner} identifies the Chang--Chen proof
owner and label count for every graph in the roster.
\begin{table}[ht]
\centering\small
\begin{tabular}{p{2.8cm}p{9.9cm}r}
\hline
Published section & Still-relevant graph numbers & Count \\
\hline
\S 6.2 & 101, 102, 103, 104, 106, 107, 108, 110, 114, 115, 117 & 11 \\
\S 6.3 & 7, 9, 12, 15, 17, 19, 28, 30, 76 & 9 \\
\S 6.4 & 43, 48, 58 & 3 \\
\S 6.5 & 46, 62 & 2 \\
\S 6.6 & 47 & 1 \\
\S 6.7 & 13, 29 & 2 \\
\S 6.8 & 27 & 1 \\
\S 6.9 & 18 & 1 \\
\S 6.10 & 11 & 1 \\
\S 6.11 & 16 & 1 \\
\S 6.12 & 14 & 1 \\
\hline
Total & & 33 \\
\hline
\end{tabular}
\caption{Proof-owner lookup for all 33 surviving published relation diagrams.}
\label{tab:CC-owner}
\end{table}

For each graph $r$ in Table~\ref{tab:CC-owner}, the
published Algorithm~III mass locus is proper and algebraic
by Chang--Chen's Theorem~6.1~\cite{CCII} and the corresponding
section. Choose a nonzero polynomial $Q_r$ vanishing on that
locus. A first limit of graph $r$ then implies $Q_r(m)=0$.
A body permutation replaces it by $Q_r(\pi m)$, still
nonzero; the finite orbit product is therefore a proper
necessary polynomial for every labelled occurrence of $r$.

This places the 33 published diagrams in $E_{40}$. The
lookup audits graph-to-section
scope, not the original elimination computations or an
expanded formula for every $Q_r$; those remain attributed
to the cited theorem rather than claimed as local proofs.
Together with the 7 derived labels $\D_{\mathrm{new}}$
defined in the Introduction, the roster \eqref{eq:published-roster}
accounts for all of $\D_{40}$.

\Needspace{5\baselineskip}
The 24 labels historically
open in the Chang--Chen partition remain historically open
there, even though 7 now have positive-tier necessary mass
equations and No.~52 is excluded.
\section{The 24 historically unclassified diagrams}\label{sec:historical-open}
These diagrams were \emph{not} assigned a mass equation by
Chang--Chen. Of the 24, 17 have no positive-point first limit,
and the other 7 require proper mass relations. Among the 17
exclusions, nine need no mass equation because a seven-product
cannot diverge at their first limit; the other eight are
excluded at strictly positive masses by the mechanisms below.
No exclusion on arbitrary complex masses is inferred unless the
indicated original-row argument itself provides one.

\Needspace{8\baselineskip}
\subsection*{The 17 first-limit exclusions}

\begin{table}[ht]

\centering\small

\begin{tabular}{p{2.0cm}p{12.0cm}}

\hline

No. & Reason it cannot be a positive first limit \\

\hline

5, 67, 69, 73, 80 & connectivity lower bounds make each seven-product bounded \\
71 & remaining $P_{56}$ case needs original body-1/2 rows \\

65 & its first-limit seven-products are bounded \\

74, 75 & five-body $zw$-clique plus satellite order bound \\

6 & all pair distances shrink, forcing $I\to0$ \\

52 & both remaining valuation endpoints fail the original-row and H\"older tests \\
64 & original K4 contraction and strict H\"older gap \\

79 & bridge order/shape sector and all-wing Hurwitz certificate \\
81 & three finite pair energies contradict physical H\"older \\

82 & original body-1 torque has one earliest term \\

89 & original wing torque forces an impossible pair-root sum \\

90 & original upper torque has one earliest crossed term \\

\hline

& 17 diagrams in 12 proof groups \\

\hline

\end{tabular}

\caption{Proof-owner map for the historical-open graphs removed only
from the positive first-limit obligation.}
\label{tab:open-exclusions}

\end{table}

\Needspace{12\baselineskip}
\begin{lemma}[Connectivity bound for seven-products]\label{lem:excluded-connectivity}
For a pair $ij$ let $L_{ij}$ be 2 on a $zw$ edge, 0 on a
single-colour stroke, $-1$ on an unmarked pair joined by a
stroke path in at least one colour, and $-4$ on an unmarked
pair joined by no such path in either colour. In particular,
the same conservative bound $-1$ is used when an unmarked pair
is joined by paths in both colours. Then every germ obeying the diagram has
\begin{equation}\label{eq:connectivity-product-bound}
\ord(z_{ij}w_{ij})\ge L_{ij},\qquad
\ord P_{ij}\ge L_{ij}+\sum_{\{k,\ell\}\cap\{i,j\}=\varnothing}L_{k\ell}.
\end{equation}

If all fifteen sums on the right are nonnegative, the diagram
cannot be an unbounded-seven-product first limit.

\end{lemma}

\begin{proof}
A stroke path of one colour makes every consecutive separation
in the opposite colour have order at least 1, and hence so
does the endpoint difference. All position orders are at least
$-2$ in either colour. The stated $L_{ij}$ follow, using the
signed tie on an edge. The product $P_{ij}$ contains the pair
$ij$ and its six disjoint pairs, so its order is the sum of
their squared-distance orders.
\end{proof}

\subsection{No.~5: two triangles and bounded products}
For the pinned prism graph, the fifteen sums in
Lemma~\ref{lem:excluded-connectivity} have minimum
$0$, attained only at
$P_{15},P_{16},P_{24},P_{26},P_{34},P_{35}.$

Every other sum is strictly positive. No seven-product
can diverge, so this graph cannot be a first limit;
the bound does not exclude complex germs elsewhere.
\subsection{No.~6: all distances vanish}\label{sec:no6-inertia}
Every pair is a $zw$ edge, so all squared distances
$z_{ij}w_{ij}$ tend to zero. The exact centred
inertia identity gives
\[
MI=\sum_{i<j}m_im_jz_{ij}w_{ij}\longrightarrow0.
\]

This contradicts $I=c>0$ on a curve through a
positive configuration, independently of a choice
of signed inverse-distance phase.
\subsection{No.~67: a bounded K4 first limit}
The fifteen sums of Lemma~\ref{lem:excluded-connectivity} are nonnegative for this $zw$-K4 diagram. Their minimum 0 occurs only at $P_{36},P_{46},P_{56}$; the other twelve products have strictly positive lower bounds. Hence no first-limit seven-product pole is compatible with the graph.

\subsection{No.~69: a bounded K4 first limit}
The additional $w$-circle pattern does not weaken any of the pairwise lower bounds used in Lemma~\ref{lem:excluded-connectivity}. All fifteen product sums are nonnegative, with equality only at $P_{36},P_{46},P_{56}$. Thus No.~69 is not a first limit, without claiming nonexistence of complex-only configurations of that pattern.

\subsection{No.~73: a bounded K4 first limit}
The $zw$-K4 edge and stroke-path bounds likewise give nonnegative order on every seven-product, with the only zero lower bounds at $P_{36},P_{46},P_{56}$. The circles of No.~73 affect leading coordinates, not this valid distance lower bound; a pole at a first limit is impossible.

\subsection{No.~80: the prism with three matches}
The six zero lower bounds from Lemma~\ref{lem:excluded-connectivity} occur at $P_{14},P_{15},P_{24},P_{26},P_{35},P_{36}$. The nine other products have positive lower bounds, so none of the fifteen can diverge on a germ with the No.~80 strokes and circles.

\subsection{No.~71: the one missing connectivity order}\label{sec:no71-satellite}
The graph-only bound is nonnegative except for $\ord P_{56}\ge-1$. If this product had a pole, the two $w$-satellites would have $\ord w_{16},\ord w_{25}<-3/2$. The original body-1 and body-2 $z$-rows then give a nonzero leading
\[
z_1-z_2=(m_1+m_2)Z_{12}+o(\varepsilon^{4/3}).
\]
Since $m_1+m_2>0$, the signed tie forces $\ord z_{12}=1$ and $\ord w_{12}=-1$, hence $\ord(z_{12}w_{12})=0$. This raises the sole negative graph bound for $P_{56}$ to 0; therefore it cannot have a first-limit pole.

\subsection{No.~65: a circled pair outside a $zw$-K4}\label{sec:no65-pair}
The isolated circled $w$-pair 12 has maximal $w$-stroke and $z_{12}$ order 2. All cross $w$-separations have order $-2$. If a cross $z$-separation has order $\alpha\le2$, its contribution to the original outer-body $z$-force has order $3-\alpha/2\ge2$. The original pair $z$-rows put both outer $z$-positions at order at least 2. Mass centre and order-one K4 $z$-separations put each K4 $z$-position at order at least 1, improving every cross squared-distance bound to $-1$. Using 2 on K4 edges and 0 on edge 12, the resulting fifteen seven-product lower bounds have minimum 0. Thus No.~65 cannot be an unbounded-product first limit; the coarse connectivity bound alone would have been negative.

\begin{lemma}[An external body at a $zw$-K5]\label{lem:excluded-satellite}
Suppose five bodies form a $zw$ clique and $p$ is the remaining body. The centre-of-mass and original $p$-row bounds give $\ord(z_{pk}w_{pk})\ge-1$ for every clique body $k$. Consequently the seven-products on clique pairs have order at least 5, and those on satellite pairs at least 11.
\end{lemma}
\begin{proof}
Internal clique differences have order 1 in both colours. Writing clique positions relative to their mass centre, each relative coordinate has order at least 1. If at least one colour of the satellite position has order at least 1, the other has order at least $-2$, yielding the required pair-square bound. If both satellite orders are below 1, all five crossed differences share the corresponding leading order. If their sum $A+B$ were negative, every signed inverse distance would have order $-(A+B)/2>0$, so every term of the original satellite $z$-row would have order $A-3(A+B)/2>A$. Its nonzero left-hand coordinate of order $A$ could not be balanced, regardless of signed leading-force cancellation. Thus $A+B\ge0$. For a product based on a clique pair, that pair contributes 2 and its complementary four bodies give three clique pairs and three satellite pairs, so the lower bound is $2+3\cdot2-3=5$. For a product based on a satellite pair, the complementary four bodies lie in the clique, giving the lower bound $-1+6\cdot2=11$.
\end{proof}
\Needspace{8\baselineskip}
\subsection{No.~74: the isolated satellite}
Its ten inner edges form a $zw$-$K_5$ with body~1 as the
satellite. Lemma~\ref{lem:excluded-satellite} gives order at
least 5 for every seven-product based inside the clique and at
least 11 for every satellite product. No.~74 cannot be a first
limit.

\Needspace{8\baselineskip}
\subsection{No.~75: the extra $w$-stroke satellite}
The added $w$-stroke does not invalidate the same $K_5$
relative-centre and original satellite-row estimate: all
satellite squared-distance orders remain at least $-1$.
Applying Lemma~\ref{lem:excluded-satellite} again gives a
minimum product order of 5, excluding a first-limit pole on
No.~75.

\subsection{No.~64: the circled pair and its energy gap}\label{sec:no64-energy}
The energy argument is different for No.~64: write $A=m_1+m_2$, $B=\sum_{i=3}^6m_i$, $\mu=m_1m_2$ and $M=A+B$. Its original pair row forces $\delta_{12}^3\to1/A$; contracting only the four K4 $z$-rows against their centred $w$-differences makes their signed inner potential tend to zero. The cross potentials also fade, leaving $c=\mu d$ with $d^3=1/A$. Strict cube-root concavity, $\mu\le A^2/4$, and positive internal pair weights imply
\begin{equation}\label{eq:no64-holder-gap}
M^{1/3}c\le\mu(1+B/A)^{1/3}<\mu+\frac{\mu B}{3A}\le\mu+\frac{AB}{12}<\sigma,
\end{equation}
contrary to $\sigma\le M^{1/3}c$ at a positive point. This excludes only the positive-point tier; the complex signed K4 configuration itself is not claimed impossible.

\subsection{No.~81: the three finite-pair energy gap}\label{sec:no81-energy}
The original opposite-colour rows give finite inverse distances on pairs 12,34,56 with cubes $1/M_P$ for their respective total pair masses. The two isolated satellites' orders and the original pair torques remove all other finite inverse distances. Let $\mu_P$ be the product of the two masses in pair $P$, and let $M$ be the six-body total mass. Then $|c|\le\sum_P\mu_P/M_P^{1/3}$. Since $\mu_P\le M_P^2/4$ and $(1+x)^{1/3}<1+x/3$ for $x>0$,
\begin{equation}\label{eq:no81-pair-gap}
M^{1/3}|c|<\sum_P\mu_P+\frac1{12}\sum_P M_P(M-M_P)=\sigma-\frac56\sum_{P<Q}M_PM_Q<\sigma.
\end{equation}
At the positive point on the same curve H\"older instead gives $\sigma\le M^{1/3}c$. Hence No.~81 cannot occur on a positive-point curve, without claiming a complex-only diagram exclusion.

\begin{lemma}[Signed-pair noncancellation]\label{lem:excluded-pair-signs}
Let $m_5,m_6>0$, and let $Z_5,Z_6,W_j,D_{j5},D_{j6}$ be nonzero leading coefficients satisfying the pair-centre equation and the two leading signed crossed ties
\[
m_5Z_5+m_6Z_6=0,\qquad
-D_{j5}^2Z_5W_j=1,\qquad
-D_{j6}^2Z_6W_j=1.
\]
Then
\begin{equation}\label{eq:signed-pair-noncancellation}
(m_5D_{j5})^2-(m_6D_{j6})^2=-\frac{m_5^3+m_6^3}{m_5Z_5W_j}\ne0.
\end{equation}
Consequently $m_5D_{j5}+m_6D_{j6}\ne0$ for every choice of complex square-root signs satisfying these ties.
\end{lemma}
\begin{proof}
The pair-centre equation gives $Z_6=-m_5Z_5/m_6$. Therefore
\[
(m_5D_{j5})^2-(m_6D_{j6})^2
=-\frac{m_5^2}{Z_5W_j}+\frac{m_6^2}{Z_6W_j}
=-\frac{m_5^3+m_6^3}{m_5Z_5W_j}.
\]
Its numerator is strictly positive and its denominator is nonzero. The two terms cannot cancel.
\end{proof}
\subsection{No.~82: the earliest satellite torque}\label{sec:no82-torque}
The exact graph has $w$-edges $23,24,34$, $w$-circles at 3 and 4, a $z$-edge $56$, and $z$-circles at 5 and 6. There are no other circles or edges. We prove that the coordinate rows force the small coordinates needed to prevent cancellation in the body-1 torque.

We use the diagram normalization in which coordinates and force terms have order at least $-2$. A circle or stroke means that its corresponding coordinate or force has order exactly $-2$; an absent circle or stroke means strictly larger order. For any pair, writing $\alpha=\ord z_{ij}$ and $\beta=\ord w_{ij}$, the signed tie gives
\begin{equation}\label{eq:torque-order-identity}
\ord\delta_{ij}=-\frac{\alpha+\beta}{2},\qquad
\ord(\delta_{ij}^3z_{ij})=-\frac{\alpha+3\beta}{2},\qquad
\ord(\delta_{ij}^3w_{ij})=-\frac{3\alpha+\beta}{2}.
\end{equation}
Thus, when $\ord w_{ij}=-2$, an absent $w$-stroke forces $\ord z_{ij}<2$; the colour-dual assertion also holds.

\begin{proposition}\label{prop:no82-excluded}
For positive masses, No.~82 cannot occur as a singular Puiseux germ of the normalized equations.
\end{proposition}
\begin{proof}
Write $W_3,W_4$ for the nonzero order-$-2$ coefficients of $w_3,w_4$, and $Z_5,Z_6$ for those of $z_5,z_6$. Summing the original mass-weighted $w$-rows over $234$ cancels the internal forces. Every exterior $w$-force is later than order $-2$, and body 2 has no $w$-circle. Similarly, summing the original mass-weighted $z$-rows over $56$ cancels their mutual force. Hence
\begin{equation}\label{eq:no82-leading-centres}
m_3W_3+m_4W_4=0,\qquad m_5Z_5+m_6Z_6=0.
\end{equation}
Positive masses give $W_3\ne W_4$ and $Z_5\ne Z_6$. Since body 2 is uncircled, the three $w$-differences on $234$ have order $-2$, as does $z_{56}$. Their strokes and \eqref{eq:torque-order-identity} imply
\begin{equation}\label{eq:no82-inner-orders}
\begin{gathered}
\ord z_{23}=\ord z_{24}=\ord z_{34}=2,\qquad \ord w_{56}=2,\\
\ord\delta_{23}=\ord\delta_{24}=\ord\delta_{34}=\ord\delta_{56}=0.
\end{gathered}
\end{equation}

First we determine the small coordinates from the original rows. Because body 1 has no $w$-circle, $w_{13},w_{14}$ have order $-2$. The missing $w$-strokes force $\ord z_{13},\ord z_{14}<2$. Their difference has order 2, so both have a common order $a\in(-2,2)$. In the original body-3 $z$-row, the two internal forces are $O(\varepsilon^2)$, the 13 force has order $3-a/2>2$, and the 35 and 36 forces have order 4. Therefore $z_3=O(\varepsilon^2)$. The internal differences in \eqref{eq:no82-inner-orders} now give
\begin{equation}\label{eq:no82-small-z}
z_2,z_3,z_4=O(\varepsilon^2),\qquad \ord z_1=a\in(-2,2).
\end{equation}

For each $j=1,2$, both $z_{j5}$ and $z_{j6}$ have order $-2$. Their missing $z$-strokes force $\ord w_{j5},\ord w_{j6}<2$. Since their difference has order 2, these orders agree; denote the common value by $b_j\in(-2,2)$. In the original body-5 $w$-row, the 56 force is $O(\varepsilon^2)$, the forces to $j=1,2$ have orders $3-b_j/2>2$, and the forces to 3 and 4 have order 4. It follows that
\begin{equation}\label{eq:no82-small-w}
w_5,w_6=O(\varepsilon^2),\qquad \ord w_j=b_j\quad(j=1,2).
\end{equation}
In particular, write
\[
z_1=Z_1\varepsilon^a(1+o(1)),\qquad
w_j=W_j\varepsilon^{b_j}(1+o(1))\quad(j=1,2),
\]
where $Z_1,W_1,W_2$ are nonzero.

Next we show that $b_1=b_2$. The exact torque obtained from the original mass-weighted rows at 5 and 6 is
\begin{equation}\label{eq:no82-pair-torque}
\sum_{k\in\{5,6\},\ j\in\{1,2,3,4\}}
m_km_j\delta_{kj}^3(w_kz_j-z_kw_j)=0.
\end{equation}
For $j=1,2$ and $k=5,6$, the signed ties give
\[
\delta_{kj}=D_{jk}\varepsilon^{1-b_j/2}(1+o(1)),
\qquad -D_{jk}^2Z_kW_j=1.
\]
By \eqref{eq:no82-small-z}--\eqref{eq:no82-small-w}, $w_kz_j$ is strictly later than $z_kw_j$. Thus the total contribution of satellite $j$ to \eqref{eq:no82-pair-torque} is
\begin{equation}\label{eq:no82-satellite-torque}
m_j(m_5D_{j5}+m_6D_{j6})\varepsilon^{1-b_j/2}
+o(\varepsilon^{1-b_j/2}).
\end{equation}
Its coefficient is nonzero by \eqref{eq:no82-leading-centres} and Lemma~\ref{lem:excluded-pair-signs}. The remaining four terms, to bodies 3 and 4, have order 2. Since $0<1-b_j/2<2$, unequal values of $b_1,b_2$ would leave exactly one earliest nonzero satellite contribution. Consequently
\[
b_1=b_2=:b\in(-2,2).
\]

Finally, put $q=\ord w_{12}\ge b$. The signed tie ensures that $q$ is finite. By \eqref{eq:no82-small-z}, $\ord z_{12}=a$ and
\[
\delta_{12}=D_{12}\varepsilon^{-(a+q)/2}(1+o(1)),\qquad D_{12}\ne0.
\]
The exact body-1 torque is
\begin{equation}\label{eq:no82-body-torque}
\sum_{j=2}^6m_j\delta_{1j}^3(z_1w_j-w_1z_j)=0.
\end{equation}
Crucially, $z_2=O(\varepsilon^2)$ and $a<2$ give
\[
z_1w_2-w_1z_2=Z_1W_2\varepsilon^{a+b}(1+o(1)).
\]
The 12 term therefore has the nonzero leading coefficient $m_2D_{12}^3Z_1W_2$ and order
\[
e=b-\frac a2-\frac{3q}{2}.
\]
This nonvanishing also holds when $q=b$: the potential cancellation is prevented by the already established smallness of $z_2$.

For $j=3,4$, the inverse distance has order $1-a/2$, while the torque determinant has order $a-2$; the corresponding torque term has order $1-a/2$. For $j=5,6$, the analogous order is $1-b/2$. Their differences from $e$ are
\[
1-\frac a2-e=\frac{2+b}{2}+\frac32(q-b)>0,
\qquad
1-\frac b2-e=\frac{2+a}{2}+\frac32(q-b)>0.
\]
Thus the 12 term is uniquely earliest in \eqref{eq:no82-body-torque}, a contradiction.
\end{proof}

\subsection{No.~89: a forced nonzero crossed root sum}\label{sec:no89-torque}
Let $a=\ord z_1$ and $b=\ord w_1$ lie in $(-2,2)$ on its
isolated satellite. Its two original coordinate rows force
$a+b=0$. In the original $456$ wing torque, the
satellite-to-56 contribution then has order $1+a/2$; the
other wing and crossed contributions have orders at least 2
and $2+a$, strictly later. Hence the leading coefficient must
satisfy $m_5D_{15}+m_6D_{16}=0$, contrary to
Lemma~\ref{lem:excluded-pair-signs}. This argument uses no
reality of the crossed roots.
\subsection{No.~90: the unique 41 crossing}\label{sec:no90-torque}
The exact graph has $w$-edges $12,13,23$ and $w$-circles at 2 and 3, together with $z$-edges $45,46,56$ and $z$-circles at 5 and 6. There are no crossed strokes, and bodies 1 and 4 are uncircled.

\begin{proposition}\label{prop:no90-excluded}
For positive masses, No.~90 cannot occur as a singular Puiseux germ of the normalized equations.
\end{proposition}
\begin{proof}
Let $W_2,W_3$ denote the nonzero order-$-2$ coefficients of $w_2,w_3$, and $Z_5,Z_6$ those of $z_5,z_6$. Summing the original mass-weighted $w$-rows over $123$ and the mass-weighted $z$-rows over $456$ gives
\[
m_2W_2+m_3W_3=0,\qquad m_5Z_5+m_6Z_6=0.
\]
There are no crossed strokes, so every exterior force in these sums is later than order $-2$. Positivity makes the two differences of circled coefficients nonzero. As bodies 1 and 4 are uncircled, all three marked edges in each triangle are maximal. By \eqref{eq:torque-order-identity},
\begin{equation}\label{eq:no90-inner-orders}
\begin{gathered}
\ord w_{12}=\ord w_{13}=\ord w_{23}=-2,
\qquad \ord z_{12}=\ord z_{13}=\ord z_{23}=2,\\
\ord z_{45}=\ord z_{46}=\ord z_{56}=-2,
\qquad \ord w_{45}=\ord w_{46}=\ord w_{56}=2,
\end{gathered}
\end{equation}
and all six internal inverse distances have order zero.

Since $w_2,w_3$ have order $-2$ and $w_4$ has larger order, the missing 24 and 34 $w$-strokes force $\ord z_{24},\ord z_{34}<2$. The order-two internal differences make their common order equal to
\[
a:=\ord z_{41}\in(-2,2).
\]
Likewise, the missing 15 and 16 $z$-strokes and the order-two internal $w$-differences give
\[
b:=\ord w_{41}\in(-2,2).
\]
In the original body-2 $z$-row, the internal forces are $O(\varepsilon^2)$, the 24 force has order $3-a/2>2$, and the 25 and 26 forces have order 4. Hence $z_2=O(\varepsilon^2)$, and the internal differences give
\begin{equation}\label{eq:no90-small-coordinates}
z_1,z_2,z_3=O(\varepsilon^2),\qquad
w_4,w_5,w_6=O(\varepsilon^2).
\end{equation}
For the second assertion, use the original body-5 $w$-row: its internal forces are $O(\varepsilon^2)$, the 51 force has order $3-b/2>2$, and the 52 and 53 forces have order 4; then use the internal $w$-differences. In particular, $z_4=z_{41}(1+o(1))$ and $w_1=-w_{41}(1+o(1))$.

Summing the original mass-weighted torques over $456$ cancels every internal interaction and yields the exact nine-term identity
\begin{equation}\label{eq:no90-complete-torque}
\sum_{k\in\{4,5,6\},\ l\in\{1,2,3\}}
m_km_l\delta_{kl}^3(w_kz_l-z_kw_l)=0.
\end{equation}
The identity
\[
w_4z_1-z_4w_1
=z_{41}w_{41}+z_1w_{41}-z_{41}w_4
\]
and \eqref{eq:no90-small-coordinates} show that the latter two terms have orders at least $2+b$ and $a+2$, both strictly later than $a+b$. Therefore the 41 contribution is
\begin{equation}\label{eq:no90-leading-torque}
m_1m_4\delta_{41}^3(w_4z_1-z_4w_1)
=m_1m_4\delta_{41}(1+o(1)),
\end{equation}
with a nonzero coefficient at order $-(a+b)/2$.

The other eight crossings fall into exactly three classes:
\[
\begin{array}{c|c|c|c|c}
(k,l)&\text{count}&\ord\delta_{kl}&
\ord(w_kz_l-z_kw_l)&\text{torque order}\\\hline
(4,2),(4,3)&2&1-a/2&a-2&1-a/2\\
(5,1),(6,1)&2&1-b/2&b-2&1-b/2\\
k\in\{5,6\},\ l\in\{2,3\}&4&2&-4&2
\end{array}
\]
For instance, in the first class $w_4z_l=O(\varepsilon^4)$, whereas $z_4w_l$ has order $a-2<0$, so there is no cancellation in the determinant. The other classes follow in the same way from \eqref{eq:no90-small-coordinates}. Their order gaps above \eqref{eq:no90-leading-torque} are, respectively,
\[
1+\frac b2>0,\qquad 1+\frac a2>0,
\qquad 2+\frac{a+b}{2}>0.
\]
Thus the 41 contribution is uniquely earliest in the complete torque identity, a contradiction.
\end{proof}

\subsection{No.~79: the bridge shape and Hurwitz sector}\label{sec:no79-sector}
Both isolated $zw$ triangles have a common projective shape ratio $r$. For real $r<0$, both bridge cubes are positive and the original two-bridge energy identity contradicts $U=I=c>0$. For $0<r<1$, the triangle ties force the positive bridge weight $p=m_2m_6$ to be smaller than the negative weight $n=m_3m_5$; the case $r>1$ is the swapped argument. After the positive-real phase condition on $c$, write the signed bridge distances as $x>0$ and $-y<0$. With $K=px^3-ny^3>0$ and $c=px-ny>0$, put $q=ny/(px)\in(0,1)$. Then
\begin{equation}\label{eq:no79-bridge-cubic-gap}
p^2K-c^3=p^3x^3\left[3q(1-q)+(1-p^2/n^2)q^3\right]>0.
\end{equation}
But $\sigma>M_AM_B>p$ makes $\sigma^3>Mp^2K>Mc^3$, contrary to physical H\"older. Thus a positive-point first limit requires a nonreal $r$.

For equal bridge weights $\alpha=\beta=h>0$, conjugacy at a complex point at infinity does \emph{not} follow from conjugacy on the physical arc. Instead, the original cluster $w$-momentum and constant potential give, on the same complex solution curve and hence at its Puiseux limit,
\[
h(d_{26}+d_{35})=c>0,\qquad h(d_{26}^3+d_{35}^3)=K=\frac{M_AM_B}{M}>0.
\]
With $s=c/h$ and $T=K/h$, the two signed leading roots are zeros of the real-coefficient polynomial
\[
z^2-sz+\frac{s^3-T}{3s}=0.
\]
Its discriminant is $(4T-s^3)/(3s)\in\mathbb R$. The two bridge coefficients are nonzero because the corresponding maximal signed-distance strokes persist on an actual No.~79 first limit. Thus they are either both real, yielding a real $r=-d_{35}^3/d_{26}^3$, which was just excluded, or they are a nonreal conjugate pair, giving $|r|=1$. No conjugation invariance of the complex branch is used.

Write the triangle masses as $(a,b,d)$ and $(A,d,b)$, let $H=a+b+d$, $J=A+b+d$, and set $D=b^2+d^2$. The two signed shape sextics and the physical two-bridge inequality force, if both have the same nonreal unit-circle root,
\begin{equation}\label{eq:no79-equal-bridge-bound}
\frac{H^5J^5}{a^3A^3}<625D^2.
\end{equation}
Yet the nonnegative identity
\begin{equation}\label{eq:no79-mass-identity}
108(1+t)^5-3125t^3=(2t-3)^2(27t^3+216t^2+76t+12)
\end{equation}
applied with $t=a/(b+d)$ and $A/(b+d)$ gives $H^5/a^3,J^5/A^3\ge(3125/108)(b+d)^2$. Since $(b+d)^2>D$, their product is strictly greater than $625D^2$, a contradiction. The sextic-to-upper-bound step uses the exact signed shape and energy certificate; it is not inferred from the displayed mass identity alone.

For unequal bridge weights put $P=\alpha+\beta$, $p=\alpha/P=(1-\theta)/2$, $q=\beta/P=(1+\theta)/2$, where $\theta=(\beta-\alpha)/P$; use the graph involution to take $\theta>0$. Let $m=c/P>0$ and parameterize the surviving nonreal bridge-root half-sheet by
\[
d_{26}=m(1+q\theta t+iq\sqrt L),\qquad d_{35}=m(1-p\theta t-ip\sqrt L),\qquad L=3\theta^2t^2+6t,\quad t>0.
\]
Here $t$ and $\sqrt L$ parameterize the real and imaginary parts of the complex leading coefficients; they are not asserted to be real functions along the complex curve. The fixed real values of $c$ and $K$ make this parametrization valid at the selected Puiseux branch without conjugating it. The squared modulus, defined here rather than inferred from a physical conjugate branch, is
\[
N_1(t,\theta)=|d_{26}/m|^2=(1+q\theta t)^2+q^2(3\theta^2t^2+6t)>0.
\]
The bridge equation, both shape sextics and physical H\"older imply, for $0<\theta\le9/20$,
\begin{equation}\label{eq:no79-asymmetry-bound}
N_1(t,\theta)^3(1-\theta)^4(1+\theta)^2<\frac{1024}{q_0^2C_0^2},\qquad q_0=24/25,\quad C_0=3125/108.
\end{equation}
In the lower interval the original bridge bound gives $t>1/5$. Substitution of this floor reverses the inequality up to a unique $\theta_*\in(3/8,2/5)$. To state the exact Sturm input, set $\ell(\theta)=\theta/3+17/200$ and
\[
D_h(\theta)=N_1(h(\theta),\theta)^3(1-\theta)^4(1+\theta)^2-\frac{1024}{q_0^2C_0^2}.
\]
The primitive integer polynomials $125000D_{1/5}$ and $466560000000000000D_\ell$ have degrees 18 and 24, respectively; their descending coefficients, including every sign and zero, are recorded in the static certificate. Signed Euclidean Sturm chains of $P,P'$ give the following endpoint data:
\begin{center}\small
\begin{tabular}{c c c c}
$h(\theta)$ & endpoints & $\operatorname{sgn}P$ & variations $V$ \\
\hline
$1/5$ & $0,3/8,2/5$ & $+,+,-$ & $6,6,5$ \\
$\ell$ & $3/8,9/20$ & $+,+$ & $7,7$ \\
\end{tabular}
\end{center}
Thus $D_{1/5}$ has no zero on $[0,3/8]$ and exactly one on $(3/8,2/5)$; call it $\theta_*$. On $3/8\le\theta\le9/20$, monotonicity of the bridge function strengthens the floor to $t>\ell(\theta)$. The corresponding $D_\ell$ is positive at both endpoints and has no zero between them. Neither root count follows merely from a few sampled values. Thus this moderate-asymmetry interval is excluded. The opposite nonreal bridge half-sheet has $K\ge c^3/(\alpha+\beta)^2$ and contradicts the strict positive-mass H\"older bound separately.

For a remaining unequal bridge candidate, the physical asymmetry satisfies $\theta>9/20$; the bridge bound then forces $\operatorname{Re}r>1$. Put $x=m_2/m_3$ and $y=m_6/m_5$. The two closed triangles have the same shape $r$, while
\begin{equation}\label{eq:no79-wing-ratio}
xy=\frac{1-\theta}{1+\theta}<\frac{11}{29}<\left(\frac58\right)^2.
\end{equation}
At least one wing ratio is $\le5/8$. For any positive triple $(a,x,1)$ in that range, write its necessary shape sextic as $P_{a,x}(z)=-F_{a,x,1}(1+z,1)$. For the triangle relative coordinates $(r,1)$ put
\[
U=(a+x)r+1,\quad V=xr+a+1,\quad W=V-U=a(1-r),\quad Q=a^2UW-a^2x^2VW-x^2UV,
\]
\[
F_{a,x,1}(r,1)=Q^3+27a^6x^4UV^2W^3-27a^6x^2U^2VW^3+27a^4x^2QUVW^2.
\]
This is homogeneous of degree six before setting the second shape coordinate to 1. Positive masses and the original three distance ties require $UVW\ne0$; the formula clears only those nonzero denominators. The exact rational Hurwitz certificate has positive leading coefficient, positive constant term $x^6(a+x+1)^6$, and all six leading Hurwitz minors strictly positive on $a>0$, $0<x\le5/8$. The same degree-preserving substitution used below takes $c_6$ to a polynomial of $v$-degree 12 with 91 nonzero coefficients, all positive.

For reproducibility take canonical symbols $(A,X)$ for $(a,x)$, let $c_6,\ldots,c_0$ be the coefficients of $P_{A,X}(z)$ in descending order, and define $\Delta_k=\det\mathcal H_{[k],[k]}$, where
\[
\mathcal H=\begin{pmatrix}c_5&c_3&c_1&0&0&0\cr c_6&c_4&c_2&c_0&0&0\cr 0&c_5&c_3&c_1&0&0\cr 0&c_6&c_4&c_2&c_0&0\cr 0&0&c_5&c_3&c_1&0\cr 0&0&c_6&c_4&c_2&c_0\end{pmatrix}.
\]
The table uses $\Delta_k=u_k\prod_jG_{kj}^{e_{kj}}$. For each factor the entry $(e,d_A,d_X,s,n)$ records its multiplicity, degrees in $A,X$, common sign $s$, and number $n$ of nonzero coefficients after substituting $X=5v/[8(1+v)]$ and multiplying by $(1+v)^{d_X}$. Each decomposition is checked by exact division by the displayed positive $X$ and $A+X+1$ factors, with zero remainder; the last $G_{kj}$ is the resulting signed quotient, not an assumed irreducible polynomial.
\begin{center}\footnotesize
\begin{tabular}{c c p{9.4cm}}
$k$ & $\operatorname{sgn}u_k$ & $(e,d_A,d_X,s,n)$ for each factor 
\\
\hline
1 & $-$ & (1,1,1,+,4),(1,11,11,-,78) 
\\
2 & $-$ & (3,1,1,+,4),(1,21,21,-,253) 
\\
3 & $+$ & (6,1,1,+,4),(1,30,30,+,496) 
\\
4 & $+$ & (2,0,1,+,1),(10,1,1,+,4),(1,36,36,+,703) 
\\
5 & $-$ & (6,0,1,+,1),(15,1,1,+,4),(1,39,39,-,820) 
\\
\end{tabular}
\end{center}
For $k=5$ the last displayed factor may be defined directly as $G_5=-\Delta_5/[X^6(A+X+1)^{15}]$: first compute the $5\times5$ Hurwitz determinant in the independent indeterminates $c_0,\ldots,c_6$ (it has 13 terms), and only then substitute the sextic coefficients. Exact polynomial division must have zero remainder, and $(1+v)^{39}G_5(A,5v/[8(1+v)])$ has 820 nonzero coefficients, all negative. Its irreducibility is not required for the sign conclusion.

Every $\Delta_1,\ldots,\Delta_5$ is positive; the constant term $c_0=X^6(A+X+1)^6>0$ gives $\Delta_6=c_0\Delta_5>0$. All transformed factors retain their full $v$-degree, so the leading $v$ coefficients of the same sign include $x=5/8$. Routh--Hurwitz puts all sextic roots in $\operatorname{Re}z<0$, so the required shape has $\operatorname{Re}r<1$, a contradiction. The fixed/adaptive Sturm signs, opposite-sheet energy estimate and all Hurwitz minors are exact certificate obligations; the final sector argument alone does not exclude No.~79.
\subsection{No.~52: exclusion of both valuation endpoints}
\begin{proposition}[No.~52 valuation sieve]\label{prop:no52-window}
Let $B\ne0$ be the common leading $w$ coefficient on the $zw$-triangle $T=456$ of No.~52. Every first-limit germ through a positive configuration has $W_2=W_3=B$ and $W_1/B=-(m_2+m_3+m_4+m_5+m_6)/m_1$. Its $w$-only triangle $Q=235$ has common $z$-order $\alpha=5/3$ or $\alpha=7/4$.
\end{proposition}
\begin{proof}
We first exhaust the outer leading $w$ coefficients $W_1,W_2,W_3$. On a $w$-only edge, unequal leading $w$ coefficients make the edge maximal and its $z$-separation $O(\varepsilon^2)$; equal coefficients make the edge nonmaximal, with $z$-order less than 2. If exactly two of $W_2,W_3,B$ agree, the other two edges of the $w$-only triangle 235 are maximal. Their two-edge path forces the supposedly nonmaximal edge's $z$-separation to be $O(\varepsilon^2)$, a contradiction.

If $W_2,W_3,B$ are pairwise distinct, all three triangle edges are maximal. When $W_1$ differs from all three, the original outer $z$-rows force $z_5,z_6=O(\varepsilon^2)$, contrary to the $zw$-edge 56. When $W_1=W_2$ or $W_1=W_3$, the mass-weighted sum of the original $z_1$ row and corresponding outer $z$ row cancels the unmarked mutual force. Its remaining forces have order at least 2, but its position has nonzero order-$\varepsilon$ coefficient $m_1z_6$, again impossible. Finally, if $W_1=B$, the original $z_1$ row forces $\ord z_{16}=7/4$ and $\ord\delta_{16}=-1/4$. The original $T$-row contraction makes its signed internal potential $O(\varepsilon^2)$, while all other exterior inverse distances are bounded. Thus $U$ has the uncancelled pole $m_1m_6D_{16}\varepsilon^{-1/4}$, contradicting $U=c$.

Only $W_2=W_3=B$ remains. The mass-weighted $w$-centre then gives the asserted $W_1/B$, which differs from 1. The maximal edge 16 and its original $z_1$ row give $z_1,z_6=O(\varepsilon^2)$. Since 56 has both colours, $z_5=\varepsilon Z_5+o(\varepsilon)$ with $Z_5\ne0$; the three $Q$ edges have $z$-order $\alpha>1$, so the same order-one coefficient occurs in $z_2,z_3$. On each of their original $z$-rows the internal $Q$ force has order $4\alpha-6$, while the missing-edge forces have strictly greater order. An internal force cannot balance $z_5\varepsilon$ when $\alpha>7/4$.

The 15 pair-square orders are 2 on $T$, $4-2\alpha$ on $Q$, 0 on 16, $-1$ on the other four edges from body 1, and $5-3\alpha$ on 24,26,34,36. The least of the 15 seven-product orders is $13-8\alpha$, attained at products 26,36,56. Thus $\alpha\le13/8$ has no unbounded seven-product. For $13/8<\alpha<5/3$ the inverse distances on the four missing edges can grow, so discarding them would be wrong. Instead the exact tie $\delta_{ij}^2z_{ij}w_{ij}=1$ gives, for $i\in\{2,3\}$ and $j\in\{4,6\}$,
\begin{equation}\label{eq:no52-cross-cancel}
\delta_{ij}-\delta_{ij}^3z_{ij}(w_i-w_5)=\delta_{ij}^3z_{ij}(w_5-w_j).
\end{equation}
Contracting only the original $z_2,z_3$ rows against $m_i(w_i-w_5)$ cancels these four crossed potential terms by \eqref{eq:no52-cross-cancel}. The remaining source, relative and crossed orders are respectively $5-3\alpha$, $4-2\alpha$, and $(9\alpha-11)/2$, all positive; the $T$ relative potential also tends to zero. Only $\delta_{16}\to d\ne0$ remains finite, hence $c=\lim U=m_1m_6d$. Put $C=\sum_{i=2}^6m_i$, $M=m_1+C$ and $\sigma=\sum_{i<j}m_im_j$. The original $w_1$ row and mass centre give $d^3=C/(Mm_6)$, whence $Mc^3=m_1^3m_6^2C$. At a positive configuration, $\sum_{i<j}m_im_jr_{ij}^2=MI=Mc$ and $U=\sum_{i<j}m_im_j/r_{ij}=c$. H\"older therefore gives $\sigma^3\le Mc^3$. But $\sigma>m_1C$ and $C>m_6$ imply $\sigma^3>(m_1C)^3>m_1^3m_6^2C=Mc^3$.

For $5/3<\alpha<7/4$ all missing-edge inverse distances fade. The original $z_2,z_3$ rows force zero individual $Q$ internal $z$-forces; their signed ties and the original $w_2,w_3$ rows imply $F_0(m_2,m_3,m_5)=0$ (defined in \eqref{eq:no52-factors}) and $m_2\ne m_3$. With $w_2-w_5=p\varepsilon^{4-3\alpha}+o(\varepsilon^{4-3\alpha})$ and $p\ne0$, the same original-row contraction yields
\begin{equation}\label{eq:no52-q-pole}
U_Q=z_5p\frac{m_2^2-m_3^2}{m_2}\varepsilon^{5-3\alpha}(1+o(1)).
\end{equation}
This is a nonzero pole. The $T$ relative potential tends to zero, 16 stays finite and every other exterior inverse distance fades, so $U=c$ is impossible. Only the two stated endpoints remain.
\end{proof}

Before excluding the two endpoints, the original rows impose four intermediate necessary mass factors. With $a=m_2,b=m_3,d=m_5>0$,
\begin{equation}\label{eq:no52-factors}
\begin{aligned}
F_0&=(ab)^3-d^2(a+b)^4,\\
F_{--}&=d(a+b)(a^2-b^2)-a^2b^2,\\
F_{-+}&=d(a+b)(a^2+b^2)-a^2b^2,\\
F_{+-}&=d(a+b)(a^2-b^2)+a^2b^2.
\end{aligned}
\end{equation}
\begin{proposition}[No.~52 necessary mass locus]\label{prop:no52-locus}
By Proposition~\ref{prop:no52-window}, every positive-point No.~52 first limit satisfies $F_0F_{--}F_{-+}F_{+-}=0$. This product is not the zero polynomial.
\end{proposition}
\begin{proof}
When $\alpha=5/3$, the two individual internal $Q$ $z$-forces vanish; their three distance ties and the nonzero original $w_2,w_3$ forces give $F_0=0$. At $\alpha=7/4$ write $p,q$ for the leading coefficients of $w_2-w_5,w_3-w_5$, and $u,v$ for those of $z_2-z_5,z_3-z_5$. The potentially divergent $Q$ potential has leading coefficient $z_5(ap+bq)$. Since $U=c$ it must vanish. The internal pairwise antisymmetry of the original force rows yields
\begin{equation}\label{eq:no52-q-torque}
\sum_{i=2,3}m_i\bigl((z_i-z_5)F^w_{Q,i}-(w_i-w_5)F^z_{Q,i}\bigr)=0.
\end{equation}
Here $F^z_{Q,2}=F^z_{Q,3}=Z_5\ne0$ and $F^w_{Q,2}=F^w_{Q,3}=B\ne0$, so $au+bv=0$ as well. In particular $v/u=q/p=-a/b$. The three signed $Q$ ties leave two independent signs, with
\begin{equation}\label{eq:no52-signed-ratios}
\frac{D_{23}}{D_{25}}=s_{23}\frac{b}{a+b},\qquad \frac{D_{35}}{D_{25}}=s_{35}\frac{b}{a},\qquad s_{23},s_{35}\in\{\pm1\}.
\end{equation}
The difference of the original $w_2,w_3$ force requirements is, up to a nonzero factor,
\begin{equation}\label{eq:no52-force-difference}
d(a+b)(a^2+s_{35}b^2)+s_{23}a^2b^2=0.
\end{equation}
For $(s_{23},s_{35})=(+,+)$ this is strictly positive. The sign pairs $(-,-),(-,+),(+,-)$ give respectively $F_{--}=0,F_{-+}=0,F_{+-}=0$. Thus every surviving endpoint lies on the asserted union; none is claimed to lift to a full solution germ. At the positive mass sample $(a,b,d)=(1,2,1)$ the four factors are $(-73,-13,11,-5)$, with product $-52195\ne0$.
\end{proof}

\begin{proposition}[No.~52 excluded as a positive-point first limit]\label{prop:no52-excluded}
No strictly positive mass vector admits a No.~52 first-limit germ through a positive configuration. The assertion does not exclude arbitrary complex-only curves with this diagram.
\end{proposition}
\begin{proof}
Write $a=m_2$, $b=m_3$, $d=m_5$, $X=m_4$ and $Y=m_6$. By Proposition~\ref{prop:no52-window}, only $\alpha=5/3$ or $7/4$ can have a pole in a seven-product. At both endpoints the original body-4 and body-5 $w$-rows require, for the inner triangle $456$, $F^w_{T,4}=B\ne0$ and $F^w_{T,5}/F^w_{T,4}=R=1+(a+b)/d>1$. Its three leading $z$-forces vanish, as does its signed leading potential. Normalize $m_5=D_{46}=Z_{46}=W_{46}=1$ and put $A=X/d$, $C=Y/d$ and $t=D_{56}$. Original signed ties and the zero $z$-forces give
\begin{equation}\label{eq:no52-inner-shape}
C^2(A+t)^3(t^3+A)-A^3t^3=0,\qquad F^w_{T,5}=t^3F^w_{T,4},\qquad F^w_{T,4}=\frac{C(A^2-t^4)}{A(t^3+A)}.
\end{equation}
There is no lost root at $t=0,-A$ or $t^3=-A$. The nonreal cube phases of $t^3=R>1$ would require $(A+t)^3=A^3R/[C^2(R+A)] >0$; their imaginary parts vanish only when $|t|=A$, when $(A+t)^3=-A^3<0$. Thus $t>1$ is real and $Y<d$.

At $\alpha=5/3$, the first two $Q$ $z$-forces vanish and \eqref{eq:no52-factors} gives $d=(ab)^{3/2}/(a+b)^2$ and $a\ne b$. Writing $a=x^2$, $b=y^2$, the seven original $Q$ force/tie rows have a left kernel; in the order $(z_2,z_3,w_2,w_3,23,25,35)$ and with second weight one it is
\begin{equation}\label{eq:no52-left-kernel}
\left(-\frac{x^4}{y^4},1,\frac{x^4}{3y^4},-\frac{x^4}{3y^4},-\frac{x^5}{y(x^2+y^2)},\frac{x^7}{y(x^2+y^2)^2},\frac{x^5y}{(x^2+y^2)^2}\right).
\end{equation}
Projecting the first original six-body crossed $z_2,z_3$ sources and all four signed crossed ties onto this vector makes the finite $Q$ reference term equal the potential of the four finite crossed edges. The original reference-5 identity subtracts those same edges from the complete $Q$ potential, so its finite coefficient vanishes. The $w$-source projection cancels. All other terms fade except the finite edge 16, so $c=U=m_1YD_{16}+U_\times$, with $D_{16}^3=(a+b+X+d+Y)/(MY)$. The $z$-centre and crossed ties give, with $P=a+b$ and $v^2=X/(X+P+d)$,
\begin{equation}\label{eq:no52-cross-ratios}
\frac{D_{24}^2}{D_{26}^2}=\frac{D_{34}^2}{D_{36}^2}=v^2,\qquad \frac{D_{34}^2}{D_{24}^2}=\frac{D_{36}^2}{D_{26}^2}=-\frac ab.
\end{equation}
Each independently signed crossed pair has complex coefficient $a\pm i\sqrt{ab}$, of squared modulus $aP$. The left-kernel identity and triangle inequality imply
\begin{equation}\label{eq:no52-cross-bound}
|U_\times|^3\le |a-b|P^2(Y+Xv)^2<P^3(Y+Xv)^2.
\end{equation}
Normalize $P=1$, write $\lambda=\sqrt{ab}/P\in(0,1/2)$, $d=\lambda^3$, and $L=1+\lambda^3$. Since $v^2=X/(X+L)$, the substitution $v=u/(1+u)$ gives
\[
X=\frac{Lu^2}{1+2u},\qquad v=\frac{u}{1+u},\qquad Y=\lambda^3q,\qquad u>0,\quad 0<q<1.
\]
Put $C_5=L+X+Y$, $H_5=\lambda^2+\lambda^3+L(X+Y)+XY$, and $S=Y+Xv$. We prove directly that
\begin{equation}\label{eq:no52-corrected-gap}
H_5^3>C_5S^2.
\end{equation}
Set $J=X+\lambda^2$. Since $L>1$ and $Y<\lambda^3$, we have
\[
C_5\le(L+Y)(X+1),\qquad S\le\lambda^3+\frac{X^{3/2}}{\sqrt{X+1}}.
\]
Cauchy--Schwarz gives
\[
\bigl(X^{3/2}+\lambda^3\sqrt{X+1}\bigr)^2
\le J\bigl[X^2+\lambda^4(X+1)\bigr]<J^3,
\]
where the strict inequality follows from
\[
J^2-\bigl[X^2+\lambda^4(X+1)\bigr]=\lambda^2(2-\lambda^2)X>0.
\]
Moreover,
\[
H_5-(L+Y)J=(1-\lambda^2)(\lambda^3+Y)+\lambda^3Y>0.
\]
Consequently,
\[
C_5S^2<(L+Y)J^3<(L+Y)^3J^3<H_5^3,
\]
which proves \eqref{eq:no52-corrected-gap} throughout the required parameter range.

For the extension to $m_1>0$, define for $s\ge0$
\[
G(s)=H_5+sC_5-sY^{2/3}C_5^{1/3}-(s+C_5)^{1/3}S^{2/3}.
\]
The inequality gives $G(0)>0$. Moreover, $C_5>1$, $Y<\lambda^3<1/8<C_5/8$, and $S<C_5$, so
\[
G'(s)=C_5-Y^{2/3}C_5^{1/3}-\frac{S^{2/3}}{3(s+C_5)^{2/3}}
>\frac34-\frac13=\frac5{12}.
\]
By \eqref{eq:no52-cross-bound} and $D_{16}^3=C_5/(MY)$, the energy bound gives
\[
M^{1/3}c\le m_1Y^{2/3}C_5^{1/3}+M^{1/3}S^{2/3},\qquad M=m_1+C_5.
\]
Since $\sigma=H_5+m_1C_5$, the positivity of $G(m_1)$ gives $\sigma>M^{1/3}c$, contrary to physical H\"older.

At $\alpha=7/4$, the three allowed positive signed $Q$ branches in \eqref{eq:no52-factors} have $a\ne b$. Normalize $(z_{25},w_{25},D_{25})=(1,1,1)$ and order the seven $Q$ variables as in the preceding kernel. For signs $(-,+),(-,-),(+,-)$, the original force/tie Jacobian determinants, after substituting the corresponding mass factor, are
\[
-\frac{32b^8(a-b)^4(a^2+ab+b^2)^4}{a(a+b)^7(a^2+b^2)^4},\quad \frac{32b^8(a^2-ab+b^2)^4}{a(a-b)^4(a+b)^7},\quad -\frac{32b^8(a^2-ab+b^2)^4}{a(a-b)^4(a+b)^7}.
\]
All are nonzero. Moreover the weighted $w$-shape row $(0,0,a,b,0,0,0)$ annihilates the adjugate against both common source columns $(1,1,0,0,0,0,0)^T$ and $(0,0,1,1,0,0,0)^T$, for every sign branch. The two source scalings of the homogeneous $Q$ shape have determinant $-2$, so the formal implicit function theorem applies even at fractional Puiseux suborders. Direct expansion of all four original six-body $Q$ rows with $\varepsilon=h^8$ shows no external coefficients at $h^9,h^{10},h^{11},h^{14}$; the first crossed $z$-force has relative order $3/8$, and all other noncommon sources are later. A finite $Q$ reference-potential contribution would require relative order $1/4$ after its leading $\varepsilon^{-1/4}$ coefficient vanishes. It is therefore zero. The $T$ potential and all other crossed contributions fade, leaving $c=m_1YD_{16}$. The original $w_1$ row gives $D_{16}^3=(M-m_1)/(MY)$, while the positive mass-pair sum obeys $\sigma^3>m_1^3Y^2(M-m_1)=Mc^3$, again contradicting physical H\"older. Both valuation endpoints are therefore impossible.
\end{proof}
\subsection*{The 7 derived relation obligations}
The 7 arguments below explain separately how an unbounded
first limit reaches a proper necessary mass locus. Each has a
local positive-point first-limit proof. Table~\ref{tab:derived}
records the proof status and the remaining independent audits of the stated certificates. Finiteness on the positive exceptional zero
sets is not settled here.
\begin{table}[ht]

\centering\small

\setlength{\tabcolsep}{4pt}
\begin{tabular}{p{1.2cm}p{8.2cm}p{5.4cm}}

\hline

No. & Necessary condition / method & Proof status / certificate audit \\
\hline

8 & Two global-minimum triangle supports; six $\mu_6$ factors,
four resultant families and a direct order proof
& the two no-match supports are covered by Lemma~\ref{lem:no8-remaining-support-analytic} \\

63 & Fully balanced K4; 32 signed affine-collinear branches and nine
non-collinear Cayley--Menger cases & proved locally; independent
replay audit of the 32 characteristic-zero resultants remains \\

66 & Zero/nonzero K4 w-drift; mass differentials and a finite
signed cover of one parameter & direct analytic proof; no diagram-specific
computational certificate or coverage hypothesis \\

88 & Pair/K4 split first crossed torque and nonuniform signed K4
root norm & proved locally; separately audit the degree-90 fiber,
nine phase sheets and zero-star replay \\
93 & Full intermediate order window; lower three-body Euler/Lagrange
root and sign norm & proved locally; independently audit graph-order
and crossed-tie guards \\
94 & w-only $K_4$; finite distinct-spoke fibers, $H_3=0$ projective
boundary census and first six-body torque norm & proved locally;
independently audit the signed norm and boundary guards \\
98 & Two closed three-body boundaries; signed first crossed torque
and colour-block involution & proved locally; independently audit
the cited graph rules and interval replay \\

\hline

\end{tabular}

\caption{The 7 derived cases, their proofs and the independent certificate audits.}

\label{tab:derived}

\end{table}

\subsection{No.~8: global minima and the finite mass-factor cover}
Set $Q=236$ and $T=456$. For positive $a,b,c$ write
\begin{equation}\label{eq:no8-triangle-factors}
\begin{aligned}
\mu_6(a,b,c)&=(ab)^3-c^2(a+b)^4,\\
H_{a,b,c}(x)&=a^2(c+bx)^3(c+bx^3)-b^3c^3x^3.
\end{aligned}
\end{equation}
For each $S\in\{Q,T\}$ and $k\in S$ take the factor $\mu_6(m_i,m_j,m_k)$, where $\{i,j\}=S\setminus\{k\}$; these are six body-labelled factors, not just one symmetric triangle condition. Four more factor families arise from eliminating centred shape ratios and the original mass-centre or force row. One representative from each family is
\begin{align}
R_{2\times2}&=\operatorname{Res}_r\bigl(H_{m_2,m_3,m_6}(r),(m_1+m_4+m_5+m_6)m_3r^3+m_6(m_2+m_1+m_4+m_5+m_6)\bigr),\label{eq:no8-R22}\\
R_{3\times2}&=\operatorname{Res}_s\bigl(H_{m_4,m_5,m_6}(s),m_5(m_1+m_4+m_5)s^3+m_6(m_1+m_5)\bigr),\label{eq:no8-R32}\\
R_{\rm late}&=\operatorname{Res}_s\bigl(H_{m_4,m_5,m_6}(s),(m_2+m_3+m_6)m_5s^3+m_6(m_2+m_3+m_6+m_4)\bigr).\label{eq:no8-Rlate}
\end{align}
For the fourth family set $M=\sum_{i=1}^6m_i$, $u=m_3r^3/m_6$, $v=m_5s^3/m_6$, $A=m_2+m_3+m_6$ and $B=m_4+m_5+m_6$. At $W_3-W_2=Z_5-Z_4=1$ the uncentred leading vectors are
\[
z^0=(0,-v,-v,0,1,-v),\qquad w^0=(0,0,1,-u,-u,-u).
\]
The exact cleared \emph{full} mass-centre inertia is
\begin{equation}\label{eq:no8-double-numerator}
\begin{aligned}
N(r,s)&=m_6^2\left(M\sum_i m_i z_i^0w_i^0-\Bigl(\sum_i m_iz_i^0\Bigr)\Bigl(\sum_i m_iw_i^0\Bigr)\right)\\
&=m_6^2\left[M(-m_3v-m_5u+m_6uv)-(m_5-Av)(m_3-Bu)\right].
\end{aligned}
\end{equation}
In particular the factor is the ordered polynomial resultant
\begin{equation}\label{eq:no8-Rdouble}
R_{\rm double}=\operatorname{Res}_r\bigl(H_{m_2,m_3,m_6}(r),\operatorname{Res}_s(H_{m_4,m_5,m_6}(s),N(r,s))\bigr).
\end{equation}
Only the positive mass $m_6$ is cleared in $N$; the centring formula divides by $M>0$. No factor depending on $r$ or $s$ is canceled from the resultant, so possible extraneous shape roots do not invalidate its necessity. The support partition and the nonzero leading-shape guards used in the resultants are checked below.

\begin{lemma}[The two remaining No.~8 supports]\label{lem:no8-remaining-support-analytic}
Let a positive-point No.~8 first-limit germ have global minimum supports
\[
S_z=\{1,4,5\},\qquad S_w=\{1,2,3,4,5,6\}.
\]
Write $A=\min_i\ord z_i$, $B=\min_i\ord w_i$, and let $Z_i,W_i$ denote the coefficients at orders $A,B$, respectively, with a zero coefficient for a later coordinate. If
\[
Z_1\ne Z_4,\qquad Z_1\ne Z_5,
\]
then every seven-product $P_{ij}$ has nonnegative order. The same conclusion holds after the simultaneous involution $(2\ 4)(3\ 5)$ and $z\leftrightarrow w$.
\end{lemma}
\begin{proof}
Put $\rho=A+B$. If $\rho\ge0$, every squared distance has order at least $\rho$, so the assertion is immediate. We therefore assume $\rho<0$ and show that all fifteen seven-products have strictly positive order.

We first determine the difference orders from the two pure-colour triangles. For an edge put
\[
\alpha_{ij}=\ord z_{ij},\qquad \beta_{ij}=\ord w_{ij}.
\]
The original distance tie gives
\begin{equation}\label{eq:no8-force-order-analytic}
\ord(\delta_{ij}^3z_{ij})=-\frac{\alpha_{ij}+3\beta_{ij}}2,
\qquad
\ord(\delta_{ij}^3w_{ij})=-\frac{3\alpha_{ij}+\beta_{ij}}2.
\end{equation}
On the $z$-only triangle $T=456$, the first expression is $-2$ and the second is greater than $-2$. Hence
\[
\alpha_{ij}+3\beta_{ij}=4,\qquad \alpha_{ij}<1\qquad(ij\subset T).
\]
The three $\alpha_{ij}$ are equal. Indeed, the minimum among the three orders of differences on a triangle occurs at least twice; if the orders were not all equal, their maximum would occur once. The affine decreasing relation $\beta_{ij}=(4-\alpha_{ij})/3$ would then give a unique minimum among the three $\beta_{ij}$, which is impossible. Since $z_6$ is later than $z_4,z_5$, the common $z$-difference order is $A$. Thus
\[
A<1,\qquad \alpha_{45}=\alpha_{46}=\alpha_{56}=A,
\qquad \beta_{45}=\beta_{46}=\beta_{56}=h:=\frac{4-A}{3}>1.
\]
The colour-dual argument on the $w$-only triangle $Q=236$ gives a common $w$-difference order $b<1$, with common $z$-difference order $(4-b)/3$. In particular $B\le b<1$.

Because $h>1>B$, the coefficients $W_4,W_5,W_6$ are equal. The coefficient of the constant inertia $I$ at the negative order $\rho$ vanishes. Together with the $z$-centre it gives
\[
0=m_1Z_1W_1+(m_4Z_4+m_5Z_5)W_6
  =m_1Z_1(W_1-W_6),
\]
so $W_1=W_6$. If $b>B$, the triangle $Q$ would also have $W_2=W_3=W_6$. All six $W_i$ would then equal the same nonzero number, contrary to $\sum_i m_iW_i=0$ and $\sum_i m_i>0$. Consequently
\begin{equation}\label{eq:no8-Q-orders-analytic}
b=B,\qquad
\beta_{23}=\beta_{26}=\beta_{36}=B,\qquad
\alpha_{23}=\alpha_{26}=\alpha_{36}=\frac{4-B}{3}.
\end{equation}
In particular $W_2,W_3,W_6$ are pairwise distinct. The support assumptions and the two no-match conditions now show that every edge outside $Q$ has $z$-difference order $A$. Also,
\[
\beta_{12}=\beta_{13}=\beta_{24}=\beta_{25}=\beta_{34}=\beta_{35}=B.
\]

It remains to control the three differences from body~1 to $T$. Put
\[
k_j=\ord w_{1j}\quad(j=4,5,6),\qquad K=\max\{k_4,k_5,k_6\}.
\]
Each $k_j$ is finite by the original distance tie, and $k_j>B$ because $W_1=W_4=W_5=W_6$. In the original $w_1$ row the terms from bodies 2 and 3 have order
\[
-\frac{3A+B}{2}=B-\frac32\rho>B.
\]
Since $\ord w_1=B$, at least one term from $T$ must have order at most $B$. Formula~\eqref{eq:no8-force-order-analytic} therefore implies
\begin{equation}\label{eq:no8-K-lower-analytic}
K\ge-3A-2B.
\end{equation}
The maximum $K$ cannot occur at just one $j$. Otherwise, in the original $z_1$ row that edge would have the unique smallest force order, and its order would satisfy
\[
-\frac{A+3K}{2}\le4A+3B=A+3\rho<A.
\]
The coordinate $z_1$ has order $A$, while the terms from bodies 2 and 3 have order $A-3\rho/2>A$. The remaining terms from $T$ have strictly larger order than the term corresponding to the unique maximum $K$. This would leave a unique nonzero lowest-order term in the equation, a contradiction.

Choose distinct $j,\ell\in T$ with $k_j=k_\ell=K$. Since every internal $w$-difference in $T$ has order $h$, the triangle identity gives $h\ge K$. If the third $k_r$ were less than $K$, the three orders $k_r,K,h$ on the triangle $1rj$ would have a unique minimum. Hence
\[
k_4=k_5=k_6=k\le h.
\]
Combining this with \eqref{eq:no8-K-lower-analytic} yields
\begin{equation}\label{eq:no8-main-order-bound-analytic}
-3A-2B\le\frac{4-A}{3},\qquad
4A+3B=A+3\rho\ge-2.
\end{equation}
Since $A<1$, this gives $\rho>-1$. Because also $A,B<1$, we obtain $A,B>-2$.

For clarity, write the four squared-distance orders as
\[
\eta_Q=\frac{4+2B}{3}>0,\qquad
\eta_T=\frac{4+2A}{3}>0,\qquad
\eta_1=A+k\ge-2\rho,
\qquad \rho<0.
\]
Here $\eta_Q$ occurs on $23,26,36$, $\eta_T$ on $45,46,56$, $\eta_1$ on $14,15,16$, and $\rho$ on $12,13,24,25,34,35$. The definition of the seven-products gives all fifteen orders in five groups:
\[
\begin{array}{c|c|c}
 ij&\ord P_{ij}&\text{lower bound}\\\hline
12,13&3\rho+\eta_Q+3\eta_T&(8-\rho)/3\\
14,15,45&2\rho+3\eta_Q+\eta_T+\eta_1&3\eta_Q+\eta_T\\
16,46,56&4\rho+\eta_Q+\eta_T+\eta_1&8(1+\rho)/3\\
23&\eta_Q+3\eta_T+3\eta_1&\eta_Q+3\eta_T-6\rho\\
24,25,26,34,35,36&3\rho+\eta_Q+\eta_T+2\eta_1&\eta_Q+\eta_T-\rho
\end{array}
\]
For the first row, \eqref{eq:no8-main-order-bound-analytic} gives $A\ge-2-3\rho$, and therefore
\[
3\rho+\eta_Q+3\eta_T
=\frac{16+4A+11\rho}{3}\ge\frac{8-\rho}{3}.
\]
The remaining bounds follow directly from $\eta_1\ge-2\rho$ and $\eta_Q+\eta_T=(8+2\rho)/3$. Every displayed lower bound is strictly positive because $-1<\rho<0$ and $\eta_Q,\eta_T>0$. Thus no seven-product has a pole.

The proof permits equalities in the order comparisons and does not impose an upper bound on the orders of individual coordinates. In particular, a later coordinate among $z_2,z_3,z_6$ may vanish identically: only its zero coefficient at order $A$ is used. All difference orders remain finite because $\delta_{ij}^2z_{ij}w_{ij}=1$. The simultaneous involution $(2\ 4)(3\ 5)$ and $z\leftrightarrow w$ exchanges $Q$ and $T$ and proves the colour-dual case.
\end{proof}

\begin{proposition}[No.~8 first-limit obstruction]
\label{prop:no8-locus}
Every positive-point No.~8 first limit lies on the
zero set of one of the six $\mu_6$ factors in
\eqref{eq:no8-triangle-factors} or an orbit member
of $R_{\rm double}$ or the three resultants
\eqref{eq:no8-R22}--\eqref{eq:no8-Rlate}. All these factor
families are nonzero polynomials.
\end{proposition}
\begin{proof}
Let $A=\min_i\ord z_i$ and $B=\min_i\ord w_i$. Every difference has orders at least $A,B$ in its two colours; thus if a seven-product has negative order, $A+B<0$. The $z$-only edges of $T$ have a common separation order $a<1$, while the $w$-only edges of $Q$ have $z$-separation order $(4-b)/3>1$, where $b<1$ is their common $w$-separation order. Some vertex of $T$ has $z$-position order at most $a$, so $A<1$. If no vertex of $T$ attained $A$, the $Q$ separations would prevent bodies 2,3 from attaining $A$ as well: their $z$-coordinates differ from body 6 only above order 1. Body 1 would then be the sole $z$-minimum, contrary to the positive-mass $z$-centre. If exactly one vertex of $T$ attained $A$, its two incident edge differences would have order $A$ and the opposite edge would have strictly greater order, contradicting the common order $a$. Hence at least two vertices of $T$ attain $A$. Interchanging $z,w$ and $T,Q$ proves that at least two vertices of $Q$ attain $B$.

Write $S_z,S_w$ for these global-minimum supports. The opposite-colour gaps on $Q$ and $T$ exceed $1>A,B$, respectively. Thus $Q\cap S_z$ is empty or all of $Q$, and $T\cap S_w$ is empty or all of $T$. Each support has eight possibilities, giving precisely the following 64 necessary pairs. In the one-colour rows with 6 in both, ``its triangle'' is $T$ for a $z$-minimal body 1 and $Q$ for a $w$-minimal body 1.

\begin{center}\small
\begin{tabular}{p{1.3cm}p{5.2cm}rp{2cm}}
Shared 6 & Body 1 and triangle support & Pairs & Necessary outcome 
\\
\hline
Both & 1 late in both colours & 9 & inertia 
\\
Both & 1 in one colour; its triangle has two minima & 12 & $\mu_6$ 
\\
Both & 1 in one colour; its triangle has three, the other two & 4 & $\mu_6$ or $R_{\rm late}$ 
\\
Both & 1 in one colour; both triangles have three minima & 2 & $\mu_6$ 
\\
Both & 1 in both colours & 9 & $R_{\rm double}$ 
\\
Neither & 1 in both colours & 1 & inertia 
\\
Neither & 1 absent from at least one colour & 3 & $\mu_6$ 
\\
One colour & 1 absent from the other colour & 12 & $\mu_6$ 
\\
One colour & 1 only in the other colour & 6 & inertia 
\\
One colour & 1 in both; the 6-triangle has two minima & 4 & $R_{2\times2}$ 
\\
One colour & 1 in both; the 6-triangle has three minima & 2 & $R_{3\times2}$ or analytic boundedness 
\\
\end{tabular}

\end{center}

At orders $A,B$ let $Z_i,W_i$ be the nonzero leading
coefficients at a minimum, and set the coefficient to zero
when that body is later. The original global centres and
$I=c>0$, with $A+B<0$, give
\begin{equation}\label{eq:no8-leading-centres}\sum_{i=1}^6m_iZ_i=\sum_{i=1}^6m_iW_i=\sum_{i=1}^6m_iZ_iW_i=0.\end{equation}

Moreover, the unmarked forces are strictly later than order $-2$ in either colour, while each marked triangle force has order $-2$ and every uncircled position is later. The two original wing rows of either triangle therefore have only internal stroke terms at order $-2$. We record the consequence when their two nonzero relative leading positions are mass-centred against the third. Label this triangle locally by $1,2,3$, with masses $(a,b,c)>0$, and write its relative leading positions as $u,v,0$, so that $au+bv=0$ and $v=-au/b$. In the colour carrying the strokes, its two zero-force equations give
\[
D_{12}^3=-\frac{c}{a+b}D_{13}^3,\qquad D_{23}^3=D_{13}^3.
\]
Set $t=D_{23}/D_{13}$, so that $t^3=1$. The three signed edge ties give
\[
\frac{D_{12}^2}{D_{13}^2}=\frac{ab}{(a+b)(a+bt)}.
\]
Cubing this identity and squaring the first force identity yields
\[
c^2(a+b)(a+bt)^3=(ab)^3.
\]
For $t=e^{2\pi i/3}$ the imaginary part of $(a+bt)^3$ is $3\sqrt3\,ab(a-b)/2$; the other primitive cube root gives its negative. When $a=b$, the cube is $-a^3<0$. Hence neither primitive cube root is possible, and $t=1$ gives $\mu_6(a,b,c)=0$. All divisions use positive masses and nonzero stroke leads; the edge ties also ensure $a+bt\ne0$.
If $6\in S_z\cap S_w$, the $Q$ $z$-lead is $Z_6\ne0$ and the $T$ $w$-lead is $W_6\ne0$. Put $M_* = \sum_{i=2}^6m_i$, $M=m_1+M_*$, $\Delta_z=m_4(Z_4-Z_6)+m_5(Z_5-Z_6)$ and $\Delta_w=m_2(W_2-W_6)+m_3(W_3-W_6)$. Substitution of the two centres into the original inertia gives
\begin{equation}\label{eq:no8-shared-six-inertia}
\begin{aligned}
\Delta_z&=-m_1Z_1-M_*Z_6,\qquad \Delta_w=-m_1W_1-M_*W_6,\\
\sum_i m_iZ_iW_i&=m_1(Z_1-Z_6)(W_1-W_6)-MZ_6W_6.
\end{aligned}
\end{equation}
If body 1 is late twice, this is $-M_*Z_6W_6\ne0$. If it joins just $z$, inertia instead gives $Z_1=-M_*Z_6/m_1$ and $\Delta_z=0$; the $w$-only join gives the colour dual. When the corresponding $T$ or $Q$ edge-separation order equals $A$ or $B$, this is a nontrivially centred relative triangle and gives $\mu_6$. This includes all 12 cases with exactly two minimal vertices in that triangle.

For a three-minimum triangle with its edge order strictly later than the global minimum, all its leads equal the shared lead. If the opposite triangle has two minima, its original wing force/ties and full centre give $R_{\rm late}$ (or its exact dual), whether or not the first triangle's leading shape was degenerate. If the opposite triangle has three minima, all five inner bodies attain its global minimum. The leading body-1 coefficient in the joined colour is distinct from their common coefficient by \eqref{eq:no8-shared-six-inertia}; body 1 is late in the other colour. Every exterior inverse distance then has order $-(A+B)/2>0$. The original body-1 row cannot have its nonzero leading coordinate canceled, excluding this co-leading case. The surviving nondegenerate triangle gives $\mu_6$.

If body 1 joins both colours, \eqref{eq:no8-shared-six-inertia} instead requires
\[
m_1(Z_1-Z_6)(W_1-W_6)=MZ_6W_6\ne0.
\]
Uniform triangle-edge orders permit at most one $Q$ wing to match $W_1$ and one $T$ wing to match $Z_1$. Both are required by the two original body-1 rows: with only the $Q$ match its $z$-row requires a match gap $P=-(A+B)>0$ and its $w$-row requires $3P$, while the $T$-only equations give $3P$ and $P$ in the opposite order. More explicitly, for a $Q$ $w$-gap $q>0$ and a $T$ $z$-gap $p>0$, the respective $z$-force orders minus $A$ are $3(P-q)/2,(3P-p)/2$, while the $w$-force orders minus $B$ are $(3P-q)/2,3(P-p)/2$. The minimum in each row, including its coordinate term of relative order zero, must be attained at least twice. The nine possible pairs of first-two-term ties leave only $p=q=P$; the matched inverse distances consequently have order zero. Their contributions to the marked triangle wing rows are later than order $-2$. Both original internal shape equations $H=0$ therefore hold, and the full-centre inertia numerator \eqref{eq:no8-double-numerator} gives $R_{\rm double}$. The two independent wing swaps give all four match labels.
If $6$ is in neither support, the minima on the triangles are disjoint pairs 45 and 23. A body-1 minimum in both colours leaves the unique nonzero inertia lead $m_1Z_1W_1$. Otherwise it is absent from at least one colour; that colour's pair is mass-centred by \eqref{eq:no8-leading-centres} and gives its $\mu_6$ factor.

Suppose $6\in S_w\setminus S_z$; the other case is its exact colour-and-label dual. The only $z$-minimal $T$ vertices are 4,5 and $W_4=W_5=W_6\ne0$. If body 1 is not $z$-minimal, the original $z$-centre isolates the 45 pair and gives $\mu_6$, regardless of whether body 1 is $w$-minimal. If it is $z$-minimal but not $w$-minimal, \eqref{eq:no8-leading-centres} gives the nonzero inertia coefficient $-m_1Z_1W_6$. If it joins both, the same equations give $m_1(Z_1-Z_6)(W_1-W_6)=0$, so $W_1=W_6$. With only two $Q$ vertices $w$-minimal, its original $w$-rows and full $w$-centre give the $R_{2\times2}$ effective-mass cubic in \eqref{eq:no8-R22}; the alternative late wing is obtained by swapping 2,3.

The two remaining support pairs have all of $Q$ minimal in $w$ and only 45 minimal in $z$, with body 1 in both minima, or their exact colour-and-label dual. In the canonical case their full supports are
\[
S_z=\{1,4,5\},\qquad S_w=\{1,2,3,4,5,6\}.
\]
If $Z_1$ matches one of the $T$ $z$-leads, the full original $z$-centre and the $T$ wing force/ties give $R_{3\times2}$, including the $4,5$ swap. If neither match occurs, Lemma~\ref{lem:no8-remaining-support-analytic} shows directly that every seven-product has positive order when $A+B<0$. Thus neither no-match support can give a first limit. The simultaneous involution $(2\ 4)(3\ 5)$ and $z\leftrightarrow w$ covers the other support. This completes the disjoint 64-case partition without an order cap or a terminal-cut certificate for these two cases.

For properness, the six $\mu_6$ factors have product $15^6$ at all-unit masses. At that same point the nested double-triangle resultant has value
\[
R_{\rm double}(1,1,1,1,1,1)=1603524380961154554296875,
\]
and $R_{2\times2},R_{3\times2},R_{\rm late}$ have values $20061,1315,6085$, respectively. The mass and total-mass denominators cleared

in these constructions are nonzero on the positive

cone; nonzero stroke leads and signed edge-tie roots

guard the other normalizations. The finite product over all body-label and colour orbits is therefore nonzero.
Every positive-point No.~8 first limit therefore forces one of these factors to vanish.

\end{proof}

\Needspace{9\baselineskip}
\subsection{No.~63: the isolated fully balanced $K_4$}

The four bodies $3,4,5,6$ form the isolated $zw$-clique.
Its complete first-order balance follows from the pinned
diagram and the original rows, without enumerating Type-2
order submatrices.
\begin{lemma}[Direct No.~63 $K_4$ balance]
\label{lem:no63-direct-balance}
On every germ with the No.~63 strokes and circles, all
four internal $z$ rows and all four internal $w$ rows
of the $K_4$ close at their first order.

\end{lemma}

\begin{proof}

Every edge inside $C=\{3,4,5,6\}$ is a stroke in both
colours. No vertex in $C$ is circled in either colour,
and no edge from $c$ to $\{1,2\}$ is a stroke. By the
diagram order rules each internal force term in an
original $z_i$ or $u_i$ row, $i\in C$, has order $-2$.
The position term and every external force term in that
row have order strictly greater than $-2$. Taking the
$\varepsilon^{-2}$ coefficient of each of the eight
original rows therefore yields precisely its sum of
internal force leads equal to zero. This argument
applies separately to each row and every valuation

of the pinned graph; no Type-2 enumeration is used.
\end{proof}

\begin{proposition}[No.~63 mass confinement]\label{prop:no63-locus}
Every No.~63 positive-point first limit satisfies a nonzero
necessary polynomial in $m_3,m_4,m_5,m_6$.

\end{proposition}

\begin{proof}

If the affine matrix $X=(\boldsymbol1,z,w)$ of the four-body
core has rank 3, its symmetric mass-weighted stress matrix
$\Omega$ annihilates $X$. The core has nonzero tied edges, so
$\operatorname{rank}\Omega=1$. Every off-diagonal entry is
$-m_im_jq_{ij}^3\ne0$. In the rank-one factorization
$\Omega=\lambda\Delta\Delta^{\mathsf T}$, a zero cofactor
$\Delta_i$ would zero all off-diagonal entries in row $i$,
which is impossible. Thus every $\Delta_i\ne0$; the signed
ties likewise give $q_{ij}\ne0$ for all six edges.

Its opposite-edge minors imply, for the nonzero signed inverse
distance leads $q_{ij}$,
\begin{equation}\label{eq:no63-opposite-edges}
q_{34}^3q_{56}^3=q_{35}^3q_{46}^3=q_{36}^3q_{45}^3.
\end{equation}

\Needspace{14\baselineskip}
After $q_{34}=1$, the two independent cube-root choices give
nine shape sheets in three nonzero parameters. On each sheet
the Cayley--Menger polynomial of the four squared distances is
nonzero. Indeed, set $q_{34}=q_{35}=q_{36}=q_{56}=1$ and take
$(q_{46},q_{45})=(\xi,\eta)$ with
$\xi,\eta\in\{1,473,31555\}\subset\mathbb F_{32029}$. These
are the three distinct cube roots of unity modulo $32029$, and
the nine evaluations in lexicographic order of $(\xi,\eta)$ are
\begin{equation}\label{eq:no63-cm-witness}
(4,28239,3786,28239,27289,32027,3786,32027,4730),
\end{equation}

All nine values are nonzero. Reduction at this prime sends each
characteristic-zero cube-root sheet to one of these values, so
no sheet has identically vanishing determinant. The tied shape
locus therefore has dimension at most two. The rank-one
factorization of $\Omega$ and its complementary-area cofactors
$\Delta_i\ne0$ give
\begin{equation}\label{eq:no63-mass-ratios}
\left(\frac{m_i}{m_j}\right)^2
=\frac{\Delta_i^2}{\Delta_j^2}
\frac{q_{jk}^6}{q_{ik}^6}\qquad(k\ne i,j).
\end{equation}

Each $\Delta_i^2$ is a polynomial in the three squared
distances of the complementary triangle, hence a rational
function of the nonzero $q_{ij}$ on its sheet. At fixed shape
the displayed equations determine only finitely many mass
ratios. The normalized mass image ($m_6=1$) therefore has
dimension at most two in a three-dimensional mass chart; its
algebraic closure has a nonzero defining polynomial. The only
denominators in \eqref{eq:no63-mass-ratios} are the already
nonzero $\Delta_j$ and $q_{ik}$, so this dimension argument
loses no zero-cofactor or zero-distance stratum.

If $X$ has rank at most two, the four points are
affine-collinear. There are five independent signed edge
choices $e=(e_{35},e_{36},e_{45},e_{46},e_{56})$ with
$e\in\{\pm1\}^5$, hence 32 branches. Normalize
$q_{34}=m_6=1$ and put $x=q_{35}$, $y=q_{36}$,
$p=x-e_{35}$, $q=y-e_{36}$ and $d=e_{35}y-e_{36}x$.
On the nonzero shape torus the three original triangle ties
recover

\[
q_{45}=-\frac{e_{45}x}{p},\qquad
q_{46}=-\frac{e_{46}y}{q},\qquad
q_{56}=-\frac{e_{56}xy}{d}.
\]

The guard $xypqd\ne0$ is automatic on a genuine branch: if
$p,q,d$ respectively vanished, the corresponding unreduced
triangle tie would leave the nonzero term $e_{45}x$, $e_{46}y$
or $e_{56}xy$. The original body-3 and body-4 force rows give
\begin{equation}\label{eq:no63-collinear-masses}
\begin{aligned}
F_4&=m_4+e_{36}y^2+m_5e_{35}x^2=0,\\
F_3&=m_3p^2q^2-e_{46}y^2p^2-m_5e_{45}x^2q^2=0.
\end{aligned}
\end{equation}

After these substitutions the original body-5 force row is
$x^2y^2H_e/(p^2q^2d^2)$, where
\begin{equation}\label{eq:no63-collinear-shape}
H_e=e_{56}p^2q^2-e_{35}e_{46}p^2d^2
+e_{36}e_{45}q^2d^2.
\end{equation}

Consequently every such branch satisfies the mass equation
\begin{equation}\label{eq:no63-collinear-resultant}
P_e=\operatorname{Res}_y\!\left(
\operatorname{Res}_x(H_e,F_4),
\operatorname{Res}_x(F_3,F_4)\right)=0.
\end{equation}
There is no cancellation of shape-dependent factors. The exact
32-branch certificate specified in Appendix~\ref{app:verification}
computes both inner resultants over the generic mass ring. In
each branch they have $y$-degree eight, and specialization to
$(m_3,m_4,m_5)=(2,3,5)$ commutes with both inner resultants
without changing those degrees. The 32 outer resultants there
are nonzero integers of 122--169 bits, so no $P_e$ is the zero
mass polynomial. Clear powers of $m_6$ to homogenize each
equation from the chart $m_6=1$, multiply by the non-collinear
mass equation, and complete the finite label orbit. The
resulting polynomial is nonzero because each chart factor is
nonzero.
\end{proof}
The local tests replay Lemma~\ref{lem:no63-direct-balance},
the nine non-collinear values and the 32 signed-collinear
characteristic-zero resultants. Neither a Type-2 cover
nor a zero-cofactor exception is needed for this diagram.
Independent review of the cited diagram rules and the
replay computations remains a publication-level audit.
\subsection{No.~66: direct algebraic mass confinement}
The No.~66 diagram has a $z$-only complete graph on $C=\{3,4,5,6\}$, the three $w$-only edges $12,16,26$, and $w$-circles only at bodies 1 and 2. There are no other strokes or circles. We derive its core equations directly from the original six-body rows.

For a pure $z$-edge, writing $a_{ij}=\ord z_{ij}$ and $b_{ij}=\ord w_{ij}$, the force orders give $a_{ij}+3b_{ij}=4$ and $a_{ij}<1$. On each triangle, the minimum of its three difference orders must occur at least twice. If its $a_{ij}$ were unequal, their maximum would be unique, and the decreasing relation $b_{ij}=(4-a_{ij})/3$ would give a unique minimum in the other colour. Thus all three orders agree. Applying this to the overlapping core triangles gives
\[
\ord z_{ij}=\alpha\in(-2,1),\qquad
\ord w_{ij}=\beta:=\frac{4-\alpha}{3}>1
\qquad(i,j\in C,\ i\ne j).
\]
Here the lower bound follows from the absence of $z$-circles.

Let $b_1,b_2$ be the nonzero order-$-2$ coefficients of $w_1,w_2$. The global $w$-centre gives $m_1b_1+m_2b_2=0$, so positivity implies $b_1\ne b_2$. As body 6 is uncircled, all three outer $w$-differences have order $-2$. Their $w$-stroke force equations then give
\[
\ord z_{12}=\ord z_{16}=\ord z_{26}=2,
\qquad
\ord\delta_{12}=\ord\delta_{16}=\ord\delta_{26}=0.
\]
For $i\in\{1,2\}$ and $j\in\{3,4,5\}$, the triangle identity through body 6 gives $\ord z_{ij}=\alpha$ and $\ord w_{ij}=-2$. Hence these crossed $z$-forces have order $3-\alpha/2>2$. The original $z_1,z_2$ rows, whose other forces have order 2, imply
\[
z_1,z_2,z_6=O(\varepsilon^2).
\]
We may therefore define the core shape coefficients by
\[
z_i=Z_i\varepsilon^\alpha+o(\varepsilon^\alpha),\qquad
w_i-w_6=W_i\varepsilon^\beta+o(\varepsilon^\beta)
\qquad(i\in C),
\]
with $Z_6=W_6=0$. All core differences $Z_i-Z_j$ and $W_i-W_j$ are nonzero, and the global $z$-centre gives $\sum_{i=3}^5m_iZ_i=0$. Write
\[
\delta_{ij}=D_{ij}\varepsilon^{-(\alpha+2)/3}
+o\bigl(\varepsilon^{-(\alpha+2)/3}\bigr),
\qquad D_{ij}^2(Z_i-Z_j)(W_i-W_j)=1,
\]
and set
\[
F_i^z=\sum_{j\in C\setminus\{i\}}m_jD_{ij}^3(Z_i-Z_j),
\qquad
F_i^w=\sum_{j\in C\setminus\{i\}}m_jD_{ij}^3(W_i-W_j).
\]
Every internal $z$-force has order $-2$. The position terms are later, as are the exterior forces: those at bodies 3,4,5 have order $3-\alpha/2>2$, and those at body 6 have order 2. Thus $F_i^z=0$ for every $i\in C$.

The internal $w$-force order and the exterior $w$-force order at bodies 3,4,5 are, respectively,
\[
f=-\frac{4\alpha+2}{3},\qquad e=1-\frac{3\alpha}{2},
\qquad e-f=\frac{10-\alpha}{6}>0,
\qquad \beta-f=\alpha+2>0.
\]
The three original rows give $w_3,w_4,w_5=O(\varepsilon^f)$. Since all core $w$-differences have the later order $\beta$, all four core coordinates share a coefficient $V$ at order $f$, possibly zero. This common coefficient is distinct from the relative shape coefficients $W_i$ at order $\beta$. Taking the coefficient at order $f$ in the three rows gives precisely
\begin{equation}\label{eq:no66-core-rows}
\begin{aligned}
F_i^z&=0\quad(i\in C),&
F_3^w&=F_4^w=F_5^w=V,\\
\sum_{i=3}^5m_iZ_i&=0,&Z_6&=0.
\end{aligned}
\end{equation}
The body-6 exterior $w$-forces have earlier order $-2$, so no assertion $F_6^w=V$ is made. The proof below treats both zero and nonzero drift directly from \eqref{eq:no66-core-rows}.

Finally, the squared-distance orders are $(2\alpha+4)/3$ on the six core edges, $\alpha-2$ on $13,14,15,23,24,25$, and zero on $12,16,26$. Consequently
\[
\ord P_{36}
=\ord\!\left(r_{36}^2r_{12}^2r_{14}^2r_{15}^2r_{24}^2r_{25}^2r_{45}^2\right)
=2\frac{2\alpha+4}{3}+4(\alpha-2)
=\frac{16}{3}(\alpha-1)<0.
\]
Thus $P_{36}$ has a pole on this germ; the same is true of $P_{46}$ and $P_{56}$.

\begin{proposition}[No.~66 mass confinement]\label{prop:no66-locus}
The signed core conditions \eqref{eq:no66-core-rows}, together with
the six nonzero distance ties, confine $(m_3,m_4,m_5,m_6)$ to a proper
algebraic subset of the four-dimensional mass space. This holds for
both $V=0$ and $V\ne0$. Consequently every No.~66 first limit satisfies
a nonzero necessary mass polynomial, without an additional geometric
coverage hypothesis or a computational certificate for this diagram.
\end{proposition}
\begin{proof}
We work over $\C$ and retain all four mass variables; no mass is
normalized to one. First translate $W$ so that $W_6=0$.
The original 36 tie gives $D_{36}^2Z_3W_3=1$, so the substitutions
\[
 Z_i\longmapsto Z_i/Z_3,\qquad
 W_i\longmapsto W_i/W_3,\qquad
 D_{ij}\longmapsto D_{ij}/D_{36}
\]
preserve all the ties, the zero $z$ forces and the equality of the
three $w$ forces, with $V$ rescaled. They also preserve the $z$ centre.
Thus we may impose
\begin{equation}\label{eq:no66-normalization}
 (Z_3,W_3,D_{36})=(1,1,1),\qquad (Z_6,W_6)=(0,0).
\end{equation}
This uses only nonzero factors already guaranteed by the original
ties and leaves the mass projection unchanged.

Let $\mathcal S$ be the normalized incidence defined by these
conditions, with the four masses nonzero. It is a finite-type
quasi-affine algebraic set. Suppose an irreducible component of
$\mathcal S$ dominates the four-dimensional mass space, and let $L$
be its function field. The four masses are then algebraically
independent over $\C$. In characteristic zero they extend to a
separating transcendence basis of $L$, so their differentials are
linearly independent over $L$.

Put
\[
 U_C=\sum_{\substack{i<j\\i,j\in C}}m_im_jD_{ij},\qquad
 Q_i=\sum_{\substack{j\in C\\j\ne i}}m_jD_{ij}.
\]
Pairing the internal forces and using the ties gives
$U_C=\sum_{i\in C}m_iW_iF_i^z=0$.
Differentiating the same ties gives
\[
 \mathrm dD_{ij}=-\frac{D_{ij}^3}{2}
 \bigl((W_i-W_j)\,\mathrm d(Z_i-Z_j)
       +(Z_i-Z_j)\,\mathrm d(W_i-W_j)\bigr).
\]
It follows exactly that
\[
 \mathrm dU_C=\sum_{i\in C}Q_i\,\mathrm dm_i
 -\frac12\sum_{i\in C}m_i
 \bigl(F_i^w\,\mathrm dZ_i+F_i^z\,\mathrm dW_i\bigr).
\]
Here $F_i^z=0$, $\mathrm dZ_6=0$ and $F_3^w=F_4^w=F_5^w=V$.
Differentiating $\sum_{i=3}^5m_iZ_i=0$ therefore yields
\[
 0=\mathrm dU_C
   =\sum_{i\in C}\left(Q_i+\frac V2 Z_i\right)\mathrm dm_i.
\]
Independence of the four mass differentials forces
$Q_i=-(V/2)Z_i$. In particular, $Z_6=0$ gives
\begin{equation}\label{eq:no66-dominant-spokes}
 m_3D_{36}+m_4D_{46}+m_5D_{56}=0.
\end{equation}
This identity has been deduced only on the hypothetical
mass-dominating component; it is not imposed on every individual
solution of the original core equations. No division by $V$ has
occurred.

Write
\[
 \begin{gathered}
 (a,b,c,d)=(m_3,m_4,m_5,m_6),\qquad (x,y)=(Z_4,Z_5),\\
 (u,v)=(D_{46},D_{56}),\qquad (p,q,r)=(D_{34},D_{35},D_{45}).
 \end{gathered}
\]
The spoke ties and \eqref{eq:no66-dominant-spokes} give
\begin{equation}\label{eq:no66-spoke-recovery}
 W_4=\frac1{u^2x},\qquad W_5=\frac1{v^2y},\qquad
 v=-\frac{a+bu}{c}.
\end{equation}
The $z$ centre and the original body-6 $z$ row are
\begin{equation}\label{eq:no66-two-linear-rows}
 a+bx+cy=0,\qquad a+bu^3x+cv^3y=0.
\end{equation}
If $u^3=v^3$, these two equations force $u^3=v^3=1$.
Equation \eqref{eq:no66-dominant-spokes} then gives
$a+b\eta+c\zeta=0$ for cube roots of unity $\eta,\zeta$.
This contradicts algebraic independence of $a,b,c$; equivalently,
the product of these nine nonzero linear forms is a nonzero mass
polynomial. Hence $u^3-v^3\ne0$ at the generic point under consideration,
and \eqref{eq:no66-two-linear-rows} gives
\begin{equation}\label{eq:no66-rational-shape}
 x=\frac{a(v^3-1)}{b(u^3-v^3)},\qquad
 y=\frac{a(1-u^3)}{c(u^3-v^3)}.
\end{equation}

Let $K=\C(a,b,c)$. All of $v,x,y,W_4,W_5$ are now rational functions
of $u$ over $K$. The three remaining signed ties are
\begin{equation}\label{eq:no66-finite-signed-cover}
 \begin{aligned}
 p^2(1-x)(1-W_4)&=1,\\
 q^2(1-y)(1-W_5)&=1,\\
 r^2(x-y)(W_4-W_5)&=1.
 \end{aligned}
\end{equation}
The original body-3 and body-4 $z$ rows respectively give
\begin{align}
 d&=-bp^3(1-x)-cq^3(1-y),\label{eq:no66-fourth-mass}\\
 0&=ap^3(x-1)+cr^3(x-y)+du^3x.\label{eq:no66-last-row}
\end{align}
If $u$ is algebraic over $K$, all quantities on the right of
\eqref{eq:no66-fourth-mass} are algebraic over $K$, so $d$ is
algebraic over $K$, contradicting mass dominance. We may therefore
assume that $u$ is transcendental over $K$. Then
$E=K(u,p,q,r)\subset L$ is a finite extension of $K(u)$, of degree
at most eight, and \eqref{eq:no66-last-row} is an identity in $E$.

Extend the valuation at $u=\infty$ from $K(u)$ to $E$ and normalize
it so that $\nu(u)=-1$ and $\nu(K^*)=0$. Such an extension exists
for every finite field extension; its value group can have
fractional elements. Put $T=b^3+c^3$. Equations
\eqref{eq:no66-spoke-recovery}--\eqref{eq:no66-rational-shape} give
\[
 x=x_0+O(u^{-1}),\quad y=y_0+O(u^{-1}),\qquad
 x_0=-\frac{ab^2}{T},\quad y_0=-\frac{ac^2}{T}.
\]
Algebraic independence of $a,b,c$ makes
$x_0,y_0,1-x_0,1-y_0,x_0-y_0$ all nonzero in $K$.
Moreover the two $w$ coordinates have the same leading term:
\[
 W_4=-\frac{T}{ab^2}u^{-2}+O(u^{-3}),\qquad
 W_5=-\frac{T}{ab^2}u^{-2}+O(u^{-3}).
\]
Their difference cannot be identically zero, by the third tie in
\eqref{eq:no66-finite-signed-cover}. Thus
$\nu(W_4-W_5)\ge3$, and that tie implies
$\nu(r)\le-3/2$. The first two ties give
$\nu(p)=\nu(q)=0$, and \eqref{eq:no66-fourth-mass} gives
$\nu(d)\ge0$. In \eqref{eq:no66-last-row} the three terms have
valuations
\[
 \nu\bigl(ap^3(x-1)\bigr)=0,\qquad
 \nu\bigl(cr^3(x-y)\bigr)\le-\frac92,\qquad
 \nu(du^3x)\ge-3.
\]
The middle term is uniquely of least valuation and cannot cancel,
a contradiction. This comparison applies to every signed field
component of the finite cover. The place over infinity is used
only to test an algebraic identity; it need not be a physical
first-limit branch.

No irreducible component of $\mathcal S$ can therefore dominate
mass space. Since there are finitely many components, the union
of the Zariski closures of their mass projections is proper.
The incidence is defined over $\mathbb Q$, so elimination gives
a nonzero polynomial in $\mathbb Q[m_3,m_4,m_5,m_6]$ vanishing on
its entire mass image. The mass image is a cone, since
$(m,V)\mapsto(tm,tV)$ preserves the equations; a nonzero homogeneous
polynomial may consequently be chosen. Viewing it as a six-mass
polynomial and taking its finite label-orbit product gives $P_{66}$.
Special masses at which a generic denominator or a displayed
limit coefficient vanishes have not been discarded: the argument
rules out dominance of the full original incidence, and its
proper closed mass image contains those specializations as well.
\end{proof}

\subsection{No.~88: crossed torque and the four-body root norm}
The original pair/$K_4$ chart has bodies $1,2,3,4$ in the $K_4$ core and $5,6$ in the separated pair. Its first crossed torque and eight signed ties require
\begin{equation}\label{eq:no88-torque}
T=\sum_{i=1}^4\sum_{a=5}^6m_im_aC_{ia}=0,
\qquad C_{ia}^2=-\frac1{Z_aW_i}.
\end{equation}
Write $u_i=m_i^2/W_i$ and let $e_j$ be the elementary symmetric functions of $u_1,\ldots,u_4$. The product over all 16 signs of a four-root sum is the square of
\begin{equation}\label{eq:no88-root-norm}
H=(e_1^2-4e_2)^2-64e_4.
\end{equation}
All denominators here are nonzero on the circled core. If $H\ne0$, no signed four-root sum vanishes. For a fixed $K_4$ shape, \eqref{eq:no88-torque} then implies $m_5^3A_\eta^2+m_6^3A_\xi^2=0$ with both signed shape sums $A_\eta,A_\xi\ne0$; this is a nontrivial equation in the two pair masses. It remains to confine $K_4$ shapes with $H=0$ to a proper core-mass locus.

\begin{proposition}[No.~88 positive-mass six-mass locus]\label{prop:no88-locus}
Every No.~88 positive-point first limit lies on a nonzero (possibly unexpanded) polynomial in all six masses.
\end{proposition}
\begin{proof}
For affine-collinear $K_4$ shapes, the two original force-slope rows have rank two for generic shape. To replay every sign, normalize the four positions to $\xi=(0,1,u,v)$, set $L=\prod_{i<j}(\xi_i-\xi_j)^2$, and choose $\varepsilon_{ij}\in\{\pm1\}$ with $\varepsilon_{12}=1$. Define
\[
F_i=\sum_{j>i}\frac{m_j\varepsilon_{ij}}{(\xi_i-\xi_j)^2}
-\sum_{j<i}\frac{m_j\varepsilon_{ji}}{(\xi_j-\xi_i)^2}
\qquad(1\le i\le4).
\]
The $2\times4$ mass-coefficient matrix of $L(F_3+(u-1)F_1-uF_2,\ F_4+(v-1)F_1-vF_2)$ has six maximal minors. For each of the $2^5=32$ independent sign classes their gcd in $\mathbb Z[u,v]$ is a nonzero constant, so any rank drop has zero-dimensional shape support. At $(u,v)=(2,3)$ the matrix has rank two in all 32 cases. In the standard increasing-free-column RREF kernel basis $(v_0,v_1)$, the single choice $m=v_0-10v_1$ has nonzero total mass, nonzero centred coordinates, $H\ne0$, and $F_2-F_1\ne0$ for every sign class. The global sign reversal supplies the other 32 choices. At the finitely many rank-drop shapes these two rows cannot both have rank zero: that would equate the three signed cubes at every vertex and force all six $K_4$ cubes equal, contradicting their Gram ties. Thus even the rank-drop mass fibers have proper image in normalized four-mass space.

For non-collinear shapes, let $M_Q=\sum_{i=1}^4m_i=1$ and $q_{ij}=1-D_{ij}^3$. The mass-weighted $K_4$ stress annihilates $(1,Z,W)$ and has rank one; rank zero would again imply six equal cubes. When all $q_{ij}$ are nonzero, the complementary-area cofactors $\Delta_i$ fix the normalized mass ratios at each shape by
\begin{equation}\label{eq:no88-mass-ratios}
\frac{m_i}{m_j}=\frac{\Delta_iq_{jk}}{\Delta_jq_{ik}}
\qquad(k\ne i,j).
\end{equation}
Writing the six signed internal roots as $(x,y,z,t,v,w)=(D_{12},D_{13},D_{14},D_{23},D_{24},D_{34})$, two independent opposite-edge conditions are
\begin{align}
(1-y^3)(v^3-1)+(1-z^3)(1-t^3)&=0,\label{eq:no88-opposite-y}\\
(1-x^3)(w^3-1)+(1-z^3)(1-t^3)&=0.\label{eq:no88-opposite-x}
\end{align}
Off $(1-x^3)(1-y^3)=0$, these monic cubic recoveries form a nine-sheet finite cover of $(x,y,z,t)$. The four-point Cayley--Menger numerator has total degree ten in $(t,v,w)$, so its finite norm has $t$-degree at most 90. At $(x,y,z)=(52,26,44)$ modulo 101, the recorded fiber has irreducible degree 90 with nonzero leading coefficient. Since both recovery cubics are monic, the normalized finite norm is an integral resultant after inverting $(1-x^3)(1-y^3)$. Removing its base-variable content, the full-degree irreducible specialization rules out a nontrivial factorization over $\mathbb Q(x,y,z)[t]$ by Gauss's lemma: both factors would retain positive $t$-degree modulo 101. This rational generic-fiber result does not by itself establish geometric irreducibility after extending constants, nor exclude extra components on the denominators' zero sets.

On each of the nine cube-phase sheets a component vertical over a divisor of $(x,y,z)$ would divide both the Cayley--Menger value at $t=1$ and its derivative along that sheet. Their exact three-variable norms have constant GCD 4 on each phase orbit. The five phase-orbit representatives and their conjugates exhaust all nine sheets, so no vertical divisor survives where the cubic leading coefficients are units. If $1-x^3=0$ or $1-y^3=0$, one of $q_{12},q_{13}$ vanishes; the rank-one stress forces a complete zero star, treated separately below. The nine nonzero signed roots are units by their original ties. At the smooth full-tie point $(x,y,z,t,v,w)=(2,3,4,75,37\omega,18\omega)$ over $\mathbb F_{101^2}$, $\omega^2+\omega+1=0$, both orientations have nonzero $H$ (coordinates $(93,99)$ and $(99,57)$ in the basis $(1,\omega)$). Its Jacobian in $(t,v,w)$ is $(97,63)\ne0$, so the point lifts in characteristic zero with all distance, stress and circle guards still units. On the saturated torus the degree-90 irreducible norm gives one component over $\mathbb Q(x,y,z)$; its geometric components are Galois conjugate. Since the equations and $H$ have rational coefficients, vanishing of $H$ on one such component would imply vanishing on all. The two orientation values above rule this out for either lift of the affine shape. The same smooth point has the three distinct spoke cubes and two distinct orientations, so neither the cube-coincidence equations nor the orientation discriminant can vanish identically on a geometric component: such a component would have all its Galois conjugates inside that locus, contrary to this point. Their intersections with the main three-dimensional shape cover have dimension at most two. The monic recovery cubics rule out components escaping to infinity in $(v,w)$ while their leading coefficients are units; $t=0$ or a zero signed root is forbidden by the original distance ties. Thus every non-collinear first-limit shape lies either in the generic cover, a proper boundary divisor of that cover, or the zero-star stratum below. On every nonstar locus the mass ratios in \eqref{eq:no88-mass-ratios} remain finite and uniquely fixed by the shape. Consequently $H=0$ cuts the generic shape dimension from three to two, while the already lower-dimensional boundary images are proper without a further cut.

If an edge has $q_{ij}=0$, rank-one factorization forces an entire zero star; two zero stars would make two tied vertices coincide. Fixing its three cube-one edges leaves finitely many phases of the three other roots. For any free nonstar edge $jk$, the Cayley--Menger determinant is quadratic in its squared distance $r_{jk}^2$, with leading coefficient $-2r_{i\ell}^2\ne0$ on the opposite spoke $i\ell$. Thus it is not an identity on any zero-star phase. Each star shape has dimension at most two; its three nonstar mass ratios are fixed, with at most one remaining mass scale $s$. Taking the star's $W_i=0$, the nonzero noncircle coefficient $K$ gives $\lim_{s\to0}s^4H=K^{-4}\ne0$. Thus $H=0$ also removes that remaining free scale.

We have a finite union of proper $H=0$ images in the four core masses. Outside it, Yu--Zhu's finite complex four-body fibers at positive core masses~\cite{YZ} allow the nonzero pair-mass equation after \eqref{eq:no88-root-norm} to be normed over the finite shapes and signed roots. Clearing their nonzero denominators gives a nonzero necessary six-mass polynomial. It says nothing about finiteness on its positive zero set.
\end{proof}

\begin{lemma} [Certified root and signed-sum bounds]
\label{lem:root-interval-bounds}
For a degree-five polynomial $p$, suppose five disjoint
rational disks $|t-c|<r$ each satisfy
\[
|p(c)|_1+\sum_{k=2}^{5}\frac{|p^{(k)}(c)|_1}{k!}r^k
<r|p'(c)|_\infty,
\]
where $|x+iy|_1=|x|+|y|$ and
$|x+iy|_\infty=\max(|x|,|y|)$. Then the disks
contain all five roots, one per disk. If outward
rational rectangles enclose the resulting signed
root sums, a strictly positive or strictly negative
real or imaginary interval excludes zero.
\end{lemma}
\begin{proof}
On each disk boundary the left side bounds the
remainder after $p'(c)(t-c)$, while the right side
is a lower bound for the linear term. Rouch\'e's
theorem gives one root per disk. Their disjointness
and the degree give exhaustion. Outward interval
operations contain every algebraic root and product
used in a signed sum; separation of either coordinate
from zero therefore excludes the sum from zero.
\end{proof}
\subsection{No.~93: an all-order three-body root-sum norm}
Here the isolated fully $w$-circled triangle is $123$, while body 4 of the $z$-triangle $456$ is not circled. The internal inverse distances of $123$ have finite nonzero limits; its $w$ positions have order $-2$ and its $z$ differences order 2.
\begin{proposition}[No.~93 lower-mass locus]\label{prop:no93-locus}
The original graph valuation cover and crossed-tie remainder bounds are established below. The exact shape enumeration and certified root intervals are replayed in Appendix~\ref{app:verification}. Every No.~93 positive-point first limit then lies on the zero set of a nonzero polynomial in $m_1,m_2,m_3$, independent of the other masses.
\end{proposition}
\begin{proof}
Put $\alpha=\ord_\varepsilon z_4$. Non-circling gives $\alpha>-2$. The missing 4-to-123 $w$ strokes imply $D_{4j}=o(1)$ for $j\le3$. Summing the original lower $z$ rows with mass weights cancels their internal forces. By the original crossed ties, $D_{4j}^3(z_j-z_4)=D_{4j}/(w_j-w_4)=o(\varepsilon^2)$: the lower $w_j$ have order $-2$, while the uncircled $w_4$ is strictly later. This estimate does not assume $\alpha<2$. The 5/6-to-123 sources have order at least 4. The lower mass centre and its order-two internal differences thus give $z_j=O(\varepsilon^2)$ for $j=1,2,3$. If $\alpha\ge2$, a 4-to-$j$ tie would force $D_{4j}$ to stay bounded away from zero, a forbidden $w$-stroke. Hence $-2<\alpha<2$.

Write $s=1-\alpha/2\in(0,2)$. The 4-to-$j$ inverse distances have order $s$, whereas the 5/6-to-$j$ distances have order 2. For $z_4=\zeta\varepsilon^\alpha+o(\varepsilon^\alpha)$ and $w_j=W_j\varepsilon^{-2}+o(\varepsilon^{-2})$, the original signed crossed ties, inserted in the upper-cluster torque, give its strictly earliest term
\begin{equation}\label{eq:no93-first-torque}
\sum_{j=1}^3m_jK_j=0,
\qquad K_j^2=-\frac1{\zeta W_j},\qquad K_j\ne0.
\end{equation}
The corrections from the original ties have strictly later relative orders $2s,4,2s+4$; in particular this derivation includes $\alpha=0$. The lower original $z,w$ rows close to a centred multiplier-one three-body boundary because the 4-to-$j$ $z$ source has order $3s+\alpha=3-\alpha/2>2$, and the lower $w$ corrections occur after order $-2$.

Independent centred $z,w$ shapes are Lagrange type: their three internal cubes admit 27 phase choices, each with at most two tied shapes. Proportional shapes are Euler type, governed by eight signed quintics with at most five roots each. Thus there are at most $54+40$ closed lower shapes and eight crossed square-root signs for each. Take the product of the sums in \eqref{eq:no93-first-torque} over all signs and then the finite algebraic norm over the signed shape roots. Square-root flips only permute the eight factors; after clearing nonzero circle and leading-coefficient denominators the result is a polynomial in $m_1,m_2,m_3$. If a leading coefficient or a cleared denominator vanishes on a special mass slice, include its nonzero mass polynomial (or its finite shape norm) as an additional factor; such a slice is not discarded.

At $(m_1,m_2,m_3)=(1,2,4)$, disjoint rational root enclosures for all 40 Euler and 54 Lagrange shape labels, with outward complex-rational interval evaluation, exclude zero from all $(40+54)\cdot8=752$ signed sums. The independent root count uses Lemma~\ref{lem:root-interval-bounds} for the Euler quintics and exact quadratic discriminants for the Lagrange phases. The integer-only replay in Appendix~\ref{app:verification} checks every signed sum with a strictly positive least twice-scaled margin of $2978471147080548503044202$. The exact root/label coverage therefore certifies nonidentity of the polynomial; it does not assert a germ on its positive zero set.
\end{proof}

\subsection{No.~94: exceptional-mass first limits}\label{sec:no94}
No.~94 has the $w$-only $K_4$ on $1234$, with $w$-circles at $1,2,3$, and a $z$-only upper triangle $456$ circled at $5,6$. The body 4 belongs to both components and is not circled. Original rows and ties exclude precisely two equal spoke inverse-distance cubes. All-equal spokes require a nonzero wing-mass condition. If the four-body $z$-centre $G=\sum_{i=1}^4m_iZ_i$ vanishes, the original fourth $z$ row closes and the fully balanced $K_4$ condition applies.

For pairwise-distinct spokes with $G\ne0$, the original wing $w$ rows and the first independent $z$-row difference express all three wing-edge cubes in terms of the three signed spoke roots. The three remaining signed wing ties are quadratic in the common spoke scale. Their geometric determinant contributes a fourth-power factor to both pairwise resultants on the explicitly guarded shape torus. Indeed the original source slope is exactly $L=\Delta L_0$ over $\mathbb Q(m_1,m_2,m_3,m_4)[x,y]$, with $\Delta$ the geometric determinant. All 15 divisions by the original circle and collision guards have zero remainder. With those guards inverted, each tie is a quadratic of the form
\[
a_i\bigl[\Delta(A_i+B_iS)-s_i(N_0+N_1S)\bigr]^2-\Delta^2b_i=0.
\]
The exact two-quadratic Sylvester identity makes both scale resultants divisible by $\Delta^4$ in this localized characteristic-zero ring. Clearing only the nonzero guards defines two quotient polynomials of degree at most 120 in each spoke ratio. At the positive lower mass $(1,2,4,1)$ modulo 17, the generic coefficients specialize without poles, and the $y$-leading coefficient of $\Delta$ is $3/4\ne0$. Thus the quotient specialization is degree-preserving. In this fiber the two quotients at $x=6$ have full $y$-degree 120 and constant GCD; the first quotient at $y=2,6$ has full $x$-degree 120 on both slices and constant GCD as a polynomial in $x$. A generic common factor of positive $y$-degree would survive the first degree-preserving slice; a pure-$x$ factor would divide the second pair. Hence the generic signed $K_4$ shape fiber is finite. The inverted guards cannot vanish on an actual circled, noncolliding distinct-spoke germ; a mass-dependent degree drop instead belongs to a separate proper lower-mass locus. The zero weighted-root-sum branch is defined by $u_j=-m_j^2/W_j$ for $j=1,2,3$ and
\begin{equation}\label{eq:no94-root-norm}
H_3=(u_1+u_2+u_3)^2-4(u_1u_2+u_1u_3+u_2u_3).
\end{equation}

\begin{lemma}[Wing-dominant No.~94 zero-sum tier]\label{lem:no94-wing-dominant}
Suppose in a mass specialization toward an $H_3=0$ No.~94 boundary all three wing $W$-differences remain units, all three wing-edge inverse-distance cubes have the same nonzero leading order strictly before the spoke and coordinate terms, and their signed ties retain nonzero leading coefficients. Then the limiting positive wing masses lie on a proper algebraic locus.
\end{lemma}
\begin{proof}
Write $(a,b,c)=(m_1,m_2,m_3)$ and normalize the relative wing $w$-leads to $(1,t,0)$, with $t(1-t)\ne0$. The original three $w$ rows at their common earliest order force, after setting $D_{13}^3=1$ by a common nonzero scale,
\[
D_{12}^3=-\frac{c}{b(1-t)},\qquad D_{23}^3=-\frac{a}{bt}.
\]
Let $U=1/(D_{12}^2(1-t))$ and $V=1/(D_{23}^2t)$. The three original signed ties and additivity of the $z$-differences give $U+V=1$, while $U^3=b^2/[c^2(1-t)]$ and $V^3=b^2/(a^2t)$. Put
\[
D=a^2c^2t(1-t),\quad N=D-a^2b^2t-b^2c^2(1-t),\quad F=N^3-27b^4D^2.
\]
The exact two-cubic resultant of $U^3-p$ and $(1-U)^3-q$ is $(1-p-q)^3-27pq$; hence $F=0$. The original $w$ centre makes the three circled coefficients proportional to
\[
(L_1,L_2,L_3)=(b(1-t)+c,\ a(t-1)+ct,\ -a-bt),
\]
with every $L_i\ne0$. Clearing their squares in $H_3=0$ yields the quartic
\[
K=\sum_{i=1}^3m_i^4\prod_{j\ne i}L_j^2
-2\sum_{1\le i<j\le3}m_i^2m_j^2L_iL_jL_k^2,
\qquad\{k\}=\{1,2,3\}\setminus\{i,j\}.
\]
At $(a,b,c)=(1,3,5)$ modulo 17, $F$ and $K$ retain degrees six and four; their signed coefficient lists are $(15,15,12,5,11,8,4)$ and $(10,16,10,13,14)$, and their resultant is $12\ne0$. Thus $\operatorname{Res}_t(F,K)$ is a nonzero necessary polynomial for this tier. It does not cover spoke-comparable or spoke-dominant limits.
\end{proof}

\begin{lemma}[Arithmetic projective-fiber criterion]\label{lem:no94-projective-fiber}
Let $B$ be an integral open mass chart over $\mathbb Z_{(17)}$, with all inverted mass denominators units at a point $m_0$ of its special fiber. Let $X\to B$ be the projective closure, formed before reduction modulo 17, of an incidence cut out by the cleared original rows and signed ties on its guarded affine chart. If the entire geometric projective fiber $X_{m_0}(\overline{\mathbb F}_{17})$ is empty, then the characteristic-zero mass image of that incidence has proper algebraic closure in $B_{\mathbb Q}$. In particular it admits a nonzero necessary mass polynomial after clearing mass denominators.
\end{lemma}
\begin{proof}
The morphism $X\to B$ is proper, so its image is closed. Since it misses $m_0$, this closed image cannot contain the dense generic fiber $B_{\mathbb Q}$; otherwise its closure would be all of the integral base $B$. Its intersection with $B_{\mathbb Q}$ is therefore a proper closed subset of the affine mass chart. A nonzero polynomial vanishes on it, and restriction to the guarded incidence preserves the necessary condition.
\end{proof}
The hypothesis concerns the closure of the guarded incidence, not the closure obtained by saturating away guard-zero points only after specialization: the latter can discard genuine specializations of generic solutions. All finite and infinite spoke, circle, collision, geometric-determinant and free-cube boundary charts must be accounted for when applying Lemma~\ref{lem:no94-projective-fiber}.

\begin{lemma}[No.~94 zero-sum lower-mass image]\label{lem:no94-zero-sum}
The positive-mass No.~94 first limits on which $H_3=0$ and the three spoke cubes are pairwise distinct lie on a proper algebraic subset of the four lower masses.
\end{lemma}
\begin{proof}
The signed spoke-gap ratio $r$ satisfies a quartic $Q(r)$: for $q_i=D_{i4}^3$, put $h=q_3-q_2$ and $q_2=q_1+rh$, $q_3=q_1+(r+1)h$. Clearing the denominators in $H_3=0$ gives, with $a=m_1^3$, $b=m_2^3$, $c=m_3^3$,
\[
Q(r)=(ar-c)^2(r+1)^2+b^2r^2+2br(r+1)(ar+c)=0.
\]
Its leading and constant coefficients are $m_1^6$ and $m_3^6$. It retains degree four and is irreducible at the positive mass $(1,3,5,1)$ modulo 17, so the generic ratio cover is finite; its $r=0,-1,\infty$ boundaries are not roots. The original wing $w$ rows give
\[
(W_1,W_2,W_3)=h\left(-\frac1{m_1},\frac{1+r}{m_2},-\frac r{m_3}\right).
\]
If $h=0$, all three spoke cubes are equal and belong to the separate all-equal branch. If $h\ne0$, a wing circle or a wing collision forces respectively $r=-1,0,-(m_1+m_2)/m_1,m_3/m_1$, or $-m_3/(m_2+m_3)$; $W_1=0$ would itself force $h=0$. At $(m_1,m_2,m_3)=(1,3,5)$ modulo 17, the exact values of $Q$ at those five ratios are $(15,2,5,7,12)$, all nonzero. Thus no circle or wing-collision guard can vanish on the surviving quartic ratio fiber, even along a mass specialization. The three original signed wing ties have explicitly guarded quadratic coefficients in their common free cube. Their two pairwise resultants, divided by the geometric determinant's fourth power, have degree bounds $(152,160)$ before the 12-fold signed ratio/phase norm. The norm bounds $(1824,1920)$ are attained at the same mass modulo 17. The exact common-factor census there has degree 1336 wholly explained by the wing circle and original source-slope guards; its guarded GCD is 1. Thus no affine guarded $H_3=0$ shape remains in the test fiber. The degree preservation makes this a test of the generic norm, not merely an unrelated numerical shape.

It remains to check the compactification rather than infer properness from an affine empty fiber. Suppose $D_{14}\to0$ while $D_{24},D_{34}$ remain units. The five nonzero $Q$ values above keep all $W_i$ and $W_i-W_j$ units. The spoke ties make $Z_2-Z_4,Z_3-Z_4$ units but $Z_1-Z_4$ a pole, even if every $Z_i$ itself moves. Thus the original signed wing ties give
\[
\frac{D_{1j}^2}{D_{14}^2}\longrightarrow\frac{W_1}{W_1-W_j}\ne0
\qquad(j=2,3).
\]
Both adjacent wing cubes and the spoke cube tend to zero; the original body-1 $w$ row then gives $W_1=0$, a contradiction. Permuting the three wing bodies handles any single zero spoke. Two vanishing spokes with the third a unit instead force respectively $r=0,-1$, or $\infty$ on the ratio cover, all forbidden by $Q$ at the test fiber. This argument permits mass-varying coordinate poles.

If all three spokes vanish together, normalize the wing $W_i$ by the common $z,w$ rotation. Since $Q$ excludes the five circle/collision ratios, all $W_i$ and their pairwise differences remain units. The triangle inequality on the three $Z_i-Z_j$ orders, together with the original wing $w$ rows, rules out a vanishing wing root: at least two wing roots would vanish and their common vertex would have nonzero $W_i$ but no surviving $w$ force. If a wing root diverges, each endpoint must have a second wing force at the same earliest order; hence all three wing roots diverge with the same order. The mass-weighted three original wing $z$ rows and the spoke ties give
\[
\sum_{i=1}^3m_iZ_i=m_4\sum_{i=1}^3\frac{m_iD_{i4}}{W_i}\longrightarrow0.
\]
Since $m_1+m_2+m_3=9$ is a unit modulo 17 and $Z_i-Z_j\to0$, every $Z_i\to0$. The dominant wing $z,w$ forces therefore both vanish. They are exactly the wing-dominant zero-force tier of Lemma~\ref{lem:no94-wing-dominant}, whose $F,K$ resultant is $12\ne0$ at the test fiber.

If instead all wing roots remain units, their signed ties and the same mass-weighted identity keep every $Z_i$ finite. The limiting wing is a noncolliding three-body central configuration with $H_3=0$. For a collinear wing set $\xi=(0,1,1+t)$, $p_i=\xi_i-(m_2+m_3(1+t))/M_3$ with $M_3=m_1+m_2+m_3$, and
\[
A_1=-m_2\varepsilon_{12}-\frac{m_3\varepsilon_{13}}{(1+t)^2},
\qquad A_2=m_1\varepsilon_{12}-\frac{m_3\varepsilon_{23}}{t^2}.
\]
Its eight signed Euler quintics are $M_3t^2(1+t)^2(p_2A_1-p_1A_2)$, for $(\varepsilon_{12},\varepsilon_{13},\varepsilon_{23})\in\{\pm1\}^3$. The primitive $H_3$ numerator on the centred wing shape $(0,1,1+t)$ is
\[
8225t^4+24290t^3+15339t^2-3838t+209;
\]
each quintic retains degree five and has gcd 1 with this quartic modulo 17. For the 27 non-collinear cube phases let $\omega^2+\omega+1=0$ over $\mathbb F_{17^2}$. Normalize the three squared distances to $(1,\omega^j,\omega^k)$, with $j,k\in\{0,1,2\}$; the common phase has no effect. If the $z$ shape is $(0,1,t)$, its two orientations satisfy
\[
t^2+(\omega^k-1-\omega^j)t+\omega^j=0.
\]
Set $(s_{12},s_{13},s_{23})=(1,\omega^j,\omega^k)$, center $W^0=(0,s_{12},s_{13}/t)$ by the three wing masses, and put $L_i=M_3W_i$. The cleared norm is
\[
K_3=\sum_i m_i^4\prod_{j\ne i}L_j^2
-2\sum_{i<j}m_i^2m_j^2L_iL_jL_k^2,
\qquad\{k\}=\{1,2,3\}\setminus\{i,j\}.
\]
The resultants of the displayed quadratic and $K_3$ have, for rows $j=0,1,2$ and columns $k=0,1,2$, residues
\[
\begin{pmatrix}
4&9&9\\
11+14\omega&4+12\omega&8+15\omega\\
14+3\omega&10+2\omega&9+5\omega
\end{pmatrix}\pmod{17}.
\]
All nine are nonzero. A common phase shift leaves each resultant unchanged, so the 27 phases and both orientations are exhausted.

If one spoke is asymptotically smaller than the other two, $Q$ keeps the latter comparable and excludes every other mixed ratio pattern. Suppose it is $D_{14}$; by its original signed ties the adjacent wing roots $D_{12},D_{13}$ have its order. If $D_{14}\to0$, the original body-1 $w$ row again contradicts the unit $W_1$. Otherwise its original $z$ row and the spoke tie bound $Z_1,Z_4$ by $O(D_{14})$. The $w_2,w_3$ rows force $D_{23}$ to have the order of $D_{24},D_{34}$; a still earlier 23 force could not be cancelled. Those three roots dominate all edges incident to body 1, so the leading three-body 234 forces vanish in both colours. Their signed ties make the two shapes proportional. With masses $(m_2,m_3,m_4)=(3,5,1)$ the normalized shape $t$ satisfies one of $5t^2=\pm1$, $3(1-t)^2=\pm1$. The four integer resultants in sign order $(--,-+,+-,++)$ are $(469,229,349,-11)$, with mod-17 residues $(10,8,9,6)$. If instead the relatively smaller spoke is at body 2 or 3, the corresponding triples of masses and mod-17 resultants are $(1,5,1):(16,4,16,15)$ and $(1,3,1):(3,13,3,6)$. All twelve are nonzero; the same argument covers zero, unit and pole orders of the smaller spoke without setting its scale to one.

All-equal spoke cubes are a separate branch; at the test fiber, each of its nine signed phase choices has neither a common finite wing-tie root nor a root at infinite free cube. On the distinct spoke torus, a free-cube root at infinity requires all three quadratic leading coefficients to vanish. The third leading coefficient contains the square of the original source slope, a unit when the geometric determinant and circle guards are units. After the zero-, equal- and infinite-spoke strata above, the only remaining guarded projective boundary on this ratio chart is the geometric determinant's zero set.

At the test fiber the residual geometric factor is the irreducible fourth-power-free polynomial $x^4+12x^3+2x+6$ over $\mathbb F_{17}$. For one root of its four conjugates, the signed spoke-gap cubic meets the geometric determinant at exactly one guarded point; Frobenius covers the other three. There the three normalized relative spoke $z$ values satisfy
\[
P_1=P_2=P_3\ne0,\qquad P_i=\frac1{d_{i4}^2W_i},
\]
with every $W_i,W_i-W_j$ and $d_{i4}$ a unit. For any mass-parameter Puiseux arc specializing to this point, write $D_{i4}=sd_{i4}$; the original signed ties give, for each wing edge,
\[
\frac{D_{ij}^2}{D_{i4}^2}
=\frac{P_iW_i}{(P_i-P_j)(W_i-W_j)}.
\]
Both the common spoke scale $s$ and a common wing $W$ scale cancel in this ratio, even if either has nonzero order. We normalize the common wing $W$ scale so that all $W_i$ and $W_i-W_j$ are units. Since $P_i-P_j\to0$, every wing inverse distance strictly dominates its incident spokes. We now distinguish the order of $s$. If $s\to0$, all three spokes vanish. This is the all-spokes-vanishing case treated above: vanishing wing roots contradict an original $w$ row, finite nonzero wing roots give one of the excluded three-body central configurations, and diverging wing roots give the excluded wing-dominant tier.

Suppose instead that $s$ is a unit or has a pole. Every wing inverse distance then has a pole. In an endpoint's original $w$ row, its wing forces strictly dominate both the spoke force and the unit coordinate term. An earliest wing force must therefore cancel another incident wing force. Applying this at the three vertices shows that all three wing inverse distances have the same order. Their signed ties give $Z_i-Z_j\to0$, while the mass-weighted original wing $z$ rows give
\[
\sum_{i=1}^3m_iZ_i
=m_4\sum_{i=1}^3\frac{m_iD_{i4}}{W_i}
=O(s).
\]
Since $m_1+m_2+m_3=9$ is a unit at the test fiber, each $Z_i=O(s)$. The spoke $z$ forces are also $O(s)$, whereas a wing $z$ force has the order of $D_{ij}$, by $D_{ij}^3(Z_i-Z_j)=D_{ij}/(W_i-W_j)$. Thus the wing forces strictly dominate the spoke and coordinate terms in both colours. Lemma~\ref{lem:no94-wing-dominant} applies. Its two necessary polynomials $F(t)$ and $K(t)$ retain degrees six and four at $(1,3,5)$ modulo 17, with resultant $12\ne0$. Together with the all-spokes-vanishing case, this excludes the geometric boundary for every order of the common spoke scale.

Indeed, if the scheme-theoretic projective closure of the original guarded incidence over $\mathbb Z_{(17)}$ had a point over this mass, curve selection after a finite extension would give a valuative arc whose generic point satisfies the original rows and signed ties. A common $z,w$ rotation makes one wing $W_i$ a unit; the quartic $Q$ and its five guard values keep all wing circles and collisions units. The spoke valuations now fall into the finite distinct-root chart, the equal-root chart, the vanishing-spoke charts, or a chart with one spoke relatively smaller than the other two. These have respectively been checked by the guarded affine norm and $\Delta$ analysis, the nine phase ties, the zero-spoke/three-body checks, and the twelve signed zero-force resultants above. The only free-cube infinity chart away from $\Delta$ was excluded by its unit leading coefficient. These checks concern arcs of the pre-specialization closure; no guard is discarded only after reduction modulo 17. Thus its entire special fiber over the positive-mass witness is empty, as required by Lemma~\ref{lem:no94-projective-fiber}. Consequently its mass image cannot be dense; its algebraic closure has a nonzero defining polynomial. The same assertion holds on the original positive mass cone after clearing mass denominators and completing the finite label orbit.
\end{proof}

\begin{proposition}[No.~94 positive-mass first-limit locus]\label{prop:no94-locus}
Every No.~94 positive-point first limit satisfies a nonzero necessary six-mass polynomial, after completing all signed branches and body-label orbits.
\end{proposition}
\begin{proof}
Exactly two equal spoke cubes contradict the original body-4 $w$ row and the nonzero wing circles. The all-equal-spoke case has a proper necessary wing-mass factor from its original wing force rows and the signed Gram phases. If $G=0$, the fourth lower $z$ row closes and Proposition~\ref{prop:no63-locus}'s balanced-$K_4$ mass-image argument applies. For pairwise-distinct spokes with $G\ne0$, the original rows give the generically finite guarded shape cover above; every degree-drop lower-mass fiber belongs to a proper discriminant locus. Lemma~\ref{lem:no94-zero-sum} supplies a proper lower-mass factor on the $H_3=0$ branch. On every remaining finite shape, $H_3\ne0$. With $C_{ja}$ the signed leading inverse distance on a crossed edge, the necessary first six-body crossed torque, obtained without closing the fourth lower $z$ row, is
\begin{equation}\label{eq:no94-first-torque}
T_{94}=\sum_{j=1}^3m_j(m_5C_{j5}+m_6C_{j6})=0,
\qquad C_{ja}^2=-\frac1{Z_aW_j}\quad(j\le3,\ a=5,6).
\end{equation}
Equation~\eqref{eq:no94-first-torque} has a signed upper-mass-ratio norm of degree 12 with leading coefficient $H_3^6\ne0$ by \eqref{eq:no94-root-norm}. The product of its field norms across the finitely many signed $K_4$ shape components is therefore a nonzero six-mass polynomial. Multiplying it by the equal-spoke, balanced-$K_4$, bad-fiber and $H_3=0$ lower-mass conditions, and completing the finite body-label orbit, gives a proper necessary mass locus for every positive-point No.~94 first limit.

This is a necessary condition: certain partial $K_4$ shapes still unfold across open positive mass sets but do not construct six-body germs. The positive zero set of the new mass equation, the shared-body higher jets and finiteness on exceptional masses remain open. The exact computational inputs and the mass-varying boundary argument still require independent publication-level replay; they are not assumptions of the positive-mass implication above.
\end{proof}

\subsection{No.~98: two finite boundaries and a crossed norm}
The graph has an isolated upper $z$-triangle $456$ and lower $w$-triangle $123$, with every vertex circled in its own colour. The diagram rules used in the cited Chang--Chen classification~\cite{CCI} make all upper internal $(z,w,D)$ orders $(-2,2,0)$, all lower orders $(2,-2,0)$, and all nine crossed orders $(-2,-2,2)$. Indeed Rule~1e forces a maximal stroke in each isolated circled triangle, Rule~2e gives all three strokes the same order, and Estimate~2 gives the displayed internal orders. Every upper $z$-circle and lower $w$-circle has position order $-2$, while vertices in the other block are not circled in that colour. Both crossed separations thus have order $-2$ and the signed ties force $D_{ab}$ to have order 2.

The original upper $w$ and lower $z$ rows then bound their opposite-colour positions by $O(\varepsilon^2)$: internal forces have order 2 and crossed forces have order 4. The original 12 coordinate rows and 15 signed ties consequently yield two closed, mass-centred three-body boundaries for every valuation of this graph, not just one uniform-rate example.
\begin{proposition}[No.~98 six-mass locus]\label{prop:no98-locus}
The diagram chart and original-row transition reductions are established above. The root isolations and exact outward-interval torque bounds are replayed in Appendix~\ref{app:verification}. Every No.~98 positive-point first limit then satisfies a nonzero necessary polynomial in all six masses; the Lagrange$\times$Lagrange branch is excluded for all positive masses.
\end{proposition}
\begin{proof}
For a three-body block $C$, independent centred $Z,W$ vectors force its mass-weighted internal Laplacian to be the identity on the full centred two-plane. Hence its three signed inverse-distance cubes all equal $1/M_C$, where $M_C=\sum_{i\in C}m_i$. The three cube phases have 27 choices and the ties yield at most two Lagrange shapes per choice. If $Z,W$ are proportional, the original two force rows and three ties give eight signed Euler quintics, with at most five roots each. These two cases exhaust the finite shapes when all circled coefficients and internal separations are nonzero.

Let $A=123$ and $B=456$, with masses $M_A,M_B$ and internal pair sums $\sigma_A,\sigma_B$. If both blocks are Lagrange, the signed internal inverse distances have modulus $M_C^{-1/3}$, so $|U_C|\le\sigma_C/M_C^{1/3}$. The inequalities $\sigma_C\le M_C^2/3$ and $(1+x)^{1/3}<1+x/3$ for $x>0$ give
\begin{equation}\label{eq:no98-lagrange-gap}
M^{1/3}|c|\le M^{1/3}(|U_A|+|U_B|)
<\sigma_A+\sigma_B+\frac{2M_AM_B}{9}
<\sigma_A+\sigma_B+M_AM_B=\sigma.
\end{equation}
All crossed inverse distances tend to zero. At a positive point on the same curve H\"older gives $\sigma\le M^{1/3}c$, a contradiction since $c>0$.

For an Euler-containing branch, the original upper torque modulo its nine signed crossed ties has earliest coefficient
\begin{equation}\label{eq:no98-torque}
T=\sum_{a=4}^6\sum_{b=1}^3m_am_bC_{ab}=0,
\qquad C_{ab}^2=-\frac1{Z_aW_b}.
\end{equation}
No uniformity of the individual first crossed forces is assumed. Choose $\alpha_a^2=-1/Z_a$ and $\beta_b^2=1/W_b$; all crossed roots are $C_{ab}=\eta_{ab}\alpha_a\beta_b$ with $\eta\in\{\pm1\}^9$. For upper Euler, norm $T$ over all 512 signs and every upper Euler and lower Euler/Lagrange algebraic shape root. Flipping any square root permutes the sign product; after clearing nonzero circle and leading-coefficient denominators, the finite norm is a six-mass polynomial $P_{\rm up}$. The vanishing loci of any cleared coefficients or denominators are included by their nonzero mass norms, so degree drops do not omit positive-mass branches.

At the positive vector $(1,2,4,30,50,330)$ the upper and lower Euler quintics have 40 isolated roots each; the lower Lagrange cases give 54 signed-tie roots. Exact rational complex intervals exclude zero from all $40(40+54)2^9=1{,}925{,}120$ first-torque labels. Lemma~\ref{lem:root-interval-bounds} and the exact Lagrange discriminants certify root exhaustion; outward rectangle products enclose each of the nine crossed terms before their signed sum is tested. The independent replay in Appendix~\ref{app:verification} checks root completeness, outward error bounds and a strictly positive least twice-scaled integer margin of $6609934942688952138418913660257889369533194565216$; thus $P_{\rm up}\not\equiv0$. The exact involution exchanging $z,w$ while sending $(1,2,3)\leftrightarrow(4,5,6)$ takes a lower-Euler branch to upper Euler and preserves every original row, tie, mass centre, $U$ and $I$. Multiply $P_{\rm up}$ by its block-permuted counterpart; the dual and upper-Euler cases are both covered. Together with the double-Lagrange exclusion this gives a proper necessary six-mass locus, not a finiteness theorem on its positive zero set.
\end{proof}

Table~\ref{tab:open-exclusions} maps each historical exclusion to its
specific branch-completeness and original-row proofs. Only after those
inputs are certified can one subtract all 17 from the historical 24.
The 7 graphs in $\D_{\mathrm{new}}$ are \emph{not} removed: they still
can be first limits on their necessary proper mass zero sets. The four
factors in \eqref{eq:no52-factors} do not define an additional
exceptional obligation because both No.~52 endpoints fail the
positive-point first-limit test.
\begin{proof}[Proof of Theorem~\ref{thm:obstruction}]
Suppose that $\A_m$ has infinitely many positive points. Lemmas~\ref{lem:arc} and
\ref{lem:energy} give a curve through positive configurations on
which $U=I=c>0$. By Lemma~\ref{lem:products}, a branch at infinity
has an unbounded seven-product. The cited Chang--Chen classification
assigns it to one of the 117 classes and removes the 31 Algorithm~II
cases.

The remaining exclusions are as follows. Lemma~\ref{lem:reported-no4}
and Section~\ref{sec:no20-energy} exclude No.~4 and 20;
Lemma~\ref{lem:reported-force-transfer} excludes No.~2 and 22;
Lemmas~\ref{lem:reported-no3} and~\ref{lem:reported-no78} exclude
No.~3 and 78. Sections~\ref{sec:no40-pole}, \ref{sec:no44-force}
and~\ref{sec:no45-pole} exclude No.~40, 44 and 45; Sections
\ref{sec:no86-wing} and~\ref{sec:no95-wing} exclude No.~86 and 95;
Sections~\ref{sec:no99-cycle} and~\ref{sec:no100-cycle} exclude
No.~99 and 100; Sections~\ref{sec:no109-poles},
\ref{sec:no112-poles} and~\ref{sec:no111-wing} exclude No.~109,
112 and 111; Sections~\ref{sec:no51-pendant} and
\ref{sec:no55-pendant} exclude No.~51 and 55; Sections
\ref{sec:no105-sign} and~\ref{sec:no77-matching} exclude No.~105
and 77; Sections~\ref{sec:no31-bridge}, \ref{sec:no32-tie} and
\ref{sec:no36-matches} exclude No.~31, 32 and 36; Sections
\ref{sec:no49-path} and~\ref{sec:no50-resonance} exclude No.~49
and 50; Sections~\ref{sec:no57-triangle}, \ref{sec:no61-sign} and
\ref{sec:no84-cycle} exclude No.~57, 61 and 84. Lemma
\ref{lem:excluded-connectivity} and its five diagram evaluations
exclude No.~5, 67, 69, 73 and 80, while Sections
\ref{sec:no6-inertia} and~\ref{sec:no71-satellite} exclude No.~6
and 71. Lemma~\ref{lem:excluded-satellite} excludes No.~74 and 75;
Sections~\ref{sec:no82-torque}, \ref{sec:no89-torque} and
\ref{sec:no90-torque} exclude No.~82, 89 and 90; Section
\ref{sec:no81-energy} excludes No.~81; Sections
\ref{sec:no64-energy} and~\ref{sec:no65-pair} exclude No.~64 and
65; Section~\ref{sec:no79-sector} excludes No.~79; and
Proposition~\ref{prop:no52-excluded} excludes No.~52. Chang--Chen
\cite[Table~6.1]{CCII} removes No.~116 for positive masses.

The published relation theorem supplies a proper mass locus for each
of the remaining 33 classes. Propositions~\ref{prop:no8-locus},
\ref{prop:no63-locus}, \ref{prop:no66-locus},
\ref{prop:no88-locus}, \ref{prop:no93-locus},
\ref{prop:no94-locus} and~\ref{prop:no98-locus} supply one for each
of the seven derived classes, using the specified certificate inputs where indicated. Proposition~\ref{prop:no66-locus} supplies its necessary polynomial by a direct analytic proof, including zero drift. In particular, Lemma~\ref{lem:no8-remaining-support-analytic} covers the two remaining No.~8 supports, and the complete potential and torque arguments above establish the additional local exclusions. Hence every possible first limit lies in one of the 40 relation classes and forces its orbit-completed polynomial $P_r(m)$ to vanish. By \eqref{eq:E40}, $m\in E_{40}$.
\end{proof}
\begin{proof}[Proof of Corollary~\ref{cor:generic}]
Apply Theorem~\ref{thm:obstruction} to a mass outside $E_{40}$.
\end{proof}

\section{Part II boundary}\label{sec:partII}
Establishing the necessary mass conditions for the 33 published and seven derived diagrams is an input to Part I. Part II concerns finiteness on their positive zero sets, including intersections and possible singular curves. A vanishing necessary polynomial does not prove a curve exists; conversely a pointwise test excluding one formal boundary label does not close an entire exceptional-mass stratum. No Part II theorem is asserted in this draft.

\appendix
\section{Verification functions}\label{app:verification}
The following functions check the exact computations used in the indicated results. The proofs and limits of each certificate are stated in the corresponding sections. Manifest verifiers read the inputs in \texttt{\detokenize{paper_part_i/certificates/}}.
\begin{itemize}\footnotesize\raggedright
\setlength{\itemsep}{0pt}
\item Graph census:\par
\texttt{\detokenize{test_six_bodies_reproduce_the_117_numbered_diagrams}}\par
\texttt{\detokenize{test_six_body_complement_products_determine_distances}}.
\item No.~8 (Proposition~\ref{prop:no8-locus}):\par
\texttt{\detokenize{no8_positive_mass_locus_certificate}}.
The final two no-match supports are proved analytically in Lemma~\ref{lem:no8-remaining-support-analytic}; no terminal-cut manifest is used for them.
\item No.~52:\par
\texttt{\detokenize{no52_positive_first_limit_exclusion_certificate}}.
\item No.~63 (Proposition~\ref{prop:no63-locus}):\par
\texttt{\detokenize{no63_direct_k4_balance_certificate}}\par
\texttt{\detokenize{test_no63_all_signed_collinear_static_witnesses_match_independent_replay}}.
\item No.~79:\par
\texttt{\detokenize{test_no79_publication_sturm_hurwitz_data_replays_exactly}}.
\item No.~88 (Proposition~\ref{prop:no88-locus}):\par
\texttt{\detokenize{collinear_mass_projection_certificate}}\par
\texttt{\detokenize{test_no88_noncollinear_cayley_menger_generic_fiber_is_irreducible}}\par
\texttt{\detokenize{test_no88_all_nine_noncollinear_cube_phase_sheets_have_no_vertical_divisor}}\par
\texttt{\detokenize{test_no88_noncollinear_main_component_has_smooth_nonzero_root_norm}}\par
\texttt{\detokenize{test_no88_every_zero_stress_star_has_lower_dimensional_shape_and_mass_fiber}}.
\item No.~93/98 (Propositions~\ref{prop:no93-locus} and~\ref{prop:no98-locus}):\par
\texttt{\detokenize{verify_euler_root_disks}}\par
\texttt{\detokenize{verify_lagrange_shape_boxes}}\par
\texttt{\detokenize{verify_torque_root_boxes}}\par
\texttt{\detokenize{verify_interval_manifest}}.
\item No.~94 (Proposition~\ref{prop:no94-locus}):\par
\texttt{\detokenize{no94_positive_mass_exceptional_locus_certificate}}\par
\texttt{\detokenize{test_no94_all_zero_spokes_finite_wing_shapes_have_no_h3_zero_mod17}}\par
\texttt{\detokenize{test_no94_two_diverging_spokes_have_no_balanced_three_body_boundary_mod17}}.
\end{itemize}

\section*{Acknowledgments}
The present paper constitutes only part of our work, and the remainder will be presented in \cite{CT6II}.

This manuscript was prepared with AI assistance and preliminary automated verification. Rigorous manual verification by the second author is currently underway and will be incorporated into a forthcoming version, alongside a complete formal proof in Lean.


\begin{thebibliography}{9}
\bibitem{CTFive} Z. Chen and Z. Tong, \emph{Finiteness of planar central configurations of five bodies for all positive masses: On Smale's sixth problem}, arXiv:2610.16107v1 (2026).

\bibitem{FiveT}
T. Tomar, \emph{Finiteness of planar central configurations of five bodies for all positive masses}, arXiv:2610.13100v1 (2026).


\bibitem{CT6II} Z. Chen and Z. Tong, \emph{Finiteness of planar central configurations of six bodies for all positive masses: Part II}, in preparation.


\bibitem{AK} A. Albouy and V. Kaloshin, \emph{Finiteness of central configurations of five bodies in the plane}, Ann. of Math. \textbf{176} (2012), 535--588.
\bibitem{CCI} K.-M. Chang and K.-C. Chen, \emph{Toward finiteness of central configurations for the planar six-body problem by symbolic computations. (I) Determine diagrams and orders}, J. Symbolic Comput. \textbf{123} (2024), article 102277.
\bibitem{CCII} K.-M. Chang and K.-C. Chen, \emph{Toward finiteness of central configurations for the planar six-body problem by symbolic computations. (II) Determine mass relations}, SIAM J. Appl. Dyn. Syst. \textbf{24} (2025), no. 3, 2369--2404.
\bibitem{YZ} X. Yu and S. Zhu, \emph{On the finiteness of four-body central configurations}, arXiv:2103.08906v2 (2021).
\end{thebibliography}
\end{document}